\documentclass[11pt,reqno]{amsart}
\usepackage{mathrsfs}
\usepackage{amsmath,amscd}
\usepackage{amssymb, amsmath}
\usepackage{graphicx}
\usepackage{color}
\usepackage{refcheck}

\usepackage{tikz}
\usepackage{amsaddr}
\allowdisplaybreaks
\newcommand{\newcom}{\newcommand}

\newcom{\al}{\alpha}
\newcom{\Del}{\Delta}
\newcom{\be}{\beta}
\newcom{\eps}{\epsilon}
\newcom{\e}{\varepsilon}
\newcom{\ga}{\gamma}
\newcom{\Ga}{\Gamma}
\newcom{\ka}{\kappa}
\newcom{\Lam}{\Lambda}
\newcom{\lam}{\lambda}
\newcom{\Om}{\Omega}
\newcom{\om}{\omega}
\newcom{\Si}{\Sigma}
\newcom{\si}{\sigma}
\newcom{\tht}{\theta}
\newcom{\dtri}{\nabla}
\newcom{\tri}{\triangle}
\newcom{\oo}{\infty}
\newcom{\vphi}{\varphi}
\newcom{\cA}{{\mathcal A}}
\newcom{\cB}{{\mathcal B}}
\newcom{\cC}{{\mathcal C}}
\newcom{\cD}{{\mathcal D}}
\newcom{\cE}{{\mathcal E}}
\newcom{\ce}{{\mathcal e}}
\newcom{\cF}{{\mathcal F}}

\newcom{\cJ}{{\mathcal J}}
\newcom{\cK}{{\mathcal K}}
\newcom{\cL}{{\mathcal L}}
\newcom{\cM}{{\mathcal M}}
\newcom{\cP}{{\mathcal P}}
\newcom{\cR}{{\mathcal R}}
\newcom{\cS}{{\mathcal S}}
\newcom{\cQ}{{\mathcal Q}}
\newcom{\cT}{{\mathcal T}}
\newcom{\cU}{{\mathcal U}}
\newcom{\cY}{{\mathcal Y}}
\newcom{\cZ}{{\mathcal Z}}
\newcom{\R}{\mathbb R}
\newcom{\T}{\mathbb T}
\newcom{\N}{\mathbb N}
\newcom{\Z}{\mathbb Z}
\newcom{\C}{\mathbb C}
\newcom{\E}{\mathbb E}

\newcom{\f}{\frac}
\newcom{\di}{\displaystyle\int}
\newcom{\ds}{\displaystyle\sum}
\newcom{\dl}{\displaystyle\lim}
\newcom{\ov}{\overline}
\newcom{\sset}{\subset}
\newcom{\wt}{\widetilde}
\newcom{\wh}{\widehat}
\newcom{\pa}{\partial}
\newcom{\p}{\partial}
\newcom\na{\nabla}
\newcom\lan{\langle}
\newcom\ran{\rangle}
\newcom{\suml}{\sum\limits}
\newcom{\supl}{\sup\limits}
\newcom{\intl}{\int\limits}
\newcom{\infl}{\inf\limits}
\newcom{\disp}{\displaystyle}
\newcom{\non}{\nonumber}
\newcom{\no}{\noindent}
\newcom{\QED}{$\square$}

\def\div{\mathop{\rm div}\nolimits}

\def\ef{\hphantom{MM}\hfill\llap{$\square$}\goodbreak}

\newtheorem{athm}{\bf \t}[section]
\newenvironment{thm} [1] {\def\t{#1}\begin{athm} \bf \rm} {\end{athm}}
\newcom{\bthm}{\begin{thm}}\newcom{\ethm}{\end{thm}}

\newtheorem{theorem}{Theorem}[section]

\newtheorem{remark}{Remark}[section]

\newtheorem{proposition}{Proposition}[section]

\newcom{\beq}{\begin{equation}}
\newcom{\eeq}{\end{equation}}
\newcom{\ben}{\begin{eqnarray}}
\newcom{\een}{\end{eqnarray}}
\newcom{\beno}{\begin{eqnarray*}}
\newcom{\eeno}{\end{eqnarray*}}
\newcom{\bal}{\begin{aligned}}
\newcom{\eal}{\end{aligned}}

\numberwithin{equation}{section}

\begin{document}

\title[ 2D MHD equations on the half space]
{Asymptotic behavior of solution of the non-resistive 2D MHD equations on the half space }

\author{Jiakun Jin, Yoshiyuki Kagei, Xiaoxia Ren*, Lei Wang and Cuili Zhai}

\footnotetext{*Corresponding author}

\maketitle

\begin{abstract}

In this paper, we  obtain the global well-posedness and the asymptotic behavior of solution of non-resistive 2D MHD problem on the half space. We overcome the difficulty of zero spectrum gap by building the relationship between half space and the whole space, and  get the  resolvent estimate for the weak diffusion system.
We use  the two-tier energy method that couples the boundedness of high-order $(H^3)$ energy to the decay of low-order energy, the latter of which is necessary to control  the growth of the highest energy.

\end{abstract}

{\small

{\textbf{AMS (2010) Subject Classifications:}} 35A01; 35Q35; 76W05 }

 \section{Introduction}
\setcounter{equation}{0}

In this paper, we investigate two dimensional non-resistive magneto-hydrodynamical equations
\ben\label{eq:MHD0}
\left\{
\begin{array}{l}
\p_t u-\Delta u+u\cdot\na u+\na p=B\cdot \na B,  \ \ \ \ x\in\Omega,\ t>0, \\
\p_tB+u\cdot\na B=B\cdot \na u,  \ \ \ \ x\in\Omega,\ t>0,\\
\div u=\div B=0,  \ \ \ \ x\in\Omega,\ t>0,\\
u(x,0)=u_0(x),\ \ B(x,0)=B_0(x),\ \ \ \ x\in \Omega,
\end{array}\right.
\een
here
 $u=(u_1, u_2)$ denotes the velocity field,  $B=(B_1,B_2)$ is the magnetic field and $p$ is the pressure.
 The system \eqref{eq:MHD0}  can be applied to model plasmas when the plasmas are strongly collisional, or the resistivity due to these collisions are extremely small. We refer to \cite{Cab} for some detailed discussions on the relevant physical background of this system. The question of whether smooth solution of  \eqref{eq:MHD0} develops singularity in finite time has been a long-standing open problem \cite{S}.

Recently, Lin, Xu and Zhang \cite{LXZ}  proved the  global well-posedness for the Cauchy problem of the system \eqref{eq:MHD0}  around the equilibrium state  $(0, {\bf e}_1)$ for a class of  admissible perturbations (see \cite{Z} for an elementary proof).  Ren, Wu, Xiang and Zhang \cite{RWXZ} established the global existence and  time decay rate of smooth solutions for general perturbations, which confirms the numerical observation of \cite{CC} that the energy of the MHD equations is dissipated at a rate independent of the ohmic resistivity (see also \cite{W}).  There are many related studies in this field, such as the compressible case \cite{WW}, the 3D case \cite{DZ} and the periodic case \cite{PZZ}.

For the question with physical boundaries, the problem will be more difficult. In \cite{RXZ}, global solutions are obtained under  non-slip boundary condition in horizontal infinite  domain $\R\times(0,1)$ (See \cite{DR} for the case of slip boundary condition); Tan and Wang \cite{TW} obtain the decay rate in the same domain around a non-parallel magnetic field.

In this paper, we consider the global well-posedness and long time behavior for \eqref{eq:MHD0} in the half space
$$
 \Om:=\big\{x=(x_1,x_2)\ |\ x_1 \in \R,\: x_2\in \R^+ \big\},
$$
with the most common boundary condition, i.e. the velocity field satisfies the classical non-slip boundary condition
$$
u=0\ \ \ \ \mathrm{on}\ \p\Om,
$$
which gives rise to the phenomenon of strong boundary layers in general as
formally derived by Prandtl, and the container  is perfectly conducting for the magnetic field
$$
B\cdot n=0 \ \ \ \mathrm{on}\ \p\Om,
$$
here $n$  denotes the outward normal vector of $\p\Om$.

The half space problem is harder than whole space problem because of  the appearance of physical boundary, and is harder than strip domain since Poincar\'e's inequalities cannot be used due to the infinity  of the region.

Motivated by  \cite{LXZ},  we will investigate   small perturbation of the system \eqref{eq:MHD0} around the  equilibrium state $(0, {\bf e}_1)$. Thus,  we can set  $b=B-{\bf e}_1$ and reformulate our first problem as the following initial-boundary problem
\ben\label{eq:MHDT}
\left\{
\begin{array}{l}
\p_t u-\Delta u-\p_{x_1} b+\na p=-u\cdot\na u+b\cdot \na b,  \ \ \ \ x\in\Omega,\ t>0,\\
\p_tb-\p_{x_1} u=-u\cdot\na b+b\cdot \na u,  \ \ \ \ x\in\Omega,\ t>0,\\
\div u=\div b=0,  \ \ \ \ x\in\Omega,\ t>0,\\
u=0, \ \ \ b_2=0,   \ \ \ \ x\in\p\Omega,\ t>0,\\
u(x,0)=u_0(x),\ \ b(x,0)=b_0(x),\ \ \ \ x\in \Om.
\end{array}\right.
\een
\begin{theorem}\label{thm:main}
Assume that the initial data $(u_0, b_0)$ satisfies $(u_0, b_0)\in W^{2, 1}(\Om)\cap H^3(\Om)$,  $u_0\in H_0^1(\Om)$,   $b_{1, 0}=\p_2 b_{1, 0}=0$ on $\p\Om$, $\mathcal{P}(\Delta u_0 -u_0 \cdot \na u_0 +b_0 \cdot \na b_0) \in H_0^1(\Om)$ (Here $\mathcal{P}$ is the Helmholtz projection), and
\begin{align*}
\|(u_0, b_0)\|_{W^{2, 1}}+\|(u_0, b_0)\|_{H^3}\lesssim \e
\end{align*}
with  $\e$ is a small positive constants. Then  the MHD system (\ref{eq:MHDT}) has a unique global solution $(u, b)$ satisfying
\beno
(u, b)\in C([0,+\infty); H^3(\Om)).
\eeno
Moreover, it holds that
\begin{align}\label{eq:decay}
&\|(u, b_2)\|_{L^2}\lesssim \lan t \ran^{-\f12},\quad  \|b_1\|_{L^2}\lesssim \lan t\ran^{-\f14}, \quad \|(u, b_2)\|_{L^\infty} \lesssim \lan t\ran^{-1},\quad \|b_1\|_{L^\infty} \lesssim \lan t\ran^{-\f12},\non\\
&\|\na u\|_{H^1} \lesssim \lan t\ran^{-1}, \quad\|\p_1 u\|_{H^1}\lesssim \lan t\ran^{-\f34}, \quad  \|\p_1 b\|_{L^2}\lesssim \lan t\ran^{-\f34},   \quad \|\p_1 u\|_{L^\infty}\lesssim \lan t\ran^{-1-\delta}
\end{align}
for any $t\in [0,+\infty)$.
\end{theorem}

\begin{remark}
Here we need a zero boundary value on $b_1$ in order to avoid  the trouble boundary terms in {\it a-priori} estimate. This boundary condition could not be imposed. Fortunately, the boundary condition $b_1=0$ on $\p\Omega$  can be propagated by the equation \eqref{eq:MHDT}$_2$ if $b_{1, 0}=0$  on $\p\Omega$.

What is more, we also need a zero boundary value on $\p_2 b_1$ in order to obtain the $H^3$ regularity in {\it a-priori} estimate. The boundary condition $\p_2 b_1=0$ on $\p\Omega$  can also be propagated by the equation \eqref{eq:MHDT}$_2$ if $\p_2 b_{1, 0}=0$  on $\p\Omega$.
\end{remark}

\begin{remark}

We note that the set of initial data satisfying the compatibility condition
is not empty.

Indeed, let $b_0, f \in C^\infty_0 (\Om)$ with $\div b_0=0, \div f=0$, if $\|b_0\|_{H^3}+\|f\|_{H^1}\leq \delta _0$, ($\delta _0>0$ is a constant), then there exists a $(u_0, p_0)$ with $u_0\in H^3(\Om)\cap H^1_0(\Om)$, $\na p_0\in H^1(\Om)$ $(p_0\in L^2_{loc}(\Om)$) satisfies
\[
\left\{
\begin{array}{l}
u_0-\Delta u_0-\p_{x_1} b_0+\na p_0+u_0\cdot\na u_0-b_0\cdot \na b_0=f,  \ \ \ \ x\in\Omega,\ t>0,\\
\div u_0=0,  \ \ \ \ x\in\Omega,\ t>0,\\
u_0=0,   \ \ \ \ x\in\p\Omega,\ t>0,
\end{array}\right.
\]
then $(u_0, b_0)$ satisfies the compatibility condition since
\begin{align*}
-\Delta u_0+\na p_0-\p_1 b_0+u_0 \cdot\na u_0-b_0\cdot\na b_0= f-u_0, \quad\cP(f-u_0)=f-u_0\in H^1_0(\Om),
\end{align*}
where $u(0)=u_0, \p_t u(0)=f-u_0$.
\end{remark}

\begin{remark}
Due to our low requirement for the regularity of the initial values, the decay rate only achieves what is needed to obtain the global  solution.

We point out that the half space (with physical boundary) does not reduce decay rate compared to whole space. What's more, the decay rate of $b_1$ is slower than $b_2$, since formally speaking
$b_2=\na^{-1}\na b_2\sim \na ^{-1}\p_1 b$ by divergence free condition.

\end{remark}
To prove the Theorem \ref{thm:main}, we will use the spectral analysis method motivated by \cite{KK},  the two-tier energy estimate motivated by \cite{GT} and build the relationship between half space and the whole space \cite{W}. Before going into  details, let us briefly outline two main ideas of our proof.

\smallskip

We  first use the spectrum analysis method to obtain the resolvent estimate of linearized system. For the solution of the linearized system, we take  the suitable branch cut and integration contour to build the relationship between half space and the whole space. The most trouble problem is no spectral gap for the weak dissipation   of \eqref{eq:MHDT}.

 Indeed, the solution of linearized system can be rewritten as $u_{W}+u_{H}$, where $u_W$ is the whole space part through odd extension and residue theorem,  $u_{H}$ is the half space part should select appropriate integration contour and transformation.

For the sake of completeness, we have imitated some estimates of \cite{W} for the whole  space. It should be pointed out  that we have lower requirements for regularity.

In order to obtain the global well-posedness of  weak dissipation system  with physical boundary, we utilized some techniques such as the  time derivative $\p_t$  without  changing the boundary condition; the Helmholtz projection $\cP$ to eliminate  the pressure; the estimates for the Stokes system to build the relationship between $\p_t$ and regularity; the anisotropic estimate, cross term estimate and transformation of magnetic field equation to control  nonlinear term $b\cdot\na b$ by weak magnetic dissipation.

What is more, since the magnetic field satisfies a transport equation and its $H^3$ norm is bounded as
\beno
\|b(t)\|_{H^3}\le \|b_0\|_{H^3}\exp\Big(\int_0^t\|\p_1 u(\tau)\|_{L^\infty}d\tau\Big)+\cdots,
\eeno
 there needs that $\|\p_1 u\|_{L^\infty}$ has $\lan t\ran^{-1-\delta}$ decay at least.

With the aid of linear decay rate, we use more carefully frequency localization method and the structure of equation, to obtain the required nonlinear decay rate under low regularity energy.

\smallskip

The paper is organized as follows. In section 2, we present the solution formula  of the linearized problem and obtain the resolvent estimates for $u, b$ directly. In section 3,  We obtain the decay estimates of the nonlinear estimate based on the linearized analysis and boundedness of the high order energy. The local well-posedness is given in section 4. In section 5, we prove the boundedness of the high order energy and dissipation with the aid of the decay rate of low order energy. That is to say, we obtain the global well-posedness by the  bootstrap argument.

\smallskip

\no{\bf Notations.} Throughout this paper, for simplicity, we set  $\p_i=\f{\p}{\p x^i}$ for $i=1,2,3$,  $\p_t=\f{\p}{\p t}$. We will also use $A\lesssim B$ to denote the statement that $A \le CB$ for some absolute constant $C > 0$,  which may be different on different lines, and  $A \sim B$ to denote the statement that $A \lesssim B$ and $B \lesssim A$. We set $\lan t\ran:=(1+t^2)^\f12$. The symbol $\wh{*}$ means the horizontal Fourier transform. Let $\vphi(\xi)$ be a smooth bump function adapted to $\{|\xi|\leq 2\}$ and equal to  1 on $\{|\xi|\leq 1\}$. For $j>0$, we define the Fourier multipliers
\begin{align*}
P_{\leq j} f&:=\cF^{-1}\Big(\varphi(\f{|\xi|}{j})\,\cF f(\xi)\Big), \quad P_{\geq j} f:=\cF^{-1}\Big((1-\varphi(\f{|\xi|}{j}))\,\cF f(\xi)\Big), \non\\
P_{j} f&:=\cF^{-1}\Big((\varphi(\f{2|\xi|}{j})-\varphi(\f{|\xi|}{j})\,\cF f(\xi)\Big),
\end{align*}
where j are dyadic number, that is the form of $2^{\Z}$ in general.

\section{The resolvent estimate of linearized problem}
\subsection{Solution formula  of the linearized problem}
\setcounter{equation}{0}
 The linearized equation of \eqref{eq:MHDT} is
 \ben\label{eq:MHDL}
\left\{
\begin{array}{l}
\p_t u-\Delta u-\p_1 b+\na p =f,  \ \ \ \ x\in\Omega,\ t>0,\\
\p_tb-\p_1 u=g,  \ \ \ \ x\in\Omega,\ t>0,\\
\div u=\div b=0,  \ \ \ \ x\in\Omega,\ t>0,\\
u=0, \ \ \ b=0,   \ \ \ \ x\in\p\Omega,\ t>0,\\
u(x,0)=u_0(x),\ \ b(x,0)=b_0(x),\ \ \ \ x\in \Om,
\end{array}\right.
\een
where $f=-u\cdot\na u+b\cdot\na b$ and $g=-u \cdot\na b+b\cdot\na u$.

 Taking Laplace transform in $t$, we have
\begin{align}\label{eq:Lap trans}
\left\{
\begin{array}{l}
\lam u_{\lam}-\Delta u_{\lam}-\p_1 b_{\lam}+\na p_{\lam}= u_0+f_{\lam}, \\
\lam b_{\lam}-\p_1 u_{\lam}=b_0+g_{\lam}, \\
\div u_{\lam}=\div b_{\lam}=0,  \\
(u_{\lam}, b_{\lam})|_{x_2=0}=0,
\end{array}\right.
\end{align}
whereas we denote that
\beno
f_\lam=\cL\big(f(t)\big)=\int^{\oo}_0 e^{-\lam t}  f(t) dt.
\eeno
By taking horizontal Fourier transform $\wh{*}$  of \eqref{eq:Lap trans}, we obtain
 \ben\label{eq:MHDh}
\left\{
\begin{array}{l}
\lam \wh{u}_{1, \lam}-(\p^2_2-|\xi_1|^2)  \wh{u}_{1, \lam}-i\xi_1 \wh{b}_{1, \lam}+i\xi_1 \wh{p}_{\lam} = \wh{u}_{1, 0}+\wh{f}_{1, \lam}, \\
\lam \wh{u}_{2, \lam}-(\p^2_2-|\xi_1|^2)  \wh{u}_{2, \lam}-i\xi_1 \wh{b}_{2, \lam}+\p_2 \wh{p}_{\lam}= \wh{u}_{2, 0}+\wh{f}_{2, \lam}, \\
\lam \wh{b}_{\lam}-i\xi_1 \wh{u}_\lam=\wh{b}_0+\wh{g}_{\lam},\\
i\xi_1 \wh{u}_{1, \lam}+\p_2 \wh{u}_{2, \lam}=i\xi_1\wh{b}_{1, \lam}+\p_2 \wh{b}_{2, \lam}=0,\\
\wh{u}_{\lam}|_{x_2=0}=\wh{b}_{\lam}|_{x_2=0}=0.
\end{array}\right.
\een
 By $\eqref{eq:MHDh}_{1, 2, 4}$, we have
\beno
(\p_2^2-|\xi_1|^2) \wh{p}_{\lam}=i\xi_1 \wh{f}_{1, \lam}+\p_2 \wh{f}_{2, \lam},
\eeno
 the solution is
\begin{align}\label{pressure}
\wh{p}_{\lam}=C(\xi_1) e^{-|\xi_1|x_2}-E_{|\xi_1|}[i\xi_1  \wh{f}_{1, \lam}+\p_2\wh{f}_{2, \lam}],
\end{align}
where
 \beno
E_{|\xi_1|}[f]:=\f{1}{2|\xi_1|} \int^{\infty}_0 e^{-|\xi_1||x_2-y_2|} f(y_2) dy_2.\eeno
Note that
\begin{align}\label{pressure2}
&\p_2 E_{|\xi_1|}[i\xi_1  \wh{f}_{1, \lam}+\p_2\wh{f}_{2, \lam}]\non\\
&=-\f12\Big(  e^{-|\xi_1|x_2}\int^{x_2}_0 e^{|\xi_1| y_2} (i\xi_1  \wh{f}_{1, \lam}+\p_2\wh{f}_{2, \lam}) dy_2-e^{|\xi_1|x_2}\int^\infty_{x_2} e^{-|\xi_1| y_2}( i\xi_1  \wh{f}_{1, \lam}+\p_2\wh{f}_{2, \lam} ) dy_2\Big).
\end{align}
Taking \eqref{pressure}, \eqref{pressure2} into \eqref{eq:MHDh} and defining
\begin{align*}
\om(\lam, \xi_1)= \sqrt{\lam+|\xi_1|^2+\f{|\xi_1|^2}{\lam}},
\end{align*}
we see that $\wh{u}_\lam$ satisfies the following system
 \ben\label{eq:MHDh2}
\left\{
\begin{array}{l}
(\om^2-\p_2^2) \wh{u}_{1, \lam}= \wh{u}_{1, 0}+\f{i\xi_1}{\lam}  \wh{b}_{1, 0}+\wh{f}_{1, \lam}+\f{i\xi_1}{\lam} \wh{g}_{1, \lam}-i \xi_1 C(\xi_1)e^{-|\xi_1|x_2} +i\xi_1E_{|\xi_1|}[i\xi_1  \wh{f}_{1, \lam}+\p_2\wh{f}_{2, \lam}],\\
(\om^2-\p_2^2) \wh{u}_{2, \lam}= \wh{u}_{2, 0}+\f{i\xi_1}{\lam}  \wh{b}_{2, 0}+\wh{f}_{2, \lam}+\f{i\xi_1}{\lam} \wh{g}_{2, \lam}+|\xi_1|  C(\xi_1)e^{-|\xi_1|x_2}
+\p_2 E_{|\xi_1|}[i\xi_1  \wh{f}_{1, \lam}+\p_2\wh{f}_{2, \lam}],\\
\wh{u}_{\lam}|_{x_2=0}=\p_2\wh{u}_{2, \lam}|_{x_2=0}=0,
\end{array}\right.
\een
and  the  unique solution is
$$
\left\{
\begin{array}{l}
 \wh{u}_{1, \lam}= A(\xi_1)e^{-\om x_2}+E_{\om}[ \wh{u}_{1, 0}+\f{i\xi_1}{\lam}  \wh{b}_{1, 0}+ \wh{f}_{1, \lam}+\f{i\xi_1}{\lam} \wh{g}_{1, \lam}]-i\xi_1 C(\xi_1)E_{\om} [ e^{-|\xi_1|x_2}]\non\\
 \qquad\quad+i\xi_1 E_{\om} [E_{|\xi_1|}[i\xi_1 \wh{f}_{1, \lam}+\p_2\wh{f}_{2, \lam}]],\\
\wh{u}_{2, \lam}= B(\xi_1)e^{-\om x_2}+E_{\om}[ \wh{u}_{2, 0}+\f{i\xi_1}{\lam}  \wh{b}_{2, 0}+ \wh{f}_{2, \lam}+\f{i\xi_1}{\lam} \wh{g}_{2, \lam}]+|\xi_1 |C(\xi_1)E_{\om} [ e^{-|\xi_1|x_2}]\non\\
\qquad\quad+E_{\om} [\p_2 E_{|\xi_1|}[i\xi_1  \wh{f}_{1, \lam}+\p_2\wh{f}_{2, \lam}]],
\end{array}\right.
$$
 where $A(\xi_1), B(\xi_1)$ and $C(\xi_1)$ depend only on $\xi_1$, Re $\om>0$.
Using the boundary condition $(\ref{eq:MHDh2})_3$, and noticing that
\begin{align*}
&\p_2 E_{\om}[f]\Big|_{x_2=0}=\f{1}{2} \int^{\infty}_0 e^{-\om y_2} f( y_2) dy_2=\om E_{\om}[f]_0,
\end{align*}
(where we define $E_\om[f]_0:=E_\om[f]\Big|_{x_2=0}$), we have
$$\left\{
\begin{array}{l}
 A(\xi_1)=-E_{\om}[ \wh{u}_{1, 0}+\f{i\xi_1}{\lam}  \wh{b}_{1, 0}+ \wh{f}_{1, \lam}+\f{i\xi_1}{\lam} \wh{g}_{1, \lam}]_0-i\xi_1 E_{\om} [E_{|\xi_1|} [i\xi_1 \wh{f}_{1, \lam}+\p_2 \wh{f}_{2, \lam}]]_0\\
\qquad\qquad -\f{i\xi_1}{|\xi_1|} \{E_{\om} [\wh{u}_{2, 0}+\f{i\xi_1}{\lam} \wh{b}_{2, 0}+\wh{f}_{2, \lam}+\f{i\xi_1}{\lam} \wh{g}_{2, \lam}]_0+E_{\om}[\p_2 E_{|\xi_1|}[i\xi_1\wh{f}_{1, \lam}+\p_2 \wh{f}_{2, \lam}]]_0\},\\
B(\xi_1)=0,\\
C(\xi_1)=-\f{1}{|\xi_1|E_{\om}[e^{-|\xi_1|x_2}]_0}\{E_{\om} [\wh{u}_{2, 0}+\f{i\xi_1}{\lam} \wh{b}_{2, 0}+\wh{f}_{2, \lam}+\f{i\xi_1}{\lam} \wh{g}_{2, \lam}]_0+E_{\om}[\p_2 E_{|\xi_1|}[ i\xi_1\wh{f}_{1, \lam}+\p_2 \wh{f}_{2, \lam}]]_0\},
\end{array}\right.
$$
 So we obtain the solution of system \eqref{eq:MHDh2}
\begin{align}\label{ulam}
\left\{
\begin{array}{l}
 \wh{u}_{1, \lam}= -\Big(E_{\om}[ \wh{u}_{1, 0}+\f{i\xi_1}{\lam}  \wh{b}_{1, 0}+ \wh{f}_{1, \lam}+\f{i\xi_1}{\lam} \wh{g}_{1, \lam}]_0+i\xi_1 E_{\om} [E_{|\xi_1|} [i\xi_1 \wh{f}_{1, \lam}+\p_2 \wh{f}_{2, \lam}]]_0\\
\qquad\qquad +\f{i\xi_1}{|\xi_1|} \{E_{\om} [\wh{u}_{2, 0}+\f{i\xi_1}{\lam} \wh{b}_{2, 0}+\wh{f}_{2, \lam} +\f{i\xi_1}{\lam} \wh{g}_{2, \lam}]_0
+E_{\om}[\p_{2} E_{|\xi_1|}[i\xi_1\wh{f}_{1, \lam}+\p_2 \wh{f}_{2, \lam}]]_0\}\Big) e^{-\om x_2}\\
\qquad\quad+E_{\om}[ \wh{u}_{1, 0}+\f{i\xi_1}{\lam}  \wh{b}_{1, 0}+ \wh{f}_{1, \lam}+\f{i\xi_1}{\lam} \wh{g}_{1, \lam}]+i\xi_1 E_{\om} [E_{|\xi_1|}[i\xi_1 \wh{f}_{1, \lam}+\p_2 \wh{f}_{2, \lam}]]\\
\qquad\quad+ \f{i\xi_1E_{\om}[e^{-|\xi_1|x_2}]}{|\xi_1|E_{\om}[e^{-|\xi_1|x_2}]_0}\{E_{\om} [\wh{u}_{2, 0}+\f{i\xi_1}{\lam} \wh{b}_{2, 0}+\wh{f}_{2, \lam}+\f{i\xi_1}{\lam} \wh{g}_{2, \lam}]_0+E_{\om}[\p_{2}E_{|\xi_1|}[ i\xi_1\wh{f}_{1, \lam}+\p_2 \wh{f}_{2, \lam}]]_0\}\\
\wh{u}_{2, \lam}= E_{\om}[ \wh{u}_{2, 0}+\f{i\xi_1}{\lam}  \wh{b}_{2, 0}+ \wh{f}_{2, \lam}+\f{i\xi_1}{\lam} \wh{g}_{2, \lam}]+E_{\om} [\p_2 E_{|\xi_1|}[i\xi_1 \wh{f}_{1, \lam}+\p_2 \wh{f}_{2, \lam}]]\\
\qquad\quad-\f{E_{\om} [ e^{-|\xi_1|x_2}]}{E_{\om}[e^{-|\xi_1|x_2}]_0}\{E_{\om} [\wh{u}_{2, 0}+\f{i\xi_1}{\lam} \wh{b}_{2, 0}+\wh{f}_{2, \lam}+\f{i\xi_1}{\lam} \wh{g}_{2, \lam}]_0+E_{\om}[\p_{2} E_{|\xi_1|}[ i\xi_1\wh{f}_{1, \lam}+\p_2 \wh{f}_{2, \lam}]]_0\}
\end{array}\right.
\end{align}
By $\eqref{eq:MHDh}_3$,
\begin{align*}
\wh{b}_\lam=\f{i\xi_1}{\lam}  \wh{u}_{\lam}+\f{1}{\lam} \big(\wh{b}_0+\wh{g}_\lam\big),
\end{align*}
 so we obtain
\begin{align}\label{blam}
\left\{
\begin{array}{l}
 \wh{b}_{1, \lam}= \Big( -\f{i\xi_1} {\lam} E_{\om} [ \wh{u}_{1, 0}+\f{i\xi_1}{\lam}  \wh{b}_{1, 0}+\wh{f}_{1, \lam}+\f{i\xi_1}{\lam} \wh{g}_{1, \lam} ]_0+\f{|\xi_1|^2}{\lam}  E_{\om} [E_{|\xi_1|}[i\xi_1  \wh{f}_{1, \lam}+\p_2\wh{f}_{2, \lam}]]_0\\
\qquad\qquad +\f{|\xi_1|}{\lam} \{ E_{\om} [\wh{u}_{2, 0}+\f{i\xi_1}{\lam} \wh{b}_{2, 0}+\wh{f}_{2, \lam}+\f{i\xi_1}{\lam} \wh{g}_{2, \lam}]_0+E_{\om}[\p_2 E_{|\xi_1|}[ i\xi_1\wh{f}_{1, \lam}+\p_2 \wh{f}_{2, \lam}]]_0\} \Big) e^{-\om x_2}\\
 \qquad \qquad + \f{i\xi_1} {\lam} E_{\om} [ \wh{u}_{1, 0}+\f{i\xi_1}{\lam}  \wh{b}_{1, 0}+\wh{f}_{1, \lam}+\f{i\xi_1}{\lam} \wh{g}_{1, \lam} ]-\f{|\xi_1|^2}{\lam}  E_{\om} [E_{|\xi_1|}[i\xi_1  \wh{f}_{1, \lam}+\p_2\wh{f}_{2, \lam}]] \\
 \qquad\qquad-\f{|\xi_1|}{\lam} \f{E_{\om}[e^{-|\xi_1|x_2}]}{ E_{\om}[ e^{-|\xi_1|x_2}]_0} \{ E_{\om} [\wh{u}_{2, 0}+\f{i\xi_1}{\lam} \wh{b}_{2, 0}+\wh{f}_{2, \lam}+\f{i\xi_1}{\lam} \wh{g}_{2, \lam}]_0+E_{\om}[\p_2 E_{|\xi_1|}[ i\xi_1\wh{f}_{1, \lam}+\p_2 \wh{f}_{2, \lam}]]_0\}\\
  \qquad \qquad +\f{1}{\lam} (\wh{b}_{1, 0}+\wh{g}_{1,\lam}) \\
 \wh{b}_{2, \lam}= \f{i\xi_1} {\lam} E_{\om} [ \wh{u}_{2, 0}+\f{i\xi_1}{\lam}  \wh{b}_{2, 0}+\wh{f}_{2, \lam}+\f{i\xi_1}{\lam} \wh{g}_{2, \lam} ]+\f{i\xi_1}{\lam}E_{\om} [  \p_2 E_{|\xi_1|}[i\xi_1  \wh{f}_{1, \lam}+\p_2\wh{f}_{2, \lam}]]\\
\qquad \quad  -\f{i\xi_1 }{\lam} \f{E_{\om}[e^{-|\xi_1|x_2}]}{E_{\om}[e^{-|\xi_1|x_2}]_0} \{ E_{\om} [\wh{u}_{2, 0}+\f{i\xi_1}{\lam} \wh{b}_{2, 0}+\wh{f}_{2, \lam}+\f{i\xi_1}{\lam} \wh{g}_{2, \lam}]_0+E_{\om}[\p_2 E_{|\xi_1|}[ i\xi_1\wh{f}_{1, \lam}+\p_2 \wh{f}_{2, \lam}]]_0\}\\
  \qquad \qquad +\f{1}{\lam} (\wh{b}_{2, 0}+\wh{g}_{2,\lam}) \\
 \end{array}\right.
\end{align}

Taking the  inverse Laplace transform in $t$:
\beno
\cL^{-1}(F(\lam))=\f{1}{2\pi i} \int^{\beta+i\infty}_{\beta-i\infty} e^{\lam t} F(\lam) d\lam
\eeno
of \eqref{ulam} and  \eqref{blam} ($F(\lam)$ has no singularity  on the right hand side of $\beta=Re\,\lam$) , we  get the horizontal Fourier transform of the solution of \eqref{eq:MHDL}
\begin{align}\label{eq:u sol2}
\left\{
\begin{array}{l}
\wh{u}_1(\xi_1,x_2, t)=\f{1}{2\pi i}\int_{\Ga}e^{\lam t}
\Big\{ -\Big(E_{\om}[ \wh{u}_{1, 0}+\f{i\xi_1}{\lam}  \wh{b}_{1, 0}+ \wh{f}_{1, \lam}+\f{i\xi_1}{\lam} \wh{g}_{1, \lam}]_0+i\xi_1 E_{\om} [E_{|\xi_1|} [i\xi_1 \wh{f}_{1, \lam}+\p_2 \wh{f}_{2, \lam}]]_0\\
\qquad\qquad +\f{i\xi_1}{|\xi_1|} \{E_{\om} [\wh{u}_{2, 0}+\f{i\xi_1}{\lam} \wh{b}_{2, 0}+\wh{f}_{2, \lam} +\f{i\xi_1}{\lam} \wh{g}_{2, \lam}]_0
+E_{\om}[\p_{2} E_{|\xi_1|}[i\xi_1\wh{f}_{1, \lam}+\p_2 \wh{f}_{2, \lam}]]_0\}\Big) e^{-\om x_2}\\
\qquad\qquad+E_{\om}[ \wh{u}_{1, 0}+\f{i\xi_1}{\lam}  \wh{b}_{1, 0}+ \wh{f}_{1, \lam}+\f{i\xi_1}{\lam} \wh{g}_{1, \lam}]+i\xi_1 E_{\om} [E_{|\xi_1|}[i\xi_1 \wh{f}_{1, \lam}+\p_2 \wh{f}_{2, \lam}]]\\
\qquad\qquad+ \f{i\xi_1E_{\om}[e^{-|\xi_1|x_2}]}{|\xi_1|E_{\om}[e^{-|\xi_1|x_2}]_0}\{E_{\om} [\wh{u}_{2, 0}+\f{i\xi_1}{\lam} \wh{b}_{2, 0}+\wh{f}_{2, \lam}+\f{i\xi_1}{\lam} \wh{g}_{2, \lam}]_0+E_{\om}[\p_{2}E_{|\xi_1|}[ i\xi_1\wh{f}_{1, \lam}+\p_2 \wh{f}_{2, \lam}]]_0\}\Big\}d\lam,\\
\wh{u}_2(\xi_1,x_2, t)=\f{1}{2\pi i}\int_{\Ga}e^{\lam t}
\Big\{  E_{\om}[ \wh{u}_{2, 0}+\f{i\xi_1}{\lam}  \wh{b}_{2, 0}+ \wh{f}_{2, \lam}+\f{i\xi_1}{\lam} \wh{g}_{2, \lam}]+E_{\om} [\p_2 E_{|\xi_1|}[i\xi_1 \wh{f}_{1, \lam}+\p_2 \wh{f}_{2, \lam}]]\\
\qquad\qquad-\f{E_{\om} [ e^{-|\xi_1|x_2}]}{E_{\om}[e^{-|\xi_1|x_2}]_0}\{E_{\om} [\wh{u}_{2, 0}+\f{i\xi_1}{\lam} \wh{b}_{2, 0}+\wh{f}_{2, \lam}+\f{i\xi_1}{\lam} \wh{g}_{2, \lam}]_0+E_{\om}[\p_{2}E_{|\xi_1|}[ i\xi_1\wh{f}_{1, \lam}+\p_2 \wh{f}_{2, \lam}]]_0\}\Big\}d\lam,\\
\wh{b}_1(\xi_1, x_2,t)=\f{1}{2\pi i}\int_{\Ga}e^{\lam t}
\Big\{\Big( -\f{i\xi_1} {\lam} E_{\om} [ \wh{u}_{1, 0}+\f{i\xi_1}{\lam}  \wh{b}_{1, 0}+\wh{f}_{1, \lam}+\f{i\xi_1}{\lam} \wh{g}_{1, \lam} ]_0+\f{|\xi_1|^2}{\lam}  E_{\om} [E_{|\xi_1|}[i\xi_1  \wh{f}_{1, \lam}+\p_2\wh{f}_{2, \lam}]]_0\\
\qquad\qquad +\f{|\xi_1|}{\lam} \{ E_{\om} [\wh{u}_{2, 0}+\f{i\xi_1}{\lam} \wh{b}_{2, 0}+\wh{f}_{2, \lam}+\f{i\xi_1}{\lam} \wh{g}_{2, \lam}]_0+E_{\om}[\p_2 E_{|\xi_1|}[ i\xi_1\wh{f}_{1, \lam}+\p_2 \wh{f}_{2, \lam}]]_0\} \Big) e^{-\om x_2}\\
 \qquad \qquad + \f{i\xi_1} {\lam} E_{\om} [ \wh{u}_{1, 0}+\f{i\xi_1}{\lam}  \wh{b}_{1, 0}+\wh{f}_{1, \lam}+\f{i\xi_1}{\lam} \wh{g}_{1, \lam} ]-\f{|\xi_1|^2}{\lam}  E_{\om} [E_{|\xi_1|}[i\xi_1  \wh{f}_{1, \lam}+\p_2\wh{f}_{2, \lam}]] \\
 \qquad\qquad-\f{|\xi_1|}{\lam} \f{E_{\om}[e^{-|\xi_1|x_2}]}{ E_{\om}[ e^{-|\xi_1|x_2}]_0} \{ E_{\om} [\wh{u}_{2, 0}+\f{i\xi_1}{\lam} \wh{b}_{2, 0}+\wh{f}_{2, \lam}+\f{i\xi_1}{\lam} \wh{g}_{2, \lam}]_0+E_{\om}[\p_2 E_{|\xi_1|}[ i\xi_1\wh{f}_{1, \lam}+\p_2 \wh{f}_{2, \lam}]]_0\} \Big\}d\lam\\
  \qquad\qquad+\wh{b}_{1, 0}+\wh{g}_{1},\\
 \wh{b}_2(\xi_1, x_2,t)=\f{1}{2\pi i}\int_{\Ga}e^{\lam t}
\Big\{ \f{i\xi_1} {\lam} E_{\om} [ \wh{u}_{2, 0}+\f{i\xi_1}{\lam}  \wh{b}_{2, 0}+\wh{f}_{2, \lam}+\f{i\xi_1}{\lam} \wh{g}_{2, \lam} ]+\f{i\xi_1}{\lam}E_{\om} [  \p_2 E_{|\xi_1|}[i\xi_1  \wh{f}_{1, \lam}+\p_2\wh{f}_{2, \lam}]]\\
\qquad \qquad  -\f{i\xi_1 }{\lam} \f{E_{\om}[e^{-|\xi_1|x_2}]}{E_{\om}[e^{-|\xi_1|x_2}]_0} \{ E_{\om} [\wh{u}_{2, 0}+\f{i\xi_1}{\lam} \wh{b}_{2, 0}+\wh{f}_{2, \lam}+\f{i\xi_1}{\lam} \wh{g}_{2, \lam}]_0+E_{\om}[\p_2 E_{|\xi_1|}[ i\xi_1\wh{f}_{1, \lam}+\p_2 \wh{f}_{2, \lam}]]_0\}\Big\}d\lam\\
  \qquad\qquad+\wh{b}_{2, 0}+\wh{g}_{2}.
\end{array}\right.
\end{align}
Here $\Ga=\{\lam=R+\eta e^{2\pi i/3}, \eta\geq 0\}\cup \{\lam=R+\eta e^{-2\pi i/3}, \eta\geq 0\}$,  $R>0$ is a sufficiently large number taken in such a way that Re $  \om (\lam;\xi_1)>0$ for all $\lam\in \Ga$.

\subsection{The contour integration of the  linear part}
The linear part of \eqref{eq:u sol2}  is
\begin{align}\label{linear}
\left\{
\begin{array}{l}
\wh{u}_{1, L}(\xi_1,x_2, t)=\f{1}{2\pi i}\int_{\Ga}e^{\lam t}
\Big\{ -\Big(E_{\om}[ \wh{u}_{1, 0}+\f{i\xi_1}{\lam}  \wh{b}_{1, 0}]_0 +\f{i\xi_1}{|\xi_1|} E_{\om} [\wh{u}_{2, 0}+\f{i\xi_1}{\lam} \wh{b}_{2, 0}]_0\Big) e^{-\om x_2}\\
\qquad\qquad\qquad\qquad\qquad+E_{\om}[ \wh{u}_{1, 0}+\f{i\xi_1}{\lam}  \wh{b}_{1, 0}]
+ \f{i\xi_1E_{\om}[e^{-|\xi_1|x_2}]}{|\xi_1|E_{\om}[e^{-|\xi_1|x_2}]_0}E_{\om} [\wh{u}_{2, 0}+\f{i\xi_1}{\lam} \wh{b}_{2, 0}]_0 \Big\}d\lam,\\
\wh{u}_{2, L}(\xi_1,x_2, t)=\f{1}{2\pi i}\int_{\Ga}e^{\lam t}
\Big\{  E_{\om}[ \wh{u}_{2, 0}+\f{i\xi_1}{\lam}  \wh{b}_{2, 0}]-\f{E_{\om} [ e^{-|\xi_1|x_2}]}{E_{\om}[e^{-|\xi_1|x_2}]_0}E_{\om} [\wh{u}_{2, 0}+\f{i\xi_1}{\lam} \wh{b}_{2, 0}]_0\Big\}d\lam,\\
\wh{b}_{1, L}(\xi_1, x_2,t)=\f{1}{2\pi i}\int_{\Ga}e^{\lam t}
\Big\{\Big( -\f{i\xi_1} {\lam} E_{\om} [ \wh{u}_{1, 0}+\f{i\xi_1}{\lam}  \wh{b}_{1, 0}]_0+\f{|\xi_1|}{\lam}  E_{\om} [\wh{u}_{2, 0}+\f{i\xi_1}{\lam} \wh{b}_{2, 0}]_0 \Big) e^{-\om x_2}\\
 \qquad \qquad \qquad\qquad\qquad+ \f{i\xi_1} {\lam} E_{\om} [ \wh{u}_{1, 0}+\f{i\xi_1}{\lam}  \wh{b}_{1, 0}] -\f{|\xi_1|}{\lam} \f{E_{\om}[e^{-|\xi_1|x_2}]}{ E_{\om}[ e^{-|\xi_1|x_2}]_0} E_{\om} [\wh{u}_{2, 0}+\f{i\xi_1}{\lam} \wh{b}_{2, 0}]_0 \Big\}d\lam+\wh{b}_{1, 0},\\
 \wh{b}_{2, L}(\xi_1, x_2,t)=\f{1}{2\pi i}\int_{\Ga}e^{\lam t}
\Big\{ \f{i\xi_1} {\lam} E_{\om} [ \wh{u}_{2, 0}+\f{i\xi_1}{\lam}  \wh{b}_{2, 0}]-\f{i\xi_1 }{\lam} \f{E_{\om}[e^{-|\xi_1|x_2}]}{E_{\om}[e^{-|\xi_1|x_2}]_0}E_{\om} [\wh{u}_{2, 0}+\f{i\xi_1}{\lam} \wh{b}_{2, 0}]_0 \Big\}d\lam+\wh{b}_{2, 0}.
\end{array}\right.
\end{align}
Noticing that
\begin{align*}
E_{\om}[e^{-|\xi_1|x_2}]&=\f{1}{2\om} \Big(\int^{x_2}_0 e^{-\om(x_2-y_2)} e^{-|\xi_1| y_2} dy_2+\int_{x_2}^\infty e^{\om(x_2-y_2)} e^{-|\xi_1| y_2} dy_2\Big)\\
&=\f{1}{2\om} \Big(e^{-\om x_2} \f{e^{(\om-|\xi_1|) x_2}-1}{\om-|\xi_1|} +e^{\om x_2} \f{e^{-(\om+|\xi_1|) x_2}}{\om+|\xi_1|} \Big)\\
&=\f{1}{2\om} \Big(\f{ e^{-|\xi_1|x_2}- e^{-\om x_2}}{\om-|\xi_1|} +\f{e^{-|\xi_1|x_2}}{\om+|\xi_1|}\Big),
\end{align*}
then
\begin{align*}
\f{E_{\om}[e^{-|\xi_1|x_2}]}{E_{\om}[e^{-|\xi_1|x_2}]_0}&=\f{\om+|\xi_1|}{\om-|\xi_1|} \Big( e^{-|\xi_1|x_2}-e^{-\om x_2}\Big)+e^{-|\xi_1|x_2}\\
&=\f{2 \om\lam (\om+|\xi_1|) }{\lam^2+|\xi_1|^2}  e^{-|\xi_1|x_2}-\f{\lam (\om+|\xi_1|)^2}{\lam^2+|\xi_1|^2} e^{-\om x_2},
\end{align*}
and
\begin{align*}
- e^{-\om x_2} +\f{E_{\om}[e^{-|\xi_1|x_2}]}{E_{\om}[e^{-|\xi_1|x_2}]_0}=\f{2\om\lam (\om+|\xi_1|)}{\lam^2+|\xi_1|^2 }\Big( e^{-|\xi_1|x_2}-e^{-\om x_2}\Big).
\end{align*}
Then $\wh{u}_1$ can be rewritten as
\begin{align}\label{u1}
\wh{u}_{1, L}(\xi_1,x_2, t)&=\f{1}{2\pi i}\int_{\Ga}e^{\lam t}
\Big\{-E_{\om}[ \wh{u}_{1, 0}+\f{i\xi_1}{\lam}  \wh{b}_{1, 0}]_0 e^{-\om x_2} +E_{\om}[ \wh{u}_{1, 0}+\f{i\xi_1}{\lam}  \wh{b}_{1, 0}]\non\\
&\qquad \qquad+\f{i\xi_1}{|\xi_1|} E_{\om} [\wh{u}_{2, 0}+\f{i\xi_1}{\lam} \wh{b}_{2, 0}]_0 \f{2\om\lam (\om+|\xi_1|)}{\lam^2+|\xi_1|^2 }\Big( e^{-|\xi_1|x_2}-e^{-\om x_2}\Big) \Big\}d\lam\non\\
&=\f{1}{2\pi i}\int_{\Ga}e^{\lam t} \f{1}{2\om} \int^\infty_0  \Big(e^{-\om|x_2-y_2|} -e^{-\om(x_2+y_2)} \Big)(\wh{u}_{1, 0}+\f{i\xi_1}{\lam} \wh{b}_{1, 0})(y_2) dy_2 d\lam\non\\
&\quad+\f{i\xi_1}{|\xi_1|} \f{1}{2\pi i}\int_{\Ga}e^{\lam t} \int^\infty_0  e^{-\om y_2} (\wh{u}_{2, 0}+\f{i\xi_1}{\lam} \wh{b}_{2, 0})(y_2) dy_2 \f{\lam (\om+|\xi_1|)}{\lam^2+|\xi_1|^2 }\Big( e^{-|\xi_1|x_2}-e^{-\om x_2}\Big)  d\lam\non\\
&:=I_1+I_2,
\end{align}
where we notice
$$
Res_{\{\lam=\pm i|\xi_1|\}} \f{i\xi_1}{|\xi_1|} e^{\lam t} \int^\infty_0  e^{-\om y_2} (\wh{u}_{2, 0}+\f{i\xi_1}{\lam} \wh{b}_{2, 0})(y_2) dy_2 \f{\lam (\om+|\xi_1|)}{\lam^2+|\xi_1|^2 }\Big( e^{-|\xi_1|x_2}-e^{-\om x_2}\Big)  =0,
$$
by $
\om (\pm i |\xi_1|)=\pm |\xi_1|.
$

Similarly,
\begin{align}\label{u2}
\wh{u}_{2, L}(\xi_1,x_2, t)
&=\f{1}{2\pi i}\int_{\Ga}e^{\lam t}
\Big\{  -E_{\om}[ \wh{u}_{2, 0}+\f{i\xi_1}{\lam}  \wh{b}_{2, 0}]_0 e^{-\om x_2}+E_{\om}[ \wh{u}_{2, 0}+\f{i\xi_1}{\lam}  \wh{b}_{2, 0}]\non\\
&\quad+E_{\om}[ \wh{u}_{2, 0}+\f{i\xi_1}{\lam}  \wh{b}_{2, 0}]_0 \Big(e^{-\om x_2}
-\f{E_{\om} [ e^{-|\xi_1|x_2}]}{E_{\om}[e^{-|\xi_1|x_2}]_0}\Big)\Big\}d\lam,
\end{align}
which can be estimated as $\wh{u}_{1, L}$.

We define $\lambda'_\pm$ by
\begin{align*}
\lam'_\pm=\left\{
\begin{array}{l}
-\f12 |\xi_1|^2\pm\f{i}{2} \sqrt{4|\xi_1|^2-|\xi_1|^4}, \quad |\xi_1|\leq 2\\
-\f12 |\xi_1|^2\pm \f12\sqrt{|\xi_1|^4-4|\xi_1|^2} ,\quad |\xi_1|> 2.
\end{array}\right.\end{align*}
When $|\xi_1|\leq 2$, we use the branch specified by the requirement
\beno
\arg(\lam-\lam_\pm')=\mp\f{\pi}{2} \ \text{at} \ \lam=\text{Re}\  \lam_\pm' \ \text{and} \ \arg \lam=0 \ \text{at}\  \lam=0,
\eeno
and take the branch cut
\beno
\{ \lam; \ \text{Re} \ \lam \leq 0,\  \text{Im} \ \lam =0\}\cup \{\lam \in \Pi; \ \text{Re} \ \lam \leq \text{Re}\ \lam'_\pm\},
\eeno
where $\Pi$ is the circle defined by
\beno
\Pi=\{ \lam=\eta+i\sigma; \ \eta^2+\sigma^2= |\xi_1|^2\}.
\eeno
When $|\xi_1|> 2$, we use the branch specified by
\beno
\arg(\lam-\lam'_\pm)=\arg  \lam=0 \ \text{at} \ \lam=0,
\eeno
and take the branch cut
\beno
\{ \lam;\  \text{Re} \ \lam \leq \lam'_-,\  \text{Im} \ \lam =0\} \cup \{  \lam'_+ \leq \text{Re} \ \lam \leq 0,\  \text{Im} \ \lam =0\}.
\eeno

Now we consider the contour in the following two cases:

Case 1: $ |\xi_1|\leq 2$.

We deform the contour $\Ga$ into  $\cup_{m=1}^6 \Ga_m$, where $\Ga_1=\Ga_1^{(+)}\cup\Ga_1^{(-)}$  wraps around the portion $\{ \lam=-\eta; \ \eta:0\to |\xi_1|\}$ in the branch cut with
\begin{align*}
{\Ga^{(+)}_1}&=\Big\{ \lam=-\eta; \ \eta:0\to |\xi_1|\Big\},\\
{\Ga^{(-)}_1}&=\Big\{ \lam=-\eta; \ \eta:|\xi_1| \to 0\Big\}.
\end{align*}

$\Ga_2=\Ga_2^{(+)}\cup\Ga_2^{(-)}$  wraps around the portion $\{ \lam\in \Pi; -|\xi_1| \leq \text{Re} \lam \leq \text{Re} \lam'_+,\  \text{Im} \lam > 0\}$ in the branch cut with
\begin{align*}
{\Ga^{(+)}_2}=\Big\{ \lam=\lam'_+-\eta+i\big(-\text{Im} \lam'_++D(\eta, \xi_1)\big); \ \eta:0\to d_0\Big\},\\
{\Ga^{(-)}_2}=\Big\{ \lam=\lam'_+-\eta+i\big(-\text{Im} \lam'_++D(\eta, \xi_1)\big); \ \eta:d_0\to 0\Big\},
\end{align*}
where we define
 $$D(\eta, \xi_1):=\sqrt{(\text{Im} \lam'_+)^2+2\text{Re} \lam'_+ \eta-\eta^2}, \quad d_0= \text{Re} \lam'_++|\xi_1|.$$

$\Ga_3=\Ga_3^{(+)}\cup\Ga_3^{(-)}$ wraps around the portion $\{ \lam\in \Pi; -|\xi_1| \leq \text{Re} \lam \leq \text{Re} \lam'_+, \ \text{Im} \lam < 0\}$ in the branch cut with
\begin{align*}
{\Ga^{(+)}_3}&=\Big\{ \lam=\lam'_--\eta-i\big(-\text{Im} \lam'_++D(\eta, \xi_1)\big); \ \eta: 0\to d_0\Big\},\\
{\Ga^{(-)}_3}&=\Big\{ \lam=\lam'_--\eta-i\big(-\text{Im} \lam'_++D(\eta, \xi_1)\big); \ \eta:d_0\to 0\Big\}.
\end{align*}

$\Ga_4=\Ga_4^{(+)}\cup\Ga_4^{(-)}$  defined as
 \begin{align*}
{\Ga^{(+)}_4}&=\Big\{ \lam=-\eta ; \ \eta:|\xi_1|\to \infty\Big\},\\
{\Ga^{(-)}_4}&=\Big\{ \lam=-\eta ; \ \eta: \infty\to|\xi_1|\Big\},
\end{align*}
and
\beno
{\Ga_5}=\Big\{\lam=\e e^{i\gamma}; \ \gamma: -\pi \to \pi\Big\},
\eeno
\beno
 {\Ga_6}=\Big\{\lam=-|\xi_1|+\e e^{i\gamma}; \ \gamma: -\pi \to -\f{\pi}{2},  -\f{\pi}{2} \to 0, 0 \to \f{\pi}{2}, \f{\pi}{2}\to \pi  \Big\}
\eeno
with $ \e\to 0$ (But not equal 0), and
\beno
 {\Ga_7}=\Big\{\lam=R e^{i\gamma}; \ \gamma: \theta_0 \leq |\gamma|< \pi \Big\}
\eeno
with $ R \to \infty$.

Case 2: $|\xi_1|> 2$.

We deform the contour $\Ga$   into $\widetilde{\Ga}^{(+)}_1\cup\widetilde{\Ga}^{(-)}_1$, where $\widetilde{\Ga}_1=\widetilde{\Ga}_1^{(+)}\cup \widetilde{\Ga}_1^{(-)}$ wraps around the portion $\{ \lam=-\eta; \ \eta:0\to -\lam'_+\}$ in the branch cut with
\begin{align*}
{\widetilde{\Ga}_1^{(+)}}=\Big\{ \lam=-\eta; \ \eta:0\to -\lam'_+\Big\},\\
{\widetilde{\Ga}_1^{(-)}}=\Big\{ \lam=-\eta; \ \eta:-\lam'_+ \to 0\Big\},
 \end{align*}
and $\widetilde{\Ga}_2=\widetilde{\Ga}_2^{(+)}\cup \widetilde{\Ga}_2^{(-)}$ wraps around the portion $\{ \lam=-\eta; \ \eta:-\lam'_-\to \infty \}$ in the branch cut with
\begin{align*}
{\widetilde{\Ga}_2^{(+)}}=\Big\{ \lam=-\eta; \ \eta:-\lam'_-\to \infty\Big\},\\
{\widetilde{\Ga}_2^{(-)}}=\Big\{ \lam=-\eta; \ \eta:\infty \to -\lam'_-\Big\}.
 \end{align*}

 \subsection{The resolvent estimate of the linear part}

 We first give the  resolvent estimate of the velocity field in  \eqref{linear}.
\begin{proposition}\label{linear u}
We have the following $L^2$, $L^\infty$ estimates
 \begin{align*}
 \|u_L\|_{L^2}&\lesssim \lan t\ran^{-\f12}\|(u_0, b_0)\|_{L^1\cap L^2  },\\
  \|u_L\|_{L^\infty}&\lesssim \lan t\ran^{-1}\Big( \|(u_0, b_0)\|_{L^1\cap L^2}+\|\na^{\f12+\delta} (u_0, b_0)  \|_{L^1_{x_1} L^2_{x_2}}\Big)
 \end{align*}
 for the velocity  field in \eqref{linear} (here $\delta>0$ small enough).
 \end{proposition}

 \begin{proof}  We only need to consider $I_1$ and $I_2$ in  \eqref{u1}.

Consider the  odd extension of $\wh{u}_{1, 0}+\f{i\xi_1}{\lam} \wh{b}_{1,0}$,
\begin{align}\label{vphi,chi}
\wh{U}_{1, 0}(y_2)=\left\{
\begin{array}{l}
(\wh{u}_{1, 0}+\f{i\xi_1}{\lam} \wh{b}_{1,0})(y_2), \ \  \text{for}  \quad y_2>0, \\
-(\wh{u}_{1, 0}+\f{i\xi_1}{\lam} \wh{b}_{1,0})(-y_2), \  \  \text{for}  \quad y_2\leq0,
\end{array}\right.
\end{align}
then  $I_1$ can be rewritten as
\begin{align*}
I_1&=\f{1}{2\pi i}\int_{\Ga}e^{\lam t} \f{1}{2\om} \int^\infty_{-\infty} e^{-\om|x_2-y_2|} \wh{U}_{1, 0}(y_2) dy_2 d\lam\non\\
&=\f{1}{2\pi i}\int_{\Ga}e^{\lam t} \f{1}{2\om} e^{-\om|x_2|} * \wh{U}_{1, 0}  d\lam:=\f{1}{2\pi i}\int_{\Ga}e^{\lam t} I_{1, \lam}  d\lam.
\end{align*}
Consider the odd extension of $\wt{I}_{1, \lam}$ of $I_{1, \lam}$, then
\begin{align*}
\cF_{x_2}(\wt{I}_1)&=\f{1}{2\pi i}\int_{\Ga}e^{\lam t}\cF_{x_2} (\wt{I}_{1, \lam})d\lam\non\\
&=\f{1}{2\pi i}\int_{\Ga}e^{\lam t}\cF_{x_2}(\f{1}{2\om} e^{-\om|x_2|} ) \cF_{x_2}(\wh{U}_{1, 0}  ) d\lam\non\\
&=\f{1}{2\pi i}\int_{\Ga}e^{\lam t}\f{1}{\om^2+\xi_2^2} \cF(U_{1, 0}) d\lam\non\\
&=\f{1}{2\pi i}\int_{\Ga}\f{\lam e^{\lam t}}{(\lam-\lam_+)(\lam-\lam_-)}\cF(U_{1, 0}) d\lam\non\\
&=\Big(\f{\lam_+ e^{\lam_+ t}}{\lam_+-\lam_-}-\f{\lam_- e^{\lam_- t}}{\lam_+-\lam_-}\Big)\cF(U_{1, 0}),
\end{align*}
where we define
\begin{align*}
\lam_\pm=\left\{
\begin{array}{l}
-\f{|\xi|^2}{2}\pm\f{i}{2} \sqrt{4|\xi_1|^2-|\xi|^4}, \quad |\xi|^2\leq 2|\xi_1|,\\
-\f{|\xi|^2}{2} \pm \f12\sqrt{|\xi|^4-4|\xi_1|^2} ,\quad |\xi|^2> 2|\xi_1|.
\end{array}\right.\end{align*}
The decay of $I_1$ can be obtained by the result of  whole space \cite{W}.

 $I_2$  can be rewritten as
 \begin{align*}
I_2=&\f{i\xi_1}{|\xi_1|} \f{1}{2\pi i}\int_{\Ga}e^{\lam t} \int^\infty_0  e^{-\om y_2} (\wh{u}_{2, 0}+\f{i\xi_1}{\lam} \wh{b}_{2, 0})(y_2) dy_2 \f{1}{\om-|\xi_1|} e^{-|\xi_1|x_2}  d\lam\\
&-\f{i\xi_1}{|\xi_1|} \f{1}{2\pi i}\int_{\Ga}e^{\lam t} \int^\infty_0  e^{-\om (x_2+y_2)} (\wh{u}_{2, 0}+\f{i\xi_1}{\lam} \wh{b}_{2, 0})(y_2) dy_2 \f{1}{\om-|\xi_1|} d\lam\\
:=&J_{1}+J_{2}.
\end{align*}

 {\bf Case 1.  $|\xi_1|\leq 2$.}

We first consider $J_2$ and divide it as
\beno J_2=-\f{i\xi_1}{|\xi_1|} \f{1}{2\pi i}\sum_{i=1}^{7}\int_{\Ga_i}e^{\lam t} \int^\infty_0  e^{-\om (x_2+y_2)} (\wh{u}_{2, 0}+\f{i\xi_1}{\lam} \wh{b}_{2, 0})(y_2) dy_2 \f{1}{\om-|\xi_1|}  d\lam:= \sum_{i=1}^{7} J_2^i.
\eeno

For the $\wh{u}_{2,0}$ part of $J_2^1$, by the definition of $\Ga_1: \lam= -\eta \ (\eta: 0 \to |\xi_1|)$, we have
\begin{align}\label{I11}
&-\f{1}{2\pi i} \int_{\Ga_1^{(+)}\cup \Ga_1 ^{(-)} } \f{ e^{\lam t} }{\om-|\xi_1|} \int^\infty_0 e^{-\om(x_2+y_2)}  \wh{u}_{2, 0} (y_2) dy_2 d\lam\non\\
&=-\f{1}{2\pi i} \Big(\int^{|\xi_1|}_{0}\f{ e^{-\eta t} }{-i|\om|-|\xi_1|} \int^\infty_0 e^{i|\om|(x_2+y_2)}   \wh{u}_{2, 0} (y_2)  dy_2 d\eta\non\\
&\qquad-\int^{|\xi_1|}_{0}\f{ e^{-\eta t} }{i|\om|-|\xi_1|} \int^\infty_0 e^{-i|\om|(x_2+y_2)}   \wh{u}_{2, 0} (y_2) dy_2 d\eta\Big)
\end{align}
with
\begin{align*}
\om=\sqrt{\lam+|\xi_1|^2+\f{|\xi_1|^2}{\lam}}=\mp i \sqrt{\eta-|\xi_1|^2+\f{|\xi_1|^2}{\eta}}=\mp i|\om|.
\end{align*}
We now make the change of variables $\xi_2=|\om|$. By this change of variables, $\eta$ is written as $\eta=\eta_-$. Here
\begin{align*}
 \eta_\pm:=\f{ |\xi|^2}{2} \pm \f12 \sqrt{|\xi|^4-4|\xi_1|^2}=\f{|\xi|^2}{2}\pm\f{|\xi|^2}{2} \sqrt{1-\f{4|\xi_1|^2}{|\xi|^4}},
\end{align*}
with $|\xi|^4-4|\xi_1|^2\geq 0$. This can be seen as follows.
Since $\eta-|\xi_1|^2+\f{|\xi|^2}{\eta}=|\xi_2|^2$, we have $\eta^2-|\xi|\eta+|\xi_1|^2=0$, which implies that $|\xi|^4-4|\xi_1|^2\geq 0$ and $\eta=\eta_\pm$ since $\eta\in \R$. We should then take the branch $\eta=\eta_-$ due to the condition $\eta\in (0, |\xi_1|)$.
We also note that $\eta_\pm=-\lam_\mp$. It then follows that this change of variables yields
\beno
d\eta_-=\xi_2\Big(1-\f{|\xi|^2}{\sqrt{|\xi|^4-4|\xi_1|^2}}\Big) d\xi_2 =2 \xi_2 \f{\lam_+}{\lam_+-\lam_-} d\xi_2.
\eeno

We have the following:
\begin{itemize}
\item
if $\f{4|\xi_1|^2}{|\xi|^4}\leq \f12$, since $(1+s)^\f12 =1+\f{s}{2}\int_0^1(1+\theta s)^{-\f12} d\theta$, we have
 \begin{align*}
 \f{|\xi|^2}{2}\sqrt{1-\f{4|\xi_1|^2}{|\xi|^4}}=\f{|\xi|^2}{2}\Big(1-\f{2|\xi_1|^2}{|\xi|^4}\int_0^1(1+\theta s)^{-\f12} d\theta\Big), \ s=-\f{4|\xi_1|^2}{|\xi|^4},
 \end{align*}
  and hence
   $$\f{|\xi_1|^2}{|\xi|^2} \leq \eta_- \leq \f{\sqrt{2}|\xi_1|^2}{|\xi|^2},$$
   it means $\eta_- \sim \f{|\xi_1|^2}{|\xi|^2}$ (Here $A \sim B$ means that $A \lesssim B$ and $B \lesssim A$).\\
 \item
if $\f12 \leq\f{4|\xi_1|^2}{|\xi|^4}\leq1$, we have $\f{|\xi|^2}{2}\sqrt{1-\f{4|\xi_1|^2}{|\xi|^4}}\leq\f{|\xi|^2}{2\sqrt{2}}$. It then follows that
 \begin{align*}
\eta_-=\f{|\xi|^2}{2}-\f{|\xi|^2}{2}\sqrt{1-\f{4|\xi_1|^2}{|\xi|^4}}\gtrsim|\xi|^2.
  \end{align*}
\end{itemize}
We mention that the most trouble case is
\beno
\eta\sim\f{|\xi_1|^2}{|\xi|^2} \quad (\text{when} \quad \f{4|\xi_1|^2}{|\xi|^4}\leq \f12),
\eeno
and the other cases can be controlled by heat kernel. For this case, we have
\beno
d \eta=2 \xi_2 \f{\lam_+}{\lam_+-\lam_-} d\xi_2, \quad \Big|\f{\lam_{+}}{\lam_{+}-\lam_{-}}\Big|\sim \f{|\xi_1|^2}{|\xi|^4},
\eeno
 and the first term on the right hand side of \eqref{I11} can be rewritten as
\begin{align*}
&-\f{1}{2\pi i} \int^\infty_{\wt{d}_1} e^{-\f{|\xi_1|^2}{|\xi|^2} t}\f{2\xi_2 }{i\xi_2+|\xi_1|} \int^\infty_0 e^{i\xi_2(x_2+y_2)}   \wh{u}_{2, 0} (y_2)  dy_2 \f{\lam_+}{\lam_+-\lam_-}  d\xi_2\non\\
&=-\f{1}{\pi i} \int_{\R}\vphi(\xi_2) e^{i\xi_2 x_2} \f{\lam_+}{\lam_+-\lam_-}  e^{-\f{|\xi_1|^2}{|\xi|^2} t}\f{\xi_2 }{i\xi_2+|\xi_1|} \int^\infty_0 e^{i\xi_2y_2}   \wh{u}_{2, 0} (y_2)  dy_2  d\xi_2\non\\
&=2i\cF_{x_2}^{-1}\Big(\vphi(\xi_2)\f{\lam_+}{\lam_+-\lam_-}  e^{-\f{|\xi_1|^2}{|\xi|^2} t}  \f{\xi_2 }{i\xi_2+|\xi_1|}\cF_{y_2}\big(\chi  \wh{u}_{2, 0} \big) (-\xi_2) \Big)(x_2),
\end{align*}
where   $\vphi, \chi$ are   cut-off functions defined as
\begin{align}\label{vphi,chi}
\vphi(\xi_2)=\left\{
\begin{array}{l}
1, \  \text{for} \ \xi_2\geq \wt{d}_1, \\
0, \  \text{for} \ \xi_2<\wt{d}_1,
\end{array}\right.
\quad \text{and}\quad
\chi(y_2)=\left\{
\begin{array}{l}
1, \  \text{for} \ y_2\geq 0, \\
0, \  \text{for} \ y_2<0,
\end{array}\right.
\end{align}
and $\wt{d}_1=\sqrt{2\sqrt{2}|\xi_1|-|\xi_1|^2}$.

By  the Plancherel theorem and $|\xi|^2>2 |\xi_1|$, we have the $L^2$ estimate
\begin{align*}
&\Big\|\textbf{1}_{|\xi_1|\leq 2}\cF_{x_2}^{-1}\Big(\vphi(\xi_2) \f{|\xi_1|^2}{|\xi|^4} e^{-\f{|\xi_1|^2}{|\xi|^2} t}  \f{\xi_2 }{i\xi_2+|\xi_1|}\cF_{y_2}\big(\chi  \wh{u}_{2, 0} \big) (-\xi_2) \Big)(x_2)\Big\|_{L^2_{\xi_1}L^2_{x_2}}\\
&\lesssim\Big\|\textbf{1}_{|\xi_1|\leq 2}\f{|\xi_1|^2}{|\xi|^4} e^{-\f{|\xi_1|^2}{|\xi|^2} t} \f{\xi_2 }{i\xi_2+|\xi_1|}\cF\big(\chi u_{2, 0} \big) (-\xi_2)   \Big\|_{L^2_{\xi_1}L^2_{\xi_2}}\non\\
&\lesssim\Big\|\textbf{1}_{|\xi_1|\leq 2}\f{|\xi_1|}{|\xi|^2} e^{-\f{|\xi_1|^2}{|\xi|^2} t} \cF \big(\chi u_{2, 0} \big)   \Big\|_{L^2_{\xi_1}L^2_{\xi_2}}\Big\|\textbf{1}_{|\xi_1|\leq 2}\f{\xi_2 }{i\xi_2+|\xi_1|}\Big\|_{L^\infty_{\xi_1}L^\infty_{\xi_2}}\non\\
&\lesssim\Big\|\textbf{1}_{|\xi_1|\leq 2}(\textbf{1}_{|\xi|\leq 1}+\textbf{1}_{|\xi|\geq 1})\f{|\xi_1|}{|\xi|^2} e^{-\f{|\xi_1|^2}{|\xi|^2} t} \cF \big(\chi u_{2, 0} \big)   \Big\|_{L^2_{\xi_1}L^2_{\xi_2}}.
\end{align*}

As for the $\textbf{1}_{|\xi|\leq 1}$ part, we have
\begin{align*}
&\Big\|\textbf{1}_{|\xi_1|\leq 2}\textbf{1}_{|\xi|\leq 1}\f{|\xi_1|}{|\xi|^2} e^{-\f{|\xi_1|^2}{|\xi|^2} t} \cF \big(\chi u_{2, 0} \big)   \Big\|_{L^2_{\xi_1}L^2_{\xi_2}}\non\\
&\lesssim\Big\|\textbf{1}_{|\xi_1|\leq 2}\textbf{1}_{|\xi|\leq 1}\big(\f{|\xi_1|}{|\xi|}\big)^\f12 \big(\f{|\xi_1|}{|\xi|^3}\big)^\f12 e^{-\f{|\xi_1|^2}{|\xi|^2} t} \cF \big(\chi u_{2, 0} \big)  \Big\|_{L^2_{\xi_1}L^2_{\xi_2}}\non\\
&\lesssim t^{-\f14}\Big\| \textbf{1}_{|\xi_1|\leq 2}\textbf{1}_{|\xi|\leq 1}\f{1}{|\xi|^\f12} e^{-\f{|\xi_1|^2}{|\xi|^2} t} \cF \big(\chi u_{2, 0} \big)   \Big\|_{L^2_{\xi_1}L^2_{\xi_2}}\non\\
&\lesssim t^{-\f14}\sum_{j\leq 1}\Big\|\textbf{1}_{|\xi_1|\leq 2}j^{-\f12} e^{-\f{|\xi_1|^2}{j^2} t} \cF\big( P_j (\chi u_{2, 0} )  \big)\Big\|_{L^2_{\xi_1}L^2_{\xi_2}}\non\\
&\lesssim t^{-\f14} \sum_{j\leq 1} j^{-\f12} \Big\|\textbf{1}_{|\xi_1|\leq 2} e^{-\f{|\xi_1|^2}{j^2} t}\
\Big\|_{L^2_{\xi_1}} \Big\|\cF\big( P_j (\chi u_{2, 0} )  \big)\Big\|_{L^\infty_{\xi_1}L^\infty_{\xi_2}}\big(\int_{|\xi_2|\leq j} d \xi_2\Big)^\f12\non\\
&\lesssim t^{-\f12} \sum_{j\leq 1} j^\f12\| \chi u_{2, 0}  \|_{L^1}\non\\
&\lesssim t^{-\f12} \|u_{2, 0}\|_{L^1},
\end{align*}
and as for the $\textbf{1}_{|\xi|\geq 1}$ part, we have
\begin{align*}
&\Big\|\textbf{1}_{|\xi_1|\leq 2}\textbf{1}_{|\xi|\geq 1}\f{|\xi_1|}{|\xi|^2} e^{-\f{|\xi_1|^2}{|\xi|^2} t}  \cF \big(\chi u_{2, 0} \big)   \Big\|_{L^2_{\xi_1}L^2_{\xi_2}}\non\\
&\lesssim\Big\|\textbf{1}_{|\xi_1|\leq 2}\textbf{1}_{|\xi|\geq 1}\f{|\xi_1|}{|\xi|} e^{-\f{|\xi_1|^2}{|\xi|^2} t}  \cF \big(\chi u_{2, 0} \big)   \Big\|_{L^2_{\xi_1}L^2_{\xi_2}}\non\\
&\lesssim t^{-\f12}\Big\|\textbf{1}_{|\xi_1|\leq 2}\textbf{1}_{|\xi|\geq 1} e^{-\f{|\xi_1|^2}{|\xi|^2} t}\Big\|_{L^\infty_{\xi_1}L^\infty_{\xi_2}} \Big\| \cF \big(\chi u_{2, 0} \big)  \Big\|_{L^2_{\xi_1}L^2_{\xi_2}}\non\\
&\lesssim t^{-\f12} \|u_{2, 0}\|_{L^2}.
\end{align*}
By  $|\xi|^2>2 |\xi_1|$, and the boundary condition $u_{2,0}|_{x_2=0}=0$, we have the $L^\infty$ estimate
\begin{align}\label{J2infty}
&\Big\|\textbf{1}_{|\xi_1|\leq 2}\cF_{x_2}^{-1}\Big(\vphi(\xi_2) \f{|\xi_1|^2}{|\xi|^4} e^{-\f{|\xi_1|^2}{|\xi|^2} t}  \f{\xi_2 }{i\xi_2+|\xi_1|}\cF_{y_2}\big(\chi  \wh{u}_{2, 0} \big) (-\xi_2) \Big)(x_2)\Big\|_{L^1_{\xi_1}L^\infty_{x_2}}\non\\
&\lesssim\Big\|\textbf{1}_{|\xi_1|\leq 2}\f{|\xi_1|^2}{|\xi|^4} e^{-\f{|\xi_1|^2}{|\xi|^2} t}   \cF \big(\chi u_{2, 0} \big)   \Big\|_{L^1_{\xi_1}L^1_{\xi_2}}\non\\
&\lesssim\Big\|\textbf{1}_{|\xi_1|\leq 2}\f{|\xi_1|}{|\xi|^2} e^{-\f{|\xi_1|^2}{|\xi|^2} t}   \cF \big(\chi u_{2, 0} \big)   \Big\|_{L^1_{\xi_1}L^1_{\xi_2}}\non\\
&\lesssim t^{-\f12} \sum_{j\geq 0}  j^{-1}  \Big\|\textbf{1}_{|\xi_1|\leq 2}e^{-\f{|\xi_1|^2}{j^2} t} \cF\big( P_j (\chi u_{2, 0} )  \big)\Big\|_{L^1_{\xi_1}L^1_{\xi_2}}\non\\
&\lesssim t^{-\f12} \Big(\sum_{j\leq 1} j^{-1} \Big\| \textbf{1}_{|\xi_1|\leq 2}e^{-\f{|\xi_1|^2}{j^2} t}\Big\|_{L^1_{\xi_1}} \Big\|\cF\big( P_j (\chi u_{2, 0} )  \big)\Big\|_{L^\infty_{\xi_1}L^\infty_{\xi_2}}\int_{|\xi_2|\leq j} d\xi_2\non\\
&\quad+\sum_{j\geq 1} j^{-1}\Big\| \textbf{1}_{|\xi_1|\leq 2}e^{-\f{|\xi_1|^2}{j^2} t}\Big\|_{L^1_{\xi_1}} \Big\|\cF\big( P_j (\chi u_{2, 0} )  \big)\Big\|_{L^\infty_{\xi_1}L^2_{\xi_2}} \Big(\int_{|\xi_2|\leq j} d\xi_2\Big)^\f12\Big)\non\\
&\lesssim t^{-1}\Big( \sum_{j\leq 1} j \|\chi u_{2, 0}\|_{L^1}+\sum_{j\geq 1}  j^{-\delta} \Big\| j^{\f12+\delta}\cF\big( P_j (\chi u_{2, 0} )  \big)\Big\|_{L^\infty_{\xi_1}L^2_{\xi_2}}\Big)\non\\
&\lesssim  t^{-1}\Big(\|u_{2, 0}\|_{L^1}+\Big\| \na^{\f12+\delta} u_{2, 0}  \Big \|_{L^1_{x_1}L^2_{x_2}}\Big),
\end{align}
where we use
\begin{align*}
&\Big\| j^{\f12+\delta}\cF\big( P_j (\chi u_{2, 0} )  \big)\Big\|_{L^\infty_{\xi_1}L^2_{\xi_2}}\\
\lesssim &\Big\| |\xi|^{\f12+\delta}\cF\big(\chi u_{2, 0}   \big)\Big\|_{L^\infty_{\xi_1}L^2_{\xi_2}}\\
\lesssim &\Big\||\xi|^{\f12+\delta}\int^\infty_0 e^{i\xi_2y_2}   \wh{u}_{2, 0} (y_2)  dy_2 \Big\|_{L^\infty_{\xi_1}L^2_{\xi_2}}\\
\lesssim &\Big\|\int^\infty_0 e^{i\xi_2y_2} \na^{\f12+\delta}\wh{u}_{2, 0} (y_2)  dy_2\Big\|_{L^\infty_{\xi_1}L^2_{\xi_2}} \\
\lesssim &\Big\| \na^{\f12+\delta} u_{2, 0}  \Big \|_{L^1_{x_1}L^2_{x_2}},
\end{align*}
and the fractional derivatives satisfies the generalized Leibniz rule \cite{Li}.

The $\wh{b}_{1,0}$ part of $J_2^1$ can be estimated similarly since
\beno
\f{|\xi_1|^2}{|\xi|^4} \f{|\xi_1|}{|\lam|}\sim\f{|\xi_1|^2}{|\xi|^4} \f{|\xi|^2}{|\xi_1|}=\f{|\xi_1|}{|\xi|^2}.
\eeno
Consider the case  $0<t<1$, we have,
\begin{align*}
\Big\|\f{|\xi_1|}{|\xi|^2} e^{-\f{|\xi_1|^2}{|\xi|^2} t} \cF \big(\chi u_{2, 0} \big)   \Big\|_{L^2_{\xi_1}L^2_{\xi_2}}\lesssim \|u_{2, 0}\|_{L^2},\\
\Big\|\f{|\xi_1|}{|\xi|^2} e^{-\f{|\xi_1|^2}{|\xi|^2} t}   \cF \big(\chi u_{2, 0} \big)   \Big\|_{L^1_{\xi_1}L^1_{\xi_2}}\lesssim\|u_{2, 0}\|_{L^1}.
\end{align*}

For the $\wh{u}_{2,0}$ part of $J_2^2$, by the definition of $\Ga_2$: $ \lam=\lam'_+-\eta+i\big(-\text{Im} \lam'_++D(\eta, \xi_1)\big) \ (0< \eta <d_0)$, we have
\begin{align}\label{I12}
&-\f{1}{2\pi i} \int_{\Ga_2^{(+)}\cup \Ga_2 ^{(-)} } e^{\lam t}  \f{1}{ \om-|\xi_1|} \int^\infty_0 e^{-\om(x_2+y_2)}  \wh{u}_{2, 0} (y_2) dy_2 d\lam\non\\
&=-\f{1}{2\pi i} \Big(\int^{d_0}_0\f{ e^{\lam t}}{i|\om|-|\xi_1|} \int^\infty_0 e^{-i|\om|(x_2+y_2)}   \wh{u}_{2, 0} (y_2) dy_2 (-1+i \f{\text{Re} \lam'_+-\eta}{D(\eta, \xi_1)} ) d\eta\non\\
&\quad\qquad-\int^{d_0}_0\f{ e^{\lam t}}{-i|\om|-|\xi_1|} \int^\infty_0 e^{i|\om|(x_2+y_2)}   \wh{u}_{2, 0} (y_2) dy_2 (-1+i \f{\text{Re} \lam'_+-\eta}{D(\eta, \xi_1)})  d\eta\Big),\end{align}
with
\begin{align*}
\om&=\sqrt{\f{(\lam-\lam'_+)(\lam-\lam'_-)}{\lam}}\\
&=\sqrt{\f{\Big(-\eta+i\big(-\text{Im} \lam'_++D(\eta, \xi_1)\big)\Big)\Big(-\eta+i\big(\text{Im} \lam'_++D(\eta, \xi_1)\big)\Big)}{\lam'_+-\eta+i\big(-\text{Im} \lam'_++D(\eta, \xi_1)\big)}}\\
&=\sqrt{\f{\big(-\eta+iD(\eta, \xi_1)\big)^2+(\text{Im} \lam'_+)^2}{\text{Re} \lam'_+-\eta+i D(\eta, \xi_1)}}\\
&=\sqrt{\f{2\eta^2 -2\text{Re} \lam'_+ \eta-2\eta D(\eta, \xi_1) i}{\text{Re} \lam'_+-\eta+i D(\eta, \xi_1)}}\\
&=\sqrt{-2\eta}=\pm i \sqrt{2\eta}=\pm i|w|.
\end{align*}
By changing the variables $\xi_2=|\om|$, we have
\begin{align*}
\lam&= \text{Re} \lam'_+-\eta+i D(\eta, \xi_1)\\
&=-\f{|\xi|^2}{2}+i \sqrt{|\xi_1|^2-\f{|\xi_1|^4}{4}-\f{|\xi_1|^2|\xi_2|^2}{2}-\f{|\xi_2|^4}{4}}\\
&=-\f{|\xi|^2}{2}+i \sqrt{|\xi_1|^2-\f{|\xi|^4}{4}}=\lam_+,
\end{align*}
where we define
\begin{align*}
\lam_\pm=
-\f{|\xi|^2}{2}\pm i\sqrt{|\xi_1|^2-\f{|\xi|^4}{4}}.
\end{align*}
It then follows that
\beno
-1+i \f{\text{Re} \lam'_+-\eta}{D(\eta, \xi_1)} =-1+i \f{-\f{|\xi|^2}{2}}{\text{Im} \lam_+}=-\f{2\lam_+}{\lam_+-\lam_-}.
\eeno

Thus \eqref{I12} can be rewritten as
\begin{align}\label{J22}
&\f{1}{\pi i } \Big(\int^{\wt{d}}_0 \f{\lam_+ e^{\lam_+ t}}{\lam_+-\lam_-}
\f{1}{i\xi_2-|\xi_1|} \int^\infty_0 e^{-i\xi_2(x_2+y_2)}   \wh{u}_{2, 0} (y_2) dy_2  \xi_2d\xi_2\non\\
&\quad\qquad+\int^{\wt{d}}_0 \f{\lam_+ e^{\lam_+ t}}{\lam_+-\lam_-}
\f{1}{i\xi_2+|\xi_1|} \int^\infty_0 e^{i\xi_2(x_2+y_2)}   \wh{u}_{2, 0} (y_2) dy_2  \xi_2 d\xi_2\Big),
\end{align}
where $\wt{d}=\sqrt{2|\xi_1|-|\xi_1|^2}$.

 Similarly, the $\wh{u}_{1,0}$ part of $J_2^3$ can be rewritten as
\begin{align}\label{J23}
&-\f{1}{\pi i } \Big(\int^{\wt{d}}_0 \f{\lam_- e^{\lam_- t}}{\lam_+-\lam_-}
\f{1}{i\xi_2-|\xi_1|} \int^\infty_0 e^{-i\xi_2(x_2+y_2)}   \wh{u}_{2, 0} (y_2) dy_2  \xi_2 d\xi_2\non\\
&\quad\qquad+\int^{\wt{d}}_0 \f{\lam_- e^{\lam_- t}}{\lam_+-\lam_-}
\f{1}{i\xi_2+|\xi_1|} \int^\infty_0 e^{i\xi_2(x_2+y_2)}   \wh{u}_{2, 0} (y_2) dy_2  \xi_2 d\xi_2\Big).
\end{align}

By the definition of $\wt{d}$, we have $|\xi_1|\geq \f{|\xi|^2}{2}$ and divide it as follows
\begin{itemize}
\item if $|\xi_1|\geq |\xi|^2$, we have $|\lam_+-\lam_-|=\sqrt{4|\xi_1|^2-|\xi|^4}\geq \sqrt{3}|\xi|^2$, then
\beno
\Big|\f{ e^{\lam_\pm t} \lam_\pm }{\lam_+-\lam_-}\Big|\lesssim  e^{-\f{|\xi|^2}{2}t}, \quad \Big|\f{\lam_-e^{\lam_- t}-\lam_+ e^{\lam_+ t}}{\lam_+-\lam_-}\Big|\lesssim  e^{-\f{|\xi|^2}{2}t};
\eeno
\item  if $\f{|\xi|^2}{2}\leq|\xi_1| \leq |\xi|^2$, we have
\begin{align*}
\Big|\f{\lam_-e^{\lam_- t}-\lam_+ e^{\lam_+ t}}{\lam_+-\lam_-}\Big|=\Big|e^{\lam_-t }+\f{\lam_+}{\xi_1} \Big(\f{\xi_1(e^{\lam_- t}- e^{\lam_+ t})}{\lam_+-\lam_-}\Big)\Big|\lesssim e^{-\f{|\xi|^2}{4} t},
\end{align*}
where we use
\beno
\Big|\f{\xi_1(e^{\lam_- t}- e^{\lam_+ t})}{\lam_+-\lam_-}\Big|\lesssim \Big|\xi_1 t e^{-\f{|\xi|^2}{2} t}\Big|\Big|\f{\sin(\sqrt{|\xi_1|^2 -\f{|\xi|^4}{4}} t)}{\sqrt{|\xi_1|^2 -\f{|\xi|^4}{4} } t}\Big|\lesssim e^{-\f{|\xi|^2}{4} t}.
\eeno
\end{itemize}
Thus, the $\wh{u}_{1,0}$ part  of $J_2^2+J_2^3$ can be controlled by heat kernel, and we obtain  the $L^2$ estimate
\begin{align*}
&\Big\| \int^{\wt{d}}_0\f{ \lam_- e^{\lam_- t}- \lam_+ e^{\lam_+t} }{\lam_+-\lam_-}\f{\xi_2}{i\xi_2-|\xi_1|} \int^\infty_0 e^{-i\xi_2(x_2+y_2)}   \wh{u}_{2, 0} (y_2) dy_2  d\xi_2\Big\|_{L^2_{\xi_1}L^2_{x_2}}\\
=&\Big\| \int_{\R} (1-\vphi(\xi_2))\chi(\xi_2)\f{ \lam_- e^{\lam_- t}- \lam_+ e^{\lam_+t} }{\lam_+-\lam_-}\f{\xi_2}{i\xi_2-|\xi_1|}  e^{-i\xi_2x_2} \cF_{y_2}\Big(\chi(y_2)  \wh{u}_{2, 0} (y_2)  \Big)(\xi_2) d\xi_2\Big\|_{L^2_{\xi_1}L^2_{x_2}}\\
=&\Big\| \cF^{-1}_{x_2}\Big((1-\vphi(\xi_2))\chi(\xi_2)\f{ \lam_- e^{\lam_- t}- \lam_+ e^{\lam_+t} }{\lam_+-\lam_-}\f{\xi_2}{i\xi_2-|\xi_1|} \cF_{y_2}\Big(\chi(y_2)  \wh{u}_{2, 0} (y_2)  \Big)(\xi_2) \Big)(-x_2) \Big\|_{L^2_{\xi_1}L^2_{x_2}}\\
=&\Big\| \f{ \lam_- e^{\lam_- t}- \lam_+ e^{\lam_+t} }{\lam_+-\lam_-}\f{\xi_2}{i\xi_2-|\xi_1|} \cF_{y_2}\Big(\chi(y_2)  \wh{u}_{2, 0} (y_2)  \Big)(\xi_2)\Big\|_{L^2_{\xi_1}L^2_{\xi_2}}\\
\lesssim&\Big\|\f{ \lam_- e^{\lam_- t}- \lam_+ e^{\lam_+t} }{\lam_+-\lam_-}
\f{\xi_2}{i\xi_2-|\xi_1|}\Big\|_{L^2_{\xi_1}L^2_{\xi_2}}\Big\|\cF_{y_2}\Big(\chi(y_2)  \wh{u}_{2, 0} (y_2)  \Big)\Big\|_{L^\infty_{\xi_2}}\\
\lesssim&\| e^{-\f{|\xi|^2}{4}t} \|_{L^2_{\xi_1}L^2_{\xi_2}}\|u_{2,0}\|_{L^1}\\
\lesssim& t  ^{-\f12} \|u_{2,0}\|_{L^1},
\end{align*}
and the $L^\infty$ estimate
\begin{align*}
&\Big\|\int^{\wt{d}}_0\f{ \lam_- e^{\lam_- t}- \lam_+ e^{\lam_+t} }{\lam_+-\lam_-}\f{\xi_2}{i\xi_2-|\xi_1|} \int^\infty_0 e^{-i\xi_2(x_2+y_2)}   \wh{u}_{2, 0} (y_2) dy_2   d\xi_2 \Big\|_{L^1_{\xi_1}L^\infty_{x_2}}\\
&\lesssim\Big\|\f{ \lam_- e^{\lam_- t}- \lam_+ e^{\lam_+t} }{\lam_+-\lam_-}
\cF_{x_2}\big(\wh{\chi u}_{2,0}\big)  \Big\|_{L^1_{\xi_1}L^1_{\xi_2}}\\
&\lesssim \|e^{-\f{|\xi|^2}{4}t} \|_{L^1_{\xi_1}L^1_{\xi_2}}\|\cF \big(\chi u_{2,0}\big) \|_{L^\infty_{\xi_1}L^\infty_{\xi_2}}\\
&\lesssim   t ^{-1} \|u_{2,0}\|_{L^1}.
\end{align*}
Similarly, we have a similar estimate for the $\wh{b}_{2,0}$ part  of $J_2^2+J_2^3$, since
\beno
 \f{|\xi_1|}{|\lam|}=\f{|\xi_1|}{|\lam_\pm|}\sim 1.
 \eeno

The estimate of $J_2^4$ is very similar as $J_2^1$;  by the definition of $\Ga_4: \lam= -\eta \ (\eta:  |\xi_1| \to \infty)$, we have
\begin{align}\label{I14}
&-\f{1}{2\pi i} \int_{\Ga_4^{(+)}\cup \Ga_4 ^{(-)} } e^{\lam t}  \f{1}{ \om-|\xi_1|} \int^\infty_0 e^{-\om(x_2+y_2)}  \wh{u}_{2, 0} (y_2) dy_2 d\lam\non\\
&=-\f{1}{2\pi i} \Big(\int_{|\xi_1|}^\infty\f{ e^{-\eta t}}{-i|\om|-|\xi_1|} \int^\infty_0 e^{i|\om|(x_2+y_2)}   \wh{u}_{2, 0} (y_2) dy_2 d\eta\non\\
&\qquad-\int_{|\xi_1|}^\infty\f{ e^{-\eta t}}{i|\om|-|\xi_1|} \int^\infty_0 e^{-i|\om|(x_2+y_2)}   \wh{u}_{2, 0} (y_2) dy_2 d\eta\Big)
\end{align}
with
\begin{align*}
\om=\sqrt{\lam+|\xi_1|^2+\f{|\xi_1|^2}{\lam}}=\mp i \sqrt{\eta-|\xi_1|^2+\f{|\xi_1|^2}{\eta}}=\mp i|\om|.
\end{align*}
We now make the change of variables $\xi_2=|\om|$. Here we should take the branch $\eta=\eta_+$, since $\eta\in (|\xi_1|, \infty)$; it then follows that
\beno
d\eta_+=\xi_2\Big(1+\f{|\xi|^2}{\sqrt{|\xi|^4-4|\xi_1|^2}}\Big) d\xi_2=-2 \xi_2 \f{\lam_-}{\lam_+-\lam_-} d\xi_2.
\eeno
We have the following:
\begin{itemize}
\item if 
 $\f{4|\xi_1|^2}{|\xi|^4}\leq \f12$,  we obtain
 $$\eta_+\geq \f{(2+\sqrt{2})|\xi|^2}{4}\gtrsim|\xi|^2.$$
 \item if
 $\f12 \leq\f{4|\xi_1|^2}{|\xi|^4}\leq1$, we have $\f{|\xi|^2}{2}\sqrt{1-\f{4|\xi_1|^2}{|\xi|^4}}\leq\f{|\xi|^2}{2\sqrt{2}}$. It then follows that
 \begin{align*}
\eta=\f{|\xi|^2}{2}+\f{|\xi|^2}{2}\sqrt{1-\f{4|\xi_1|^2}{|\xi|^4}}\in (\f{|\xi|^2}{2},\f{(2+\sqrt{2})|\xi|^2}{4}),
  \end{align*}
 it means $\eta \sim |\xi|^2$.
\end{itemize}
Then we have
\beno
\eta_+=\f{|\xi|^2}{2}+\f{|\xi|^2}{2} \sqrt{1-\f{4|\xi_1|^2}{|\xi|^4}} \sim|\xi|^2,
\eeno
and \beno
d\eta_+=-2 \xi_2 \f{\lam_-}{\lam_+-\lam_-} d\xi_2, \quad \Big|\f{\lam_-}{\lam_+-\lam_-}\Big|\sim 1,\eeno
 \eqref{I14} can be rewritten as
\begin{align*}
&\f{1}{2\pi i} \int^\infty_{\wt{d}} e^{-|\xi|^2 t}\f{2\xi_2 }{i\xi_2+|\xi_1|} \int^\infty_0 e^{i\xi_2(x_2+y_2)}   \wh{u}_{2, 0} (y_2)  dy_2 \f{\lam_-}{\lam_+-\lam_-}  d\xi_2\non\\
&=\f{1}{\pi i} \int_{\R}\vphi(\xi_2) e^{i\xi_2 x_2} \f{\lam_-}{\lam_+-\lam_-}  e^{-|\xi|^2 t}\f{\xi_2 }{i\xi_2+|\xi_1|} \int^\infty_0 e^{i\xi_2y_2}   \wh{u}_{2, 0} (y_2)  dy_2  d\xi_2\non\\
&=2i\cF_{x_2}^{-1}\Big(\vphi(\xi_2)\f{\lam_-}{\lam_+-\lam_-}  e^{-|\xi|^2 t}  \f{\xi_2 }{i\xi_2+|\xi_1|}\cF_{y_2}\big(\chi  \wh{u}_{2, 0} \big) (-\xi_2) \Big)(x_2),
\end{align*}
so it can be controlled by heat kernel, and we omit the details.

For the  $\wh{u}_{2,0}$ part  of $J_2^5$, by the definition of $\Ga_5: \lam= \e e^{i\gamma}$,  $\gamma \in( -\pi, \pi)$, we can see that
\begin{align}\label{I5}
|J_2^5|
&=\Big|\f{1}{2\pi i} \int_{\Ga_5 } e^{\lam t}  \f{1}{ \om-|\xi_1|} \int^\infty_0 e^{-\om(x_2+y_2)}  \wh{u}_{2, 0} (y_2) dy_2 d\lam\Big|\non\\
&\lesssim \f{e^{\e t} \e^\f32}{|\xi_1|} \int_{-\pi}^\pi \|e^{-Re \om y_2}\|_{L^2_{y_2}}d \gamma\| \wh{u}_{2, 0}\|_{L^2_{y_2}}\non\\
&\lesssim \f{e^{\e t} \e^\f74}{|\xi_1|^\f32} \| u_{2, 0}\|_{L^1_{x_1}L^2_{x_2}}
\rightarrow 0 \quad \text{when} \quad \e\rightarrow 0 \quad \text{ for  fixed} \quad \xi_1.
\end{align}
Indeed, there exist $\tilde{\gamma}^* \in (\f{\pi}{2},\pi)$ such that Re $ \om \gtrsim \f{|\xi|}{\sqrt{\e}}$ for $\gamma \in [0,\tilde{\gamma}^*]$ and Re $ \om \gtrsim \f{|\xi|}{\sqrt{\e}}\text{sin} \gamma$ for $\gamma \in [\tilde{\gamma}^*,\pi)$ if $\e>0$ is sufficiently small and $\xi_1 \neq 0 $ is fixed. By using this, one has
\begin{align*}
\int_0^\pi \|e^{-Re \om y_2}\|_{L^2_{y_2}}d \gamma \lesssim \int_0^\pi (Re\om)^{-\f12}d \gamma
\lesssim  \f{\e^{\f14}}{|\xi_1|^\f12}\Big(1+\int_{\tilde{\gamma}^*}^\pi (\text{sin}\gamma)^{-\f12} d\gamma\Big)  \lesssim \f{\e^{\f14}}{|\xi_1|^\f12},
\end{align*}
 and likewise
  $$\int_{-\pi}^0 \|e^{-Re \om y_2}\|_{L^2_{y_2}}d \gamma  \lesssim \f{\e^{\f14}}{|\xi_1|^\f12}.$$
We here note
\begin{align*}
\om^2 &=\lam+|\xi_1|^2+\f{|\xi_1|^2}{\lam}\\
&=(\e\cos \gamma+|\xi_1|^2+\f{|\xi_1|^2}{\e}\cos \gamma)+i(\e\sin\gamma-\f{|\xi_1|^2}{\e} \sin \gamma)\\
&=(\e^2+|\xi_1|^4+\f{|\xi_1|^4}{\e^2}+2\e \cos \gamma |\xi_1|^2+2|\xi_1|^2 \cos 2\gamma+\f{2|\xi_1|^4}{\e} \cos \gamma)^\f12 e^{2i\theta_5},
\end{align*}
with
$$|\om|\gtrsim \f{|\xi_1|}{\sqrt{\e}}, \quad |\om-|\xi_1||\geq|\om|-|\xi_1|\gtrsim \f{|\xi_1|}{\sqrt{\e}}.$$
  The desired estimate of \eqref{I5} is then obtained.

The  $\wh{b}_{1,0}$ part  of $J_2^5$ also goes to 0, since $\f{|\xi_1|}{|\lam|}=\f{|\xi_1|}{\e}$.\\

For the  $\wh{u}_{2,0}$ part  of $J_2^6$, by the definition of $\Ga_6: \lam=-|\xi_1|+ \e e^{i\gamma}$,   $\gamma: -\pi \to -\f{\pi}{2},  -\f{\pi}{2} \to 0, 0 \to \f{\pi}{2}, \f{\pi}{2}\to \pi $, we have
\begin{align}\label{I6}
|J_2^6|
&=\Big|\f{1}{2\pi i} \int_{\Ga_6 } e^{\lam t}  \f{1}{\om-|\xi_1|} \int^\infty_0 e^{-\om(x_2+y_2)}  \wh{u}_{2, 0} (y_2) dy_2 d\lam\Big|\non\\
&\lesssim  \f{e^{-|\xi_1| t}e^{\e \cos \gamma t} }{|\xi_1|}\int^\infty_0 \big|e^{-\om(x_2+y_2)} \big| \big|  \wh{u}_{2, 0} (y_2)\big |dy_2\Big| \text{Length \ of }\  \Gamma_6\Big|\non\\
&\lesssim \f{\e e^{\e t} }{|\xi_1|} \| \wh{u}_{2, 0}\|_{L^1_{y_2}}
\rightarrow 0 \quad \text{when} \quad \e\rightarrow 0 \quad \text{ for  fixed} \quad \xi_1.
\end{align}
Here  we use
\begin{align*}
\om^2 &=\lam+|\xi_1|^2+\f{|\xi_1|^2}{\lam}\\
&=-|\xi_1|+\e \cos \gamma+|\xi_1|^2+\f{|\xi_1|^2(-|\xi_1|+\e \cos \gamma)}{|\xi_1|^2+\e^2 -2\e |\xi_1| \cos \gamma}+i\f{\e^2 -2\e |\xi_1| \cos \gamma}{|\xi_1|^2+\e^2 -2\e |\xi_1| \cos \gamma}\e \sin \gamma\\
&=\Big(-2|\xi_1|+|\xi_1|^2\Big)+\e^2\Big(\f{\e\cos \gamma-|\xi_1|\cos 2 \gamma}{|\xi_1|^2+\e^2 -2\e |\xi_1| \cos \gamma}+i\f{\e \sin \gamma-|\xi_1|\sin 2\gamma}{|\xi_1|^2+\e^2 -2\e |\xi_1| \cos \gamma}\Big)\\
&=\Big(-2|\xi_1|+|\xi_1|^2\Big)+\e^2\f{\e e^{i\gamma} -|\xi_1|e^{2i\gamma}}{(|\xi_1|-\e\cos\gamma)^2+(\e \sin \gamma)^2}
\end{align*}
with
\begin{align*}
|\om^2|&\geq \Big|2|\xi_1|-|\xi_1|^2\Big|-\e^2\Big|\f{\e e^{i\gamma} -|\xi_1|e^{2i\gamma}}{(|\xi_1|-\e\cos\gamma)^2+(\e \sin \gamma)^2}\Big|\\
&\geq \Big|2|\xi_1|-|\xi_1|^2\Big|-\e^2\f{|\xi_1|+\e}{(|\xi_1|-\e)^2}\\
&\geq \f{2|\xi_1|-|\xi_1|^2}{2},
\end{align*}
where we use $$\lim_{\e\rightarrow 0}\e^2\f{|\xi_1|+\e}{(|\xi_1|-\e)^2}=0$$ for fixed $\xi_1$. Similarly, we obtain
\begin{align*}
|\om^2|&\leq \Big|2|\xi_1|-|\xi_1|^2\Big|+\e^2\Big|\f{\e e^{i\gamma} -|\xi_1|e^{2i\gamma}}{(|\xi_1|-\e\cos\gamma)^2+(\e \sin \gamma)^2}\Big|\\
&\leq \Big|2|\xi_1|-|\xi_1|^2\Big|+\e^2\f{|\xi_1|+\e}{(|\xi_1|-\e)^2}\\
&\leq 2|\xi_1|-|\xi_1|^2,
\end{align*}
then, since $Re\ \om^2<0, \ Im\ \om^2\rightarrow 0$, $\om$ is a pure imaginary number, we have
 $$|\om-|\xi_1||\gtrsim |\xi_1|.$$

The  $\wh{b}_{1,0}$ part  of $J_2^6$ also goes to 0, since
\beno
\f{|\xi_1|}{|\lam|}=\f{|\xi_1|}{(|\xi_1|^2+\e^2-2\e|\xi_1|\cos \gamma)^\f12}.
\eeno

For the  $\wh{u}_{2,0}$ part  of $J_2^7$, by the definition of ${\Ga_7}=\Big\{\lam=R e^{i\gamma}; \ \gamma: \theta_0 \leq |\gamma|< \pi \Big\}$, since $Re \lam<0$, and $|\pi-\theta_0|<\f\pi3 $,  we have
\begin{align}\label{I7}
|J_2^7|
&=\Big|\f{1}{2\pi i} \int_{\Ga_7 } e^{\lam t}  \f{1}{ \om-|\xi_1|} \int^\infty_0 e^{-\om(x_2+y_2)}  \wh{u}_{2, 0} (y_2) dy_2 d\lam\Big|\non\\
&\lesssim  \f{e^{R cos \theta_0 t}}{\sqrt{R}}\int^\infty_0 \big|e^{-\sqrt{R}(x_2+y_2)} \big| \big|  \wh{u}_{2, 0} (y_2)\big |dy_2\Big| \text{Length \ of }\  \Gamma_7\Big|\non\\
&\lesssim e^{-\f12Rt}\sqrt{R} \|e^{-\sqrt{R}y_2}\|_{L^2_{y_2}} \| \wh{u}_{2, 0}\|_{L^2_{y_2}}\non\\
&\lesssim e^{-\f12Rt}R^\f14 \| u_{2, 0}\|_{L^1_{x_1}L^2_{x_2}}
\rightarrow 0 \quad \text{when} \quad R\rightarrow \infty \quad \text{ for  fixed} \quad \xi_1,
\end{align}
where  we use
\begin{align*}
\om^2 &=\lam+|\xi_1|^2+\f{|\xi_1|^2}{\lam}\\
&=(R\cos \gamma+|\xi_1|^2+\f{|\xi_1|^2}{R}\cos \gamma)+i(R\sin\gamma-\f{|\xi_1|^2}{R} \sin \gamma)\\
&=(R^2+|\xi_1|^4+\f{|\xi_1|^4}{R^2}+2R \cos \gamma |\xi_1|^2+2|\xi_1|^2 \cos 2\gamma+\f{2|\xi_1|^4}{R} \cos \gamma)^\f12 e^{2i\theta_7},
\end{align*}
with
\begin{align*}
|\om^2|&\gtrsim (R^2)^\f12=R,\\
|\om-|\xi_1||&\gtrsim |\om|-|\xi_1|\gtrsim \sqrt{R}.
\end{align*}
The  $\wh{b}_{1,0}$ part  of $J_2^7$ also goes to 0, since
\beno
\f{|\xi_1|}{|\lam|}=\f{|\xi_1|}{R}\lesssim 1.
\eeno

Now we consider $J_{1}$ and divide  it as
\begin{align*}
 J_{1}&=\f{i\xi_1}{|\xi_1|}\f{1}{2\pi i} \sum_{i=1}^{7} \int_{\Gamma_i} e^{\lam t}   \int^\infty_0e^{-\om y_2} \big( \wh{u}_{2, 0} (y_2) +\f{i\xi_1}{\lam}  \wh{b}_{2, 0}(y_2)\big)dy_2\f{1}{\om-|\xi_1|} e^{-|\xi_1|x_2}  d\lam:= \sum_{i=1}^{7} J_{1}^i.
\end{align*}
For the  $u_0$ part  of $J^1_1$, we see that
 \begin{align}\label{J11}
&\Big\|\int_{\Ga_1}\textbf{1}_{|\xi_1|\leq 2}e^{\lam t} \f{1}{\om-|\xi_1|} \int^\infty_0  e^{-\om y_2} \wh{u}_{2, 0}(y_2) dy_2 e^{-|\xi_1|x_2}  d\lam\Big\|_{L^2_{\xi_1}L^2_{x_2}}\non\\
&=\Big\|\int^\infty_{\wt{d}} e^{-\f{|\xi_1|^2}{|\xi|^2} t}\f{\xi_2}{\mp i\xi_2 -|\xi_1|} \int^\infty_0 e^{\pm  i\xi_2 y_2} \wh{u}_{2, 0}(y_2 ) dy_2\|e^{-|\xi_1| x_2}\|_{L^2_{x_2}} \f{2|\xi_1|^2}{|\xi|^4} d\xi_2 \Big\|_{L^2_{\xi_1}}\non\\
&\lesssim \Big\| e^{-\f{|\xi_1|^2}{|\xi|^2} t} |\xi_1|^{-\f12}\f{|\xi_1|^2}{|\xi|^4}\cF \big(\chi u_{2, 0} \big) \Big\|_{L^2_{\xi_1}L^1_{\xi_2}} \non\\
&\lesssim \Big\| (\textbf{1}_{|\xi|\leq 1}+\textbf{1}_{|\xi|\geq 1})e^{-\f{|\xi_1|^2}{|\xi|^2} t} \f{|\xi_1|^\f12}{|\xi|^2}\cF \big(\chi u_{2, 0} \big) \Big\|_{L^2_{\xi_1}L^1_{\xi_2}}.
\end{align}
The $\textbf{1}_{|\xi|\leq 1}$ part is estimated as
\begin{align*}
&\Big\|\textbf{1}_{|\xi|\leq 1}\f{|\xi_1|^\f12}{|\xi|^2} e^{-\f{|\xi_1|^2}{|\xi|^2} t} \cF \big(\chi u_{2, 0} \big)   \Big\|_{L^2_{\xi_1}L^1_{\xi_2}}\non\\
&\lesssim\Big\|\textbf{1}_{|\xi|\leq 1}\big(\f{|\xi_1|}{|\xi|}\big)^\f12 \big(\f{1}{|\xi|^\f32}\big) e^{-\f{|\xi_1|^2}{|\xi|^2} t} \cF \big(\chi u_{2, 0} \big)  \Big\|_{L^2_{\xi_1}L^1_{\xi_2}}\non\\
&\lesssim t^{-\f14}\Big\|\textbf{1}_{|\xi|\leq 1} \f{1}{|\xi|^\f32} e^{-\f{|\xi_1|^2}{|\xi|^2} t} \cF \big(\chi u_{2, 0} \big)   \Big\|_{L^2_{\xi_1}L^1_{\xi_2}}\non\\
&\lesssim t^{-\f14}\sum_{j\leq 1}\Big\|j^{-\f32} e^{-\f{|\xi_1|^2}{j^2} t} \cF\big( P_j (\chi u_{2, 0} )  \big)\Big\|_{L^2_{\xi_1}L^1_{\xi_2}}\non\\
&\lesssim t^{-\f14} \sum_{j\leq 1} j^{-\f32} \Big\|\textbf{1}_{|\xi_1|\leq 2} e^{-\f{|\xi_1|^2}{j^2} t}\
\Big\|_{L^2_{\xi_1}} \Big\|\cF\big( P_j (\chi u_{2, 0} )  \big)\Big\|_{L^\infty_{\xi_1}L^\infty_{\xi_2}}\big(\int_{|\xi_2|\leq j} d \xi_2\Big)\non\\
&\lesssim t^{-\f12} \sum_{j\leq 1} \| \chi u_{2, 0}  \|_{L^1}\non\\
&\lesssim t^{-\f12} \|u_{2, 0}\|_{L^1},
\end{align*}
and as for the $\textbf{1}_{|\xi|\geq 1}$ part,
\begin{align*}
&\Big\|\textbf{1}_{|\xi|\geq 1}\f{|\xi_1|^\f12}{|\xi|^2} e^{-\f{|\xi_1|^2}{|\xi|^2} t}  \cF \big(\chi u_{2, 0} \big)   \Big\|_{L^2_{\xi_1}L^1_{\xi_2}}\non\\
&\lesssim t^{-\f14}\Big\|\textbf{1}_{|\xi|\geq 1} \f{1}{|\xi|^\f32}e^{-\f{|\xi_1|^2}{|\xi|^2} t}  \cF \big(\chi u_{2, 0} \big)   \Big\|_{L^2_{\xi_1}L^1_{\xi_2}}\non\\
&\lesssim t^{-\f14}\sum_{j\geq 1}\Big\|j^{-\f32} e^{-\f{|\xi_1|^2}{j^2} t} \cF\big( P_j (\chi u_{2, 0} )  \big)\Big\|_{L^2_{\xi_1}L^1_{\xi_2}}\non\\
&\lesssim t^{-\f14} \sum_{j\geq 1} j^{-\f32} \Big\| e^{-\f{|\xi_1|^2}{j^2} t}\
\Big\|_{L^2_{\xi_1}} \Big\|\cF\big( P_j (\chi u_{2, 0} )  \big)\Big\|_{L^\infty_{\xi_1}L^2_{\xi_2}}\big(\int_{|\xi_2|\leq j} d \xi_2\Big)^\f12\non\\
&\lesssim t^{-\f12} \sum_{j\geq 1} j^{-\f12}\| \chi u_{2, 0}  \|_{L^1_{x_1}L^2_{x_2}}\non\\
&\lesssim t^{-\f12} \|u_{2, 0}\|_{L^1_{x_1}L^2_{x_2}}.
\end{align*}

The $L^\infty$ estimate is very similar as \eqref{J2infty}
 \begin{align*}
&\Big\|\int_{\Ga_1}\textbf{1}_{|\xi_1|\leq 2}e^{\lam t} \f{1}{\om-|\xi_1|} \int^\infty_0  e^{-\om y_2} \wh{u}_{2, 0}(y_2) dy_2 e^{-|\xi_1|x_2}  d\lam\Big\|_{L^1_{\xi_1}L^\infty_{x_2}}\non\\
&=\Big\|\int^\infty_{\wt{d}} e^{-\f{|\xi_1|^2}{|\xi|^2} t}\f{\xi_2}{\mp i\xi_2 -|\xi_1|} \int^\infty_0 e^{\pm  i\xi_2 y_2} \wh{u}_{2, 0}(y_2 ) dy_2\|e^{-|\xi_1| x_2}\|_{L^\infty_{x_2}} \f{2|\xi_1|^2}{|\xi|^4} d\xi_2 \Big\|_{L^1_{\xi_1}}\non\\
&\lesssim \Big\| e^{-\f{|\xi_1|^2}{|\xi|^2} t} \f{|\xi_1|^2}{|\xi|^4}\cF \big(\chi u_{2, 0} \big) \Big\|_{L^1_{\xi_1}L^1_{\xi_2}} \non\\
&\lesssim \Big\| e^{-\f{|\xi_1|^2}{|\xi|^2} t} \f{|\xi_1|}{|\xi|^2}\cF \big(\chi u_{2, 0} \big) \Big\|_{L^1_{\xi_1}L^1_{\xi_2}}.
\end{align*}

For the $L^2$ estimate $\wh{u}_{2,0}$ part  of $J_1^2+J_1^3$, by $u_{2,0}|_{x_2=0}=0$, and divergence free condition, we have
\begin{align*}
&\Big\| \int^{\wt{d}}_0\f{ \lam_- e^{\lam_- t}- \lam_+ e^{\lam_+t} }{\lam_+-\lam_-}\f{\xi_2}{i\xi_2-|\xi_1|} \int^\infty_0 e^{-i\xi_2y_2}   \wh{u}_{2, 0} (y_2) dy_2 e^{-|\xi_1|x_2 } d\xi_2\Big\|_{L^2_{\xi_1}L^2_{x_2}}\\
=&\Big\| \int^{\wt{d}}_0\f{ \lam_- e^{\lam_- t}- \lam_+ e^{\lam_+t} }{\lam_+-\lam_-}\f{\xi_2}{i\xi_2-|\xi_1|} \f{1}{i\xi_2}\int^\infty_0 e^{-i\xi_2y_2}   \pa_2\wh{u}_{2, 0} (y_2) dy_2 e^{-|\xi_1|x_2 } d\xi_2\Big\|_{L^2_{\xi_1}L^2_{x_2}}\\
=&\Big\| \int_{\R} (1-\vphi(\xi_2))\chi(\xi_2)\f{ \lam_- e^{\lam_- t}- \lam_+ e^{\lam_+t} }{\lam_+-\lam_-}\f{\xi_2}{i\xi_2-|\xi_1|}  \|e^{-|\xi_1|x_2 } \f{1}{i\xi_2}\|_{L^2_{x_2}}\cF_{y_2}\Big(\chi(y_2) \pa_2 \wh{u}_{2, 0} (y_2)  \Big)(\xi_2) d\xi_2\Big\|_{L^2_{\xi_1}}\\
\lesssim&\Big\| \f{ \lam_- e^{\lam_- t}- \lam_+ e^{\lam_+t} }{\lam_+-\lam_-}\f{\xi_2}{i\xi_2-|\xi_1|}|\xi_1|^{-\f12} \f{1}{i\xi_2}\cF_{y_2}\Big(\chi(y_2)  \pa_2\wh{u}_{2, 0} (y_2)  \Big)(\xi_2)\Big\|_{L^2_{\xi_1}L^1_{\xi_2}}\\
\lesssim&\Big\| \f{\xi_2}{|\xi_1|^{\f12}|\xi|}e^{-\f{|\xi|^2}{4}t}\f{1}{i\xi_2}\cF_{y_2}\Big(\chi(y_2)  \pa_2\wh{u}_{2, 0} (y_2) \Big)(\xi_2) \Big\|_{L^2_{\xi_1}L^1_{\xi_2}}\\
=&\Big\| \f{\xi_2}{|\xi_1|^{\f12}|\xi|}e^{-\f{|\xi|^2}{4}t}\f{1}{i\xi_2}\cF_{y_2}\Big(\chi(y_2)  \wh{\pa_1 u}_{1, 0} (y_2) \Big)(\xi_2) \Big\|_{L^2_{\xi_1}L^1_{\xi_2}}\\
\lesssim&\Big\|\f{\xi_2}{|\xi_1|^{\f12}|\xi|}e^{-\f{|\xi|^2}{4}t}\f{\xi_1}{\xi_2}\cF_{y_2}\Big(\chi(y_2)  \wh{u}_{1, 0} (y_2) \Big)(\xi_2)\Big \|_{L^2_{\xi_1}L^1_{\xi_2}}\\
\lesssim&\Big\|\f{|\xi_1|^{\f12}}{|\xi|}e^{-\f{|\xi|^2}{4}t}\cF_{y_2}\Big(\chi(y_2)  \wh{u}_{1, 0} (y_2) \Big)(\xi_2) \Big\|_{L^2_{\xi_1}L^1_{\xi_2}}\\
\lesssim&\Big\|e^{-\f{|\xi|^2}{4}t}\cF_{y_2}\Big(\chi(y_2)  \wh{u}_{1, 0} (y_2) \Big)(\xi_2) \Big\|_{L^2_{\xi_1}L^2_{\xi_2}}\Big\|\f{|\xi_1|^{\f12}}{|\xi|}\Big\|_{L^\infty_{\xi_1}L^2_{\xi_2}}\\
\lesssim &  \Big\|e^{-\f{|\xi|^2}{4} t}   \Big\|_{L^2_{\xi_1}L^2_{\xi_2}}\Big\|\cF \big(\chi u_{1,0} \big)\Big\|_{L^\infty_{\xi_1}L^\infty_{\xi_2}}\non\\
\lesssim& t  ^{-\f12} \|u_{1,0}\|_{L^1},
\end{align*}
where we use
\begin{align*}
\Big\|\f{|\xi_1|^{\f12}}{|\xi|}\Big\|_{L^\infty_{\xi_1}L^2_{\xi_2}}=\Big(\int^\infty_0 \f{1}{1+\zeta^2} d\zeta\Big)^\f12=\Big( \arctan \zeta \Big|_0^{\infty} \Big)^\f12\lesssim 1. \quad (\zeta=\f{\xi_2}{|\xi_1|})
\end{align*}
And the $L^\infty$ estimate is obtained as
\begin{align*}
&\Big\|\int^{\wt{d}}_0\f{ \lam_- e^{\lam_- t}- \lam_+ e^{\lam_+t} }{\lam_+-\lam_-}\f{\xi_2}{i\xi_2-|\xi_1|} \int^\infty_0 e^{-i\xi_2y_2}   \wh{u}_{2, 0} (y_2) dy_2 e^{-|\xi_1|x_2} d\xi_2 \Big\|_{L^1_{\xi_1}L^\infty_{x_2}}\\
&\lesssim\Big\|\f{ \lam_- e^{\lam_- t}- \lam_+ e^{\lam_+t} }{\lam_+-\lam_-}
\cF_{x_2}\big(\wh{\chi u}_{2,0}\big)  \Big\|_{L^1_{\xi_1}L^1_{\xi_2}}\\
&\lesssim \|e^{-\f{|\xi|^2}{4}t} \|_{L^1_{\xi_1}L^1_{\xi_2}}\|\cF \big(\chi u_{2,0}\big) \|_{L^\infty_{\xi_1}L^\infty_{\xi_2}}\\
&\lesssim   t ^{-1} \|u_{2,0}\|_{L^1}.
\end{align*}
The  $J_1^4$ part can be controlled by heat kernel, and  $J_1^{5, 6,7}\to 0$ by $|e^{-|\xi_1|x_2}|\lesssim 1$.

\smallskip

{\bf Case 2.  $|\xi_1|> 2$.}

We first consider $J_2$ and divide it as
\beno J_2=-\f{i\xi_1}{|\xi_1|} \f{1}{2\pi i}\sum_{i=1}^{2}\int_{\wt{\Ga}_i}e^{\lam t} \int^\infty_0  e^{-\om (x_2+y_2)} (\wh{u}_{2, 0}+\f{i\xi_1}{\lam} \wh{b}_{2, 0})(y_2) dy_2 \f{1}{\om-|\xi_1|}  d\lam:= \sum_{i=1}^{2} \wt{J}_2^i.
\eeno

For the $\wh{u}_{2,0}$ part of $\wt{J}_2^1$, by the definition of $\wt{\Ga}_1: \lam= -\eta \ (\eta: 0 \to -\lam'_+)$, we have
\begin{align}\label{J21'}
&-\f{1}{2\pi i} \int_{\wt{\Ga}_1^{(+)}\cup\wt{\Ga}_1 ^{(-)} } \f{ e^{\lam t} }{\om-|\xi_1|} \int^\infty_0 e^{-\om(x_2+y_2)}  \wh{u}_{2, 0} (y_2) dy_2 d\lam\non\\
&=-\f{1}{2\pi i} \Big(\int^{-\lam_+'}_{0}\f{ e^{-\eta t} }{-i|\om|-|\xi_1|} \int^\infty_0 e^{i|\om|(x_2+y_2)}   \wh{u}_{2, 0} (y_2)  dy_2 d\eta\non\\
&\qquad-\int^{-\lam_+'}_{0}\f{ e^{-\eta t} }{i|\om|-|\xi_1|} \int^\infty_0 e^{-i|\om|(x_2+y_2)}   \wh{u}_{2, 0} (y_2) dy_2 d\eta\Big)
\end{align}
with
\begin{align*}
\om=\sqrt{\lam+|\xi_1|^2+\f{|\xi_1|^2}{\lam}}=\mp i \sqrt{\eta-|\xi_1|^2+\f{|\xi_1|^2}{\eta}}=\mp i|\om|.
\end{align*}

We now make the change of variables $\xi_2=|\om|$. Here we should take the branch $\eta=\eta_-$, since $\eta\in (0, -\lam'_+)$, then
\beno
d\eta=\xi_2\Big(1-\f{|\xi|^2}{\sqrt{|\xi|^4-4|\xi_1|^2}}\Big) d\xi_2 =2 \xi_2 \f{\lam_+}{\lam_+-\lam_-} d\xi_2.
\eeno
The most trouble case is
\beno
\eta\sim\f{|\xi_1|^2}{|\xi|^2} \quad (\text{when} \quad \f{4|\xi_1|^2}{|\xi|^4}\leq \f12),
\eeno
and the other cases can be controlled by heat kernel. For this case,
\beno
d \eta=2 \xi_2 \f{\lam_+}{\lam_+-\lam_-} d\xi_2, \quad \Big|\f{\lam_{+}}{\lam_{+}-\lam_{-}}\Big|\sim \f{|\xi_1|^2}{|\xi|^4},
\eeno
so the estimate of $\wt{J}_2^1$ is very similar as the estimate of $J^1_2$, we omit the details.

For the $\wh{u}_{2,0}$ part of $\wt{J}_2^2$, by the definition of $\wt{\Ga}_2: \lam= -\eta \ (\eta:   -\lam'_- \to \infty)$, we have
\begin{align}\label{J22}
&-\f{1}{2\pi i} \int_{\wt{\Ga}_2^{(+)}\cup \wt{\Ga}_2 ^{(-)} } \f{ e^{\lam t} }{\om-|\xi_1|} \int^\infty_0 e^{-\om(x_2+y_2)}  \wh{u}_{2, 0} (y_2) dy_2 d\lam\non\\
&=-\f{1}{2\pi i} \Big(\int^\infty_{-\lam'_-}\f{ e^{-\eta t} }{i|\om|-|\xi_1|} \int^\infty_0 e^{-i|\om|(x_2+y_2)}   \wh{u}_{2, 0} (y_2)  dy_2 d\eta\non\\
&\qquad-\int^\infty_{-\lam'_-}\f{ e^{-\eta t} }{-i|\om|-|\xi_1|} \int^\infty_0 e^{i|\om|(x_2+y_2)}   \wh{u}_{2, 0} (y_2) dy_2 d\eta\Big)
\end{align}
with
\begin{align*}
\om=\sqrt{\lam+|\xi_1|^2+\f{|\xi_1|^2}{\lam}}=\pm i \sqrt{\eta-|\xi_1|^2+\f{|\xi_1|^2}{\eta}}=\pm i|\om|.
\end{align*}

We now make the change of variables $\xi_2=|\om|$. Here we should take the branch $\eta=\eta_+$, since $\eta\in (-\lam'_-, \infty)$,
where
\beno
\eta_+=\f{|\xi|^2}{2}+\f{|\xi|^2}{2} \sqrt{1-\f{4|\xi_1|^2}{|\xi|^4}} \sim|\xi|^2.
\eeno
It then follows that
\beno
d\eta=\xi_2\Big(1+\f{|\xi|^2}{\sqrt{|\xi|^4-4|\xi_1|^2}}\Big) d\xi_2 =-2 \xi_2 \f{\lam_-}{\lam_+-\lam_-} d\xi_2,\quad \Big|\f{\lam_-}{\lam_+-\lam_-}\Big|\sim 1.
\eeno
The first term on the right hand side of \eqref{J22} can be controlled by
\begin{align*}
&\Big|-\f{1}{2\pi i} \int^\infty_{-\lam'_-}\f{ e^{-\eta t} }{i|\om|-|\xi_1|} \int^\infty_0 e^{-i|\om|(x_2+y_2)}   \wh{u}_{2, 0} (y_2)  dy_2 d\eta\non\Big|\\
&\lesssim \Big|-\f{1}{2\pi i} \int^\infty_0 e^{-4 t}\f{2\xi_2 }{i\xi_2-|\xi_1|} \int^\infty_0 e^{-i\xi_2(x_2+y_2)}   \wh{u}_{2, 0} (y_2)  dy_2 d\xi_2\non\Big|\\
&= \Big|e^{-4t} 2i\cF_{x_2}^{-1}\Big(\chi(\xi_2) \f{\xi_2 }{i\xi_2-|\xi_1|}\cF_{y_2}\big(\chi  \wh{u}_{2, 0} \big) (\xi_2) \Big)(-x_2)\Big|,
\end{align*}
this term has an exponential decay.

For the $\wh{u}_{2,0}$ part of $\wt{J}_1^1$, the $L^2$ estimate is very similar as \eqref{J11}
 \begin{align*}
&\Big\|\int_{\wt{\Ga}_1}\textbf{1}_{|\xi_1|\geq 2}e^{\lam t} \f{1}{\om-|\xi_1|} \int^\infty_0  e^{-\om y_2} \wh{u}_{2, 0}(y_2) dy_2 e^{-|\xi_1|x_2}  d\lam\Big\|_{L^2_{\xi_1}L^2_{x_2}}\non\\
&=\Big\|\int_{\R}\textbf{1}_{|\xi_1|\geq 2}\chi(\xi_2) e^{-\f{|\xi_1|^2}{|\xi|^2} t}\f{\xi_2}{\mp i\xi_2 -|\xi_1|} \int^\infty_0 e^{\pm  i\xi_2 y_2} \wh{u}_{2, 0}(y_2 ) dy_2\|e^{-|\xi_1| x_2}\|_{L^2_{x_2}} \f{2|\xi_1|^2}{|\xi|^4} d\xi_2 \Big\|_{L^2_{\xi_1}}\non\\
&\lesssim \Big\| e^{-\f{|\xi_1|2}{|\xi|^2} t} \f{|\xi_1|^\f32}{|\xi|^4}\cF \big(\chi u_{2, 0} \big) \Big\|_{L^2_{\xi_1}L^1_{\xi_2}}\non\\
&\lesssim \Big\| e^{-\f{|\xi_1|2}{|\xi|^2} t} \f{|\xi_1|^\f12}{|\xi|^2}\cF \big(\chi u_{2, 0} \big) \Big\|_{L^2_{\xi_1}L^1_{\xi_2}},
\end{align*}
and the $L^\infty$ estimate is very similar as \eqref{J2infty}
 \begin{align*}
&\Big\|\int_{\wt{\Ga}_1}\textbf{1}_{|\xi_1|\geq 2}e^{\lam t} \f{1}{\om-|\xi_1|} \int^\infty_0  e^{-\om y_2} \wh{u}_{2, 0}(y_2) dy_2 e^{-|\xi_1|x_2}  d\lam\Big\|_{L^1_{\xi_1}L^\infty_{x_2}}\non\\
&=\Big\|\int^\infty_{0}\int_{\wt{\Ga}_1}\textbf{1}_{|\xi_1|\geq 2} e^{-\f{|\xi_1|^2}{|\xi|^2} t}\f{\xi_2}{\mp i\xi_2 -|\xi_1|} \int^\infty_0 e^{\pm  i\xi_2 y_2} \wh{u}_{2, 0}(y_2 ) dy_2\|e^{-|\xi_1| x_2}\|_{L^\infty_{x_2}} \f{2|\xi_1|^2}{|\xi|^4} d\xi_2 \Big\|_{L^1_{\xi_1}}\non\\
&\lesssim \Big\| e^{-\f{|\xi_1|^2}{|\xi|^2} t} \f{|\xi_1|^2}{|\xi|^4}\cF \big(\chi u_{2, 0} \big) \Big\|_{L^1_{\xi_1}L^1_{\xi_2}} \non\\
&\lesssim \Big\| e^{-\f{|\xi_1|^2}{|\xi|^2} t} \f{|\xi_1|}{|\xi|^2}\cF \big(\chi u_{2, 0} \big) \Big\|_{L^1_{\xi_1}L^1_{\xi_2}}.
\end{align*}

For the $\wh{u}_{2,0}$ part of $\wt{J}_1^2$, we have
 \begin{align*}
&\Big\|\int_{\wt{\Ga}_2}\textbf{1}_{|\xi_1|\geq 2}e^{\lam t} \f{1}{\om-|\xi_1|} \int^\infty_0  e^{-\om y_2} \wh{u}_{2, 0}(y_2) dy_2 e^{-|\xi_1|x_2}  d\lam\Big\|_{L^2_{\xi_1}L^2_{x_2}}\non\\
&=\Big\| \int_{\wt{\Ga}_2}\textbf{1}_{|\xi_1|\geq 2}e^{-4t} 2i\chi(\xi_2) \f{\xi_2 }{i\xi_2-|\xi_1|}|\xi_1|^{-\f12}\cF_{y_2}\big(\chi  \wh{u}_{2, 0} \big) (\xi_2)  \Big\|_{L^2_{\xi_1}L^1_{\xi_2}}\non\\
&\lesssim e^{-4t}\Big\|  |\xi_1|^{-\f12} \f{|\xi_1| }{|\xi|}\cF_{y_2}\big(\chi  \wh{u}_{1, 0} \big) (\xi_2)  \Big\|_{L^2_{\xi_1}L^1_{\xi_2}}\non\\
&\lesssim e^{-4t}\Big\|  \f{|\xi_1|^\f12 }{|\xi|}\Big\|_{L^\infty_{\xi_1}L^2_{\xi_2}}\Big\|\cF_{y_2}\big(\chi  \wh{u}_{1, 0} \big) (\xi_2)  \Big\|_{L^2_{\xi_1}L^2_{\xi_2}}\non\\
&\lesssim e^{-4t}\| u_{1, 0}  \|_{L^2},
\end{align*}
and
 \begin{align*}
&\Big\|\int_{\wt{\Ga}_2}\textbf{1}_{|\xi_1|\geq 2}e^{\lam t} \f{1}{\om-|\xi_1|} \int^\infty_0  e^{-\om y_2} \wh{u}_{2, 0}(y_2) dy_2 e^{-|\xi_1|x_2}  d\lam\Big\|_{L^1_{\xi_1}L^\infty_{x_2}}\non\\
&=\Big\| e^{-\f{|\xi|^2}{2}t} 2i\chi(\xi_2) \f{\xi_2 }{i\xi_2-|\xi_1|}\cF_{y_2}\big(\chi  \wh{u}_{2, 0} \big) (\xi_2)  \Big\|_{L^1_{\xi_1}L^1_{\xi_2}}\non\\
&\lesssim \Big\| e^{-\f{|\xi_1|^2}{2}t}  \f{|\xi_1| }{|\xi|}\cF_{y_2}\big(\chi  \wh{u}_{1, 0} \big) (\xi_2)  \Big\|_{L^1_{\xi_1}L^1_{\xi_2}}\non\\
&\lesssim \Big\|  \f{|\xi_1|^\f12 }{|\xi|}\Big\|_{L^\infty_{\xi_1}L^2_{\xi_2}}\Big\|  \f{1 }{|\xi_1|^{\f12+\delta}}\Big\|_{L^2_{\xi_1}}\Big\||\xi_1|^{1+\delta} e^{-\f{|\xi_1|^2}{2}t}\cF_{y_2}\big(\chi  \wh{u}_{1, 0} \big) (\xi_2)  \Big\|_{L^2_{\xi_1}L^2_{\xi_2}}\non\\
&\lesssim e^{-2t}t^{-(\f12+\delta)}\| u_{1, 0}  \|_{L^2}.
\end{align*}
 Thus we complete the proof of Proposition \ref{linear u}.
\end{proof}

\smallskip

Now we  give the  resolvent estimate of magnetic  field in  \eqref{linear}.

\begin{proposition}\label{linear b}
We have the following $L^2$, $L^\infty$ estimates
 \begin{align*}
  \|b_{L,1}\|_{L^2}&\lesssim \lan t\ran^{-\f14}\Big(\|(u_0, b_0)\|_{L^1\cap L^2}+\|\na^{\f12+\delta} b_0\|_{L^1_{x_1}L^2_{x_2}}\Big),\non\\
  \|b_{L,2}\|_{L^2}&\lesssim \lan t\ran^{-\f12}\Big(\|(u_0, b_0)\|_{L^1\cap L^2}+\|\na^{\f12+\delta} b_0\|_{L^1_{x_1}L^2_{x_2}}\Big),\non\\
  \|b_{L, 1}\|_{L^\infty}&\lesssim\lan t\ran^{-\f12}\Big(\|(u_0, b_0)\|_{L^1\cap L^2}+\Big\| \na^{\f32+\delta} (u_0, b_0) \Big \|_{L^1_{x_1}L^2_{x_2}}\Big),\non\\
   \|b_{L, 2}\|_{L^\infty}&\lesssim\lan t\ran^{-1}\Big(\|(u_0, b_0)\|_{L^1\cap L^2}+\Big\| \na^{\f32+\delta} (u_0, b_0) \Big \|_{L^1_{x_1}L^2_{x_2}}\Big)
 \end{align*}
 for the magnetic  field in \eqref{linear} (here $\delta>0$ small enough).
\end{proposition}

 \begin{proof}Similarly to \eqref{u1}, we rewrite the linear part of magnetic field as
\begin{align}\label{b1}
  \wh{b}_{1, L}(\xi_1, x_2,t)&=\f{1}{2\pi i}\int_{\Ga}e^{\lam t} \f{i\xi_1}{\lam} \f{1}{2\om} \int^\infty_0  \Big(e^{-\om|x_2-y_2|} -e^{-\om(x_2+y_2)} \Big)(\wh{u}_{1, 0}+\f{i\xi_1}{\lam} \wh{b}_{1, 0})(y_2) dy_2 d\lam\non\\
&\qquad- \f{1}{2\pi i}\int_{\Ga}e^{\lam t} \f{|\xi_1|}{\lam} \f{1}{\om-|\xi_1|} \int^\infty_0  e^{-\om y_2} (\wh{u}_{2, 0}+\f{i\xi_1}{\lam} \wh{b}_{2, 0})(y_2) dy_2 \Big( e^{-|\xi_1|x_2}-e^{-\om x_2}\Big)  d\lam+\wh{b}_{1, 0},
\end{align}
and
\begin{align}\label{b2}
\wh{b}_{2, L}(\xi_1,x_2, t)
&=\f{1}{2\pi i}\int_{\Ga}e^{\lam t}
\f{i\xi_1}{\lam}\Big\{  -E_{\om}[ \wh{u}_{2, 0}+\f{i\xi_1}{\lam}  \wh{b}_{2, 0}]_0 e^{-\om x_2}+E_{\om}[ \wh{u}_{2, 0}+\f{i\xi_1}{\lam}  \wh{b}_{2, 0}]\non\\
&\quad+E_{\om}[ \wh{u}_{2, 0}+\f{i\xi_1}{\lam}  \wh{b}_{2, 0}]_0 \Big(e^{-\om x_2}
-\f{E_{\om} [ e^{-|\xi_1|x_2}]}{E_{\om}[e^{-|\xi_1|x_2}]_0}\Big)\Big\}d\lam+\wh{b}_{2, 0}\non\\
&=\f{1}{2\pi i}\int_{\Ga}e^{\lam t} \f{i\xi_1}{\lam} \Big\{\f{1}{2\om} \int^\infty_0  \Big(e^{-\om|x_2-y_2|} -e^{-\om(x_2+y_2)} \Big)(\wh{u}_{2, 0}+\f{i\xi_1}{\lam} \wh{b}_{2, 0})(y_2) dy_2 \non\\
&\qquad- \f{1}{2\pi i}\int_{\Ga}e^{\lam t} \f{1}{\om-|\xi_1|} \int^\infty_0  e^{-\om y_2} (\wh{u}_{2, 0}+\f{i\xi_1}{\lam} \wh{b}_{2, 0})(y_2) dy_2 \Big( e^{-|\xi_1|x_2}-e^{-\om x_2}\Big)  \Big\}d\lam+\wh{b}_{2, 0}\non\\
&:=K_1+K_2+\wh{b}_{2, 0}.
\end{align}
We here only consider $\wh{b}_{2, L}(\xi_1,x_2, t)$, and the $\wh{b}_{1, L}(\xi_1,x_2, t)$ can be estimated similarly.

For the $K_1$ part, we use the odd extension $\wh{U}_{2, 0}(y_2)$ and $\wt{K}_{1, \lam}$ to have
\begin{align*}
\cF_{x_2}(\wt{K}_1)
&=\f{1}{2\pi i}\int_{\Ga}\f{i\xi_1e^{\lam t}}{(\lam-\lam_+)(\lam-\lam_-)}\cF(U_{2, 0}) d\lam\non\\
&=i\xi_1\f{ e^{\lam_+t} -e^{\lam_- t}}{\lam_+-\lam_-}\cF(U_{2, 0})+(i\xi_1)^2\f{ 1}{\lam_+\lam_-}\cF(\wt{\wh{b}}_{2, 0})\\
&=i\xi_1\f{ e^{\lam_+t} -e^{\lam_- t}}{\lam_+-\lam_-}\cF(U_{2, 0})-\cF(\wt{\wh{b}}_{2, 0}),
\end{align*}
where we use  $\lam_+\lam_-=|\xi_1|^2$, then
\begin{align*}
\cF_{x_2}(\wt{K}_1+\wt{\wh{b}}_{2, 0})=i\xi_1\f{ e^{\lam_+t} -e^{\lam_- t}}{\lam_+-\lam_-}\cF(U_{2, 0}).
\end{align*}
The decay of $K_1+\wh{b}_{2, 0}$ can be obtained by using the result of  whole space \cite{W}.

 $K_2$  can be rewritten as
 \begin{align*}
K_2=& \f{1}{2\pi i}\int_{\Ga}\f{e^{\lam t}}{\om-|\xi_1|} \f{i\xi_1}{\lam} \int^\infty_0  e^{-\om y_2} (\wh{u}_{2, 0}+\f{i\xi_1}{\lam} \wh{b}_{2, 0})(y_2) dy_2 e^{-|\xi_1|x_2}  d\lam\\
&- \f{1}{2\pi i}\int_{\Ga}\f{e^{\lam t}}{\om-|\xi_1|} \f{i\xi_1}{\lam} \int^\infty_0  e^{-\om (x_2+y_2)} (\wh{u}_{2, 0}+\f{i\xi_1}{\lam} \wh{b}_{2, 0})(y_2) dy_2 d\lam\\
:=&L_{1}+L_{2}.
\end{align*}

 {\bf Case 1.  $|\xi_1|\leq 2$.}

We first consider $L_2$ and divide it as
\beno L_2=- \f{1}{2\pi i}\sum_{i=1}^{7}\int_{\Ga_i}\f{e^{\lam t} }{\om-|\xi_1|}\f{i\xi_1}{\lam}\int^\infty_0  e^{-\om (x_2+y_2)} (\wh{u}_{2, 0}+\f{i\xi_1}{\lam} \wh{b}_{2, 0})(y_2) dy_2   d\lam:= \sum_{i=1}^{7} L_2^i.
\eeno
Indeed, the trouble term is the $\wh{b}_{2, 0}$ part.

For the $\wh{b}_{2,0}$ part of $L_2^1$, since
\beno
\f{|\xi_1|}{|\lam|}\sim\f{|\xi|^2}{|\xi_1|},
\eeno
then by the divergence free condition, and the boundary condition $b_{2,0}|_{x_2=0}=0$ and $b_{1,0}|_{x_2=0}=0$, the $L^2$ estimate is obtained as
\begin{align*}
\|L_2^1\|_{L^2_{\xi_1}{L^2_{x_2}}}
&\lesssim\Big\|\textbf{1}_{|\xi_1|\leq 2}\f{|\xi_1|^2}{|\xi|^4} e^{-\f{|\xi_1|^2}{|\xi|^2} t}\f{|\xi|^4}{|\xi_1|^2}\f{\xi_2 }{i\xi_2+|\xi_1|}\f{1}{\xi_2}\cF\big(\chi \pa_2b_{2, 0} \big) (-\xi_2)   \Big\|_{L^2_{\xi_1}L^2_{\xi_2}}\non\\
&\lesssim  \Big\|\f{|\xi_1|}{|\xi|} e^{-\f{|\xi_1|^2}{|\xi|^2} t}\cF\big(\chi b_{1, 0} \big) (-\xi_2)   \Big\|_{L^2_{\xi_1}L^2_{\xi_2}}\non\\
&\lesssim \sum_{j\geq 0}\Big\| \f{|\xi_1|}{j}  e^{-\f{|\xi_1|^2}{j^2} t} \cF( P_j (\chi b_{1,0})  )\Big\|_{L^2_{\xi_1}L^2_{\xi_2}}\non\\
&\lesssim  t^{-\f12} \sum_{j\geq 0} \Big\|e^{-\f{|\xi_1|^2}{j^2} t}\Big\|_{L^2_{\xi_1}} \Big\|\cF(P_j (\chi b_{1,0}))\Big\|_{L^\infty_{\xi_1}L^2_{\xi_2}}\non\\
&\lesssim t^{-\f34} \Big(\sum_{0\leq j\leq 1} j^{\f12} \Big\| \cF( P_j  (\chi b_{1,0})  )\Big\|_{L^\infty_{\xi_1}L^2_{\xi_2}}+\sum_{j\geq 1} j^{-\delta} \Big\|j^{\f12+\delta} \cF( P_j  (\chi b_{1,0} ) )\Big\|_{L^\infty_{\xi_1}L^2_{\xi_2}}\Big)\non\\
&\lesssim t^{-\f12}\Big( \|b_{1,0}\|_{L^1_{x_1}L^2_{x_2}}+\|\na^{\f12+\delta} b_{1,0}\|_{L^1_{x_1}L^2_{x_2}}\Big),
\end{align*}
and the $L^\infty$ estimate is obtained as
\begin{align}\label{L12infty}
&  \Big\| \textbf{1}_{|\xi_1|\leq 2}e^{-\f{|\xi_1|^2}{|\xi|^2} t}\cF\big(\chi b_{2, 0} \big) (-\xi_2)   \Big\|
_{L^1_{\xi_1}L^1_{\xi_2}}\non\\
& \lesssim\Big\| \f{1}{|\xi|} e^{-\f{|\xi_1|^2}{|\xi|^2} t}  \cF(\chi \na  b_{2, 0})  \Big\|_{L^1_{\xi_1}L^1_{\xi_2}}\non\\
&\lesssim \sum_{j\geq 0}\Big\|\f{|\xi_1|}{|\xi|}  e^{-\f{|\xi_1|^2}{j^2} t} \cF( P_j  (\chi b_{ 0} ) )\Big\|_{L^1_{\xi_1}L^1_{\xi_2}}\non\\
&\lesssim  t^{-\f12} \sum_{j\geq 0} \Big\| e^{-\f{|\xi_1|^2}{j^2} t}\Big\|_{L^1_{\xi_1}} \Big\|\cF(P_j  (\chi b_{ 0} ))\Big\|_{L^\infty_{\xi_1}L^2_{\xi_2}} \Big(\int_{|\xi_2|\leq j} d \xi_2\Big)^\f12\non\\
&\lesssim t^{-1} \Big(\sum_{0\leq j\leq 1} j^{\f32} \Big\| \cF( P_j   (\chi b_{ 0} ) )\Big\|_{L^\infty_{\xi_1}L^2_{\xi_2}}+\sum_{j\geq 1} j^{-\delta} \Big\|j^{\f32+\delta} \cF( P_j   (\chi b_{ 0} ) )\Big\|_{L^\infty_{\xi_1}L^2_{\xi_2}}\Big)\non\\
&\lesssim t^{-1}\Big( \|b_0\|_{L^1_{x_1}L^2_{x_2}}+\|\na^{\f32+\delta} b_0\|_{L^1_{x_1}L^2_{x_2}}\Big).
\end{align}

For the $\wh{b}_{2,0}$ part of $L_2^{2,3,4}$, since
\beno
 \f{|\xi_1|}{|\lam|}=\f{|\xi_1|}{|\lam_\pm|}\sim 1, \quad \text{when} \quad |\xi_1|\geq \f{|\xi|^2}{2},
 \eeno
then
\beno
\|L_2^{2,3,4}\|_{L^2_{\xi_1}L^2_{x_2}}\lesssim \| e^{-\f{|\xi|^2}{4}t} \|_{L^2_{\xi_1}L^2_{\xi_2}}\|b_{2,0}\|_{L^1}\lesssim \lan t \ran ^{-\f12} \|b_{2,0}\|_{L^1},
\eeno
and
\beno
\|L_2^{2,3,4}\|_{L^1_{\xi_1}L^\infty_{x_2}}\lesssim \| e^{-\f{|\xi|^2}{4}t} \|_{L^1_{\xi_1}L^1_{\xi_2}}\|b_{2,0}\|_{L^1}\lesssim \lan t \ran ^{-1} \|b_{2,0}\|_{L^1}.
\eeno
The  $\wh{b}_{1,0}$ part  of $L_2^6$  goes to 0, since
\beno
\f{|\xi_1|}{|\lam|}=\f{|\xi_1|}{(|\xi_1|^2+\e^2-2\e|\xi_1|\cos \gamma)^\f12}.
\eeno
The  $\wh{b}_{1,0}$ part  of $L_2^7$  goes to 0, since
\beno
\f{|\xi_1|}{|\lam|}=\f{|\xi_1|}{R} \to 0.
\eeno
Now we consider $L_{1}$ and divide  it as
\begin{align*}
 L_{1}&=\f{1}{2\pi i} \sum_{i=1}^{7} \int_{\Gamma_i} \f{e^{\lam t} }{\om-|\xi_1|} \f{i\xi_1}{\lam}  \int^\infty_0e^{-\om y_2} \big( \wh{u}_{2, 0} (y_2) +\f{i\xi_1}{\lam}  \wh{b}_{2, 0}(y_2)\big)dy_2e^{-|\xi_1|x_2}  d\lam:= \sum_{i=1}^{7} L_{1}^i.
\end{align*}
For the  $b_0$ part  of $L^1_1$, by the divergence free condition, the boundary conditions $b_{2,0}|_{x_2=0}=0$ and $b_{1,0}|_{x_2=0}=0$,  and \eqref{J11}, we have
\begin{align*}
&\Big\|\int_{\Ga_1}\textbf{1}_{|\xi_1|\leq 2}e^{\lam t} \f{1}{\om-|\xi_1|} \int^\infty_0  e^{-\om y_2} \big(\f{i\xi_1}{\lam}\big)^2\wh{b}_{2, 0}(y_2) dy_2 e^{-|\xi_1|x_2}  d\lam\Big\|_{L^2_{\xi_1}L^2_{x_2}}\non\\
&\lesssim \Big\| e^{-\f{|\xi_1|^2}{|\xi|^2} t} |\xi_1|^{-\f12}\f{1}{\xi_2}\cF \big(\chi \pa_2b_{2, 0} \big) \Big\|_{L^2_{\xi_1}L^1_{\xi_2}}\non\\
&\lesssim \Big\| e^{-\f{|\xi_1|^2}{|\xi|^2} t} \f{|\xi_1|^{\f12}}{|\xi|}\cF \big(\chi b_{1,0} \big) \Big\|_{L^2_{\xi_1}L^1_{\xi_2}}\non\\
&\lesssim t^{-\f14}\sum_{j\geq 0}\Big\|j^{-\f12}  e^{-\f{|\xi_1|^2}{j^2} t} \cF( P_j  (\chi b_{ 0} ) )\Big\|_{L^2_{\xi_1}L^1_{\xi_2}}\non\\
&\lesssim  t^{-\f14} \sum_{j\geq 0} j^{-\f12}\Big\| e^{-\f{|\xi_1|^2}{j^2} t}\Big\|_{L^2_{\xi_1}} \Big\|\cF(P_j  (\chi b_{1,0} ))\Big\|_{L^\infty_{\xi_1}L^2_{\xi_2}} \Big(\int_{|\xi_2|\leq j} d \xi_2\Big)^\f12\non\\
&\lesssim t^{-\f12} \Big(\sum_{0\leq j\leq 1} j^{\f12} \Big\| \cF( P_j   (\chi b_{1,0} ) )\Big\|_{L^\infty_{\xi_1}L^2_{\xi_2}}+\sum_{j\geq 1} j^{-\delta} \Big\|j^{\f12+\delta} \cF( P_j   (\chi b_{1,0} ) )\Big\|_{L^\infty_{\xi_1}L^2_{\xi_2}}\Big)\non\\
&\lesssim t^{-\f12}\Big( \|b_{1,0}\|_{L^1_{x_1}L^2_{x_2}}+\|\na^{\f12+\delta} b_{1,0}\|_{L^1_{x_1}L^2_{x_2}}\Big),
\end{align*}

and the $L^\infty$ estimate is very similar as \eqref{L12infty}
 \begin{align*}
&\Big\|\int_{\Ga_1}\textbf{1}_{|\xi_1|\leq 2}e^{\lam t} \f{1}{\om-|\xi_1|} \int^\infty_0  e^{-\om y_2} \big(\f{i\xi_1}{\lam}\big)^2\wh{b}_{2, 0}(y_2) dy_2 e^{-|\xi_1|x_2}  d\lam\Big\|_{L^1_{\xi_1}L^\infty_{x_2}}\non\\
&\lesssim \Big\| e^{-\f{|\xi_1|^2}{|\xi|^2} t} \cF \big(\chi b_{2, 0} \big) \Big\|_{L^1_{\xi_1}L^1_{\xi_2}}.\non\\
&\lesssim t^{-\f12}\big(\|   b_{ 0}  \|_{L^1_{x_1}L^2_{x_2}}+\|  \na^{\f32+\delta} b_{ 0}  \|_{L^1_{x_1}L^2_{x_2}}\big).
\end{align*}
The $\wh{b}_{2,0}$ part of $L_1^{2,3,4}$  is very similar as $J_1^{2, 3, 4}$, since
\beno
 \f{|\xi_1|}{|\lam|}=\f{|\xi_1|}{|\lam_\pm|}\sim 1, \quad \text{when} \quad |\xi_1|\geq \f{|\xi|^2}{2}.
 \eeno
The  $\wh{b}_{1,0}$ part  of $L_1^6$  goes to 0, since
\beno
\f{|\xi_1|}{|\lam|}=\f{|\xi_1|}{(|\xi_1|^2+\e^2-2\e|\xi_1|\cos \gamma)^\f12}.
\eeno
The  $\wh{b}_{1,0}$ part  of $L_1^7$  goes to 0, since
\beno
\f{|\xi_1|}{|\lam|}=\f{|\xi_1|}{R}\to 0.
\eeno

We can use the same way to deal with the $\Ga_5$ part of  $K_2$ in $b_{2, L}$ by the boundary condition $b_{2,0}|_{x_2=0}=0$,
\begin{align}\label{K25}
|K_2^5|
&=\Big|\f{1}{2\pi i} \int_{\Ga_5 } e^{\lam t}  \f{1}{ \om-|\xi_1|} \int^\infty_0 e^{-\om(x_2+y_2)} \f{|\xi_1|^2}{\lam^2} \wh{b}_{2, 0} (y_2) dy_2 d\lam\Big|\non\\
&=\Big|\f{1}{2\pi i} \int_{\Ga_5 } e^{\lam t}  \f{1}{ \om-|\xi_1|} \int^\infty_0 \f{1}{\om} \f{d}{ d y_2} \Big(e^{-\om(x_2+y_2)} \Big)\f{|\xi_1|^2}{\lam^2} \wh{b}_{2, 0} (y_2) dy_2 d\lam\Big|\non\\
&=\Big|\f{1}{2\pi i} \int_{\Ga_5 } e^{\lam t}  \f{1}{ \om-|\xi_1|} \f{1}{\om^2} \int^\infty_0 e^{-\om(x_2+y_2)} \f{|\xi_1|^2}{\lam^2} \pa_2\wh{b}_{2, 0} (y_2) dy_2 d\lam\Big|\non\\
&\lesssim \f{e^{\e t} \e}{|\xi_1|\e^\f12} \int_{-\pi}^\pi \|e^{-Re \om y_2}\|_{L^2_{y_2}}d \gamma\| \pa_2\wh{b}_{2, 0}\|_{L^2_{y_2}}\non\\
&\lesssim \f{e^{\e t}\e^\f34 }{|\xi_1|^\f32} \| \pa_2b_{2, 0}\|_{L^2_{x_2}}\rightarrow 0
 \quad \text{when} \quad \e\rightarrow 0 \quad \text{ for  fixed} \quad \xi_1,
\end{align}
where  we use
\begin{align*}
\om^2 &=\lam+|\xi_1|^2+\f{|\xi_1|^2}{\lam}\\
&=(\e\cos \gamma+|\xi_1|^2+\f{|\xi_1|^2}{\e}\cos \gamma)+i(\e\sin\gamma-\f{|\xi_1|^2}{\e} \sin \gamma)\\
&=(\e^2+|\xi_1|^4+\f{|\xi_1|^4}{\e^2}+2\e \cos \gamma |\xi_1|^2+2|\xi_1|^2 \cos 2\gamma+\f{2|\xi_1|^4}{\e} \cos \gamma)^\f12 e^{2i\theta_5}
\end{align*}
with
$$|\om|\gtrsim \f{|\xi_1|}{\sqrt{\e}},$$
and
$$|\om-|\xi_1||\geq|\om|-|\xi_1|\gtrsim \f{|\xi_1|}{\sqrt{\e}}.$$

\smallskip

{\bf Case 2.  $|\xi_1|> 2$.}

We also divide $L_{1, 2}$ as $ \wt{L}_{1, 2}^1+\wt{L}_{1, 2}^2$. By the definition of $\wt{\Ga}_{1}$, $\wt{L}_{1, 2}^1$ can be estimated as $L_{1, 2}^1$, and by the definition of $\wt{\Ga}_{2}$, $\wt{L}_{1, 2}^2$  can be estimated as $\wt{J}_{1, 2}^2$ by
$$
\f{|\xi_1|}{|\lam|}\lesssim \f{|\xi_1|}{-\lam'_-}\lesssim 1,
$$
which has  an exponential decay.

We mention that the decay rate  of $b_1$ is slower than $b_2$ because of  the whole space part
\beno
\f{1}{2\pi i}\int_{\Ga}e^{\lam t} \f{i\xi_1}{\lam} \f{1}{2\om} \int^\infty_0  \Big(e^{-\om|x_2-y_2|} -e^{-\om(x_2+y_2)} \Big)(\wh{u}_{1, 0}+\f{i\xi_1}{\lam} \wh{b}_{1, 0})(y_2) dy_2 d\lam.
\eeno
By the boundary condition $b_{1,0}|_{x_2=0}=0$, we have the  $L^2$ estimate
\begin{align*}
&\Big\|\textbf{1}_{|\xi_1|\geq 2}\f{|\xi_1|^2}{|\xi|^4} e^{-\f{|\xi_1|^2}{|\xi|^2} t}\f{|\xi|^4}{|\xi_1|^2}\cF\big(\chi b_{1, 0} \big)   \Big\|_{L^2_{\xi_1}L^2_{\xi_2}}\non\\
&\lesssim\Big\|e^{-\f{|\xi_1|^2}{|\xi|^2} t}  \cF\big( \chi b_{1, 0}\big) \Big\|_{L^2_{\xi_1}L^2_{\xi_2}}\non\\
&\lesssim \sum_{j\geq 0}\Big\|  e^{-\f{|\xi_1|^2}{j^2} t} \cF( P_j  (\chi b_{1, 0} )  )\Big\|_{L^2_{\xi_1}L^2_{\xi_2}}\non\\
&\lesssim   \sum_{j\geq 0} \Big\|e^{-\f{|\xi_1|^2}{j^2} t}\Big\|_{L^2_{\xi_1}} \Big\|\cF(P_j(\chi b_{1, 0} ) )\Big\|_{L^\infty_{\xi_1}L^2_{\xi_2}}\non\\
&\lesssim t^{-\f14} \Big(\sum_{0\leq j\leq 1} j^{\f12} \Big\| \cF( P_j (\chi b_{1, 0} )  )\Big\|_{L^\infty_{\xi_1}L^2_{\xi_2}}+\sum_{j\geq 1} j^{-\delta} \Big\|j^{\f12+\delta} \cF( P_j (\chi b_{1, 0}  ) )\Big\|_{L^\infty_{\xi_1}L^2_{\xi_2}}\Big)\non\\
&\lesssim t^{-\f14}\Big( \|b_{1, 0} \|_{L^1_{x_1}L^2_{x_2}}+\|\na^{\f12+\delta} b_{1, 0} \|_{L^1_{x_1}L^2_{x_2}}\Big),\end{align*}
and by the boundary condition $b_{1,0}|_{x_2=0}=0$, we have the $L^\infty$ estimate
\begin{align*}
&\Big\|\textbf{1}_{|\xi_1|\geq 2}e^{-\f{|\xi_1|^2}{|\xi|^2} t}  \cF (\chi b_{1, 0})  \Big\|_{L^1_{\xi_1}L^1_{\xi_2}}\non\\
&\lesssim \sum_{j\geq 0}\Big\|e^{-\f{|\xi_1|^2}{j^2} t} \cF( P_j (\chi  b_{1, 0})   )\Big\|_{L^1_{\xi_1}L^1_{\xi_2}}\non\\
&\lesssim \sum_{j\geq 0} \Big\|e^{-\f{|\xi_1|^2}{j^2} t}\Big\|_{L^1_{\xi_1}} \Big\|\cF(P_j (\chi b_{1,0}))\Big\|_{L^\infty_{\xi_1}L^2_{\xi_2}} \Big(\int_{|\xi_2|\leq j} d \xi_2\Big)^\f12\non\\
&\lesssim t^{-\f12} \Big(\sum_{0\leq j\leq 1} j^{\f32} \Big\| \cF( P_j (\chi b_{1, 0} )  )\Big\|_{L^\infty_{\xi_1}L^2_{\xi_2}}+\sum_{j\geq 1} j^{-\delta} \Big\|j^{\f32+\delta} \cF( P_j (\chi  b_{1, 0} )  )\Big\|_{L^\infty_{\xi_1}L^2_{\xi_2}}\Big)\non\\
&\lesssim t^{-\f12}\Big( \|b_{1, 0} \|_{L^1_{x_1}L^2_{x_2}}+\|\na^{\f32+\delta} b_{1, 0} \|_{L^1_{x_1}L^2_{x_2}}\Big).
\end{align*}

Thus we complete the proof of Proposition \ref{linear b}.
\end{proof}

\smallskip

Now we  give the  resolvent estimate of derivative.

\begin{proposition}\label{linear  nau}
We have the following $L^2$, $L^\infty$ estimates
 \begin{align*}
  \|\pa_1 u_L\|_{H^1}&\lesssim \lan t\ran^{-1}\Big(\|(u_0, b_0)\|_{L^1\cap L^2  }+\|\na^{\f32+\delta} (u_0, b_0)  \|_{L^1_{x_1}L^2_{x_2}}\Big),\\
  \|\pa_2 u_L\|_{H^1}&\lesssim \lan t\ran^{-\f34}\Big(\|(u_0, b_0)\|_{L^1\cap L^2  }+\|\na^{\f32+\delta} (u_0, b_0)  \|_{L^1_{x_1}L^2_{x_2}}\Big),\\
    \|\pa_1 u_L\|_{L^\infty}&\lesssim \lan t\ran^{-\f54}\Big( \|(u_0, b_0)\|_{L^1\cap L^2}+\|\na^{1+\delta} (u_0, b_0)  \|_{L^2}\Big),\\
     \|\p_1 b_L\|_{L^2}&\lesssim \lan t\ran^{-\f34}\Big(\|(u_0, b_0)\|_{L^1\cap L^2  }+\|\na^{\f32+\delta} (u_0, b_0)  \|_{L^1_{x_1}L^2_{x_2}}\Big),
     \end{align*}
 for the linearized problem \eqref{eq:MHDL}  (here $\delta>0$ small enough).
\end{proposition}
 \begin{proof}

 Recalling the structure of  $\wh{u}_{L}$  in \eqref{u1}, \eqref{u2}, we only  consider the typical term.

{\bf Step 1.} The  $L^2$ decay rate of $\pa_1 u_L$.\\

We first consider $J_2^1$ to obtain
\begin{align*}
&\Big\|\f{|\xi_1|^2}{|\xi|^4} e^{-\f{|\xi_1|^2}{|\xi|^2} t}|\xi_1|\f{\xi_2 }{i\xi_2+|\xi_1|}\cF\big(\chi u_{2, 0} \big) (-\xi_2)   \Big\|_{L^2_{\xi_1}L^2_{\xi_2}}\non\\
&\lesssim\Big\|\f{|\xi_1|^3}{|\xi|^4} e^{-\f{|\xi_1|^2}{|\xi|^2} t} \cF \big(\chi u_{2, 0} \big)   \Big\|_{L^2_{\xi_1}L^2_{\xi_2}}\non\\
&\lesssim\Big\|\f{|\xi_1|^2}{|\xi|^2} e^{-\f{|\xi_1|^2}{|\xi|^2} t} \cF \big(\chi u_{2, 0} \big)   \Big\|_{L^2_{\xi_1}L^2_{\xi_2}}\non\\
&\lesssim t^{-1} \Big\|e^{-\f{|\xi_1|^2}{|\xi|^2} t} \cF \big(\chi u_{2, 0} \big)   \Big\|_{L^2_{\xi_1}L^2_{\xi_2}}\non\\
&\lesssim t^{-1}\|u_{2, 0}\|_{L^2},
\end{align*}
and the  $J_2^2+J_2^3$ can be controlled by heat kernel as
\begin{align*}
&\Big\| \f{ \lam_- e^{\lam_- t}- \lam_+ e^{\lam_+t} }{\lam_+-\lam_-}\f{|\xi_1|\xi_2}{i\xi_2-|\xi_1|} \cF_{y_2}\Big(\chi(y_2)  \wh{u}_{2, 0} (y_2)  \Big)(\xi_2)\Big\|_{L^2_{\xi_1}L^2_{\xi_2}}\\
\lesssim&\Big\|\f{ \lam_- e^{\lam_- t}- \lam_+ e^{\lam_+t} }{\lam_+-\lam_-}
\f{|\xi_1|\xi_2}{i\xi_2-|\xi_1|}\Big\|_{L^2_{\xi_1}L^2_{\xi_2}}\Big\|\cF_{y_2}\Big(\chi(y_2)  \wh{u}_{2, 0} (y_2)  \Big)\Big\|_{L^\infty_{\xi_2}}\\
\lesssim&\| |\xi_1|e^{-\f{|\xi|^2}{4}t} \|_{L^2_{\xi_1}L^2_{\xi_2}}\|u_{2,0}\|_{L^1}\\
\lesssim& t^{-1}\|u_{2,0}\|_{L^1}.
\end{align*}
For $J_1^1$ we have
 \begin{align*}
&\Big\| e^{-\f{|\xi_1|^2}{|\xi|^2} t} |\xi_1|^{-\f12}|\xi_1|\f{|\xi_1|^2}{|\xi|^4}\cF \big(\chi u_{2, 0} \big) (-\xi_2)\Big\|_{L^2_{\xi_1}L^1_{\xi_2}} \non\\
&= \Big\| e^{-\f{|\xi_1|^2}{|\xi|^2} t} \f{|\xi_1|^\f52}{|\xi|^4}\cF \big(\chi u_{2, 0} \big) \Big\|_{L^2_{\xi_1}L^1_{\xi_2}}\non\\
&\lesssim \Big\| e^{-\f{|\xi_1|^2}{|\xi|^2} t} \f{|\xi_1|^\f32}{|\xi|^2}\cF \big(\chi u_{2, 0} \big) \Big\|_{L^2_{\xi_1}L^1_{\xi_2}}\non\\
&\lesssim  t^{-\f34} \sum_{j\geq 0}  \Big\|j^{-\f12} e^{-\f{|\xi_1|^2}{j^2} t}  \cF\big(P_j (\chi u_{2, 0})\big)\Big\|_{L^2_{\xi_1}L^1_{\xi_2}}\non\\
&\lesssim  t^{-\f34} \sum_{j\geq 0} j^{-\f12} \Big\| e^{-\f{|\xi_1|^2}{j^2} t}\Big\|_{L^2_{\xi_1}} \Big\|\cF(P_j  (\chi u_{2, 0} ))\Big\|_{L^\infty_{\xi_1}L^2_{\xi_2}} \Big(\int_{|\xi_2|\leq j} d \xi_2\Big)^\f12\non\\
&\lesssim  t^{-1} \Big(\sum_{j\leq 1} j^{\f12} \Big\|\cF(P_j  (\chi u_{2, 0} ))\Big\|_{L^\infty_{\xi_1}L^2_{\xi_2}}+\sum_{j\geq 1} j^{-\delta}\Big\| j^{\f12+\delta}\cF \big(P_j(\chi u_{2, 0}) \big)   \Big\|_{L^\infty_{\xi_1}L^2_{\xi_2}}\Big)
 \non\\
 &\lesssim t^{-1} \Big( \| u_{2, 0} \|_{L^1_{x_1}L^2_{x_2}}+\| \na^{\f12+\delta} u_{2, 0} \|_{L^1_{x_1}L^2_{x_2}}\Big),
\end{align*}
and  $J_1^2+J_1^3$ is estimated as
\begin{align*}
&\Big\| \f{ \lam_- e^{\lam_- t}- \lam_+ e^{\lam_+t} }{\lam_+-\lam_-}\f{\xi_2|\xi_1|}{i\xi_2-|\xi_1|}|\xi_1|^{-\f12} \cF_{y_2}\Big(\chi(y_2)  \wh{u}_{2, 0} (y_2)  \Big)(\xi_2)\Big\|_{L^2_{\xi_1}L^1_{\xi_2}}\\
\lesssim&\Big\| \f{\xi_2}{|\xi|}|\xi_1|^{\f12}e^{-\f{|\xi|^2}{4}t}\cF_{y_2}\Big(\chi(y_2)  \wh{u}_{2, 0} (y_2) \Big)(\xi_2) \Big\|_{L^2_{\xi_1}L^1_{\xi_2}}\\
\lesssim&\Big\||\xi_1|^{\f12}e^{-\f{|\xi|^2}{4}t}\cF_{y_2}\Big(\chi(y_2)  \wh{u}_{2, 0} (y_2) \Big)(\xi_2) \Big\|_{L^2_{\xi_1}L^1_{\xi_2}}\\
\lesssim &  \Big\||\xi_1|^\f12e^{-\f{|\xi|^2}{4} t}   \Big\|_{L^2_{\xi_1}L^1_{\xi_2}}\Big\|\cF \big(\chi u_{2,0} \big)\Big\|_{L^\infty_{\xi_1}L^\infty_{\xi_2}}\non\\
\lesssim& t  ^{-1} \|u_{2,0}\|_{L^1}.
\end{align*}

{\bf Step 2.} The  $L^2$ decay rate of $\pa_1 \na u_L$.

We first consider $J_2^1$ to obtain
\begin{align*}
&\Big\|\f{|\xi_1|}{|\xi|^2} e^{-\f{|\xi_1|^2}{|\xi|^2} t}|\xi_1||\xi|\f{\xi_2 }{i\xi_2+|\xi_1|}\cF\big(\chi u_{2, 0} \big) (-\xi_2)   \Big\|_{L^2_{\xi_1}L^2_{\xi_2}}\non\\
&\lesssim\Big\|\f{|\xi_1|^2}{|\xi|^2} e^{-\f{|\xi_1|^2}{|\xi|^2} t}|\xi| \cF \big(\chi u_{2, 0} \big)   \Big\|_{L^2_{\xi_1}L^2_{\xi_2}}\non\\
&\lesssim  t^{-1}  \|\na u_{2, 0}\|_{L^2},
\end{align*}
and  $J_2^2+J_2^3$ can be controlled by heat kernel,
\begin{align*}
&\Big\| \f{ \lam_- e^{\lam_- t}- \lam_+ e^{\lam_+t} }{\lam_+-\lam_-}\f{|\xi_1|\xi_2}{i\xi_2-|\xi_1|}|\xi| \cF_{y_2}\Big(\chi(y_2)  \wh{u}_{2, 0} (y_2)  \Big)(\xi_2)\Big\|_{L^2_{\xi_1}L^2_{\xi_2}}\\
\lesssim&\Big\|\f{ \lam_- e^{\lam_- t}- \lam_+ e^{\lam_+t} }{\lam_+-\lam_-}
\f{|\xi_1|\xi_2}{i\xi_2-|\xi_1|}|\xi|\Big\|_{L^2_{\xi_1}L^2_{\xi_2}}\Big\|\cF_{y_2}\Big(\chi(y_2)  \wh{u}_{2, 0} (y_2)  \Big)\Big\|_{L^\infty_{\xi_2}}\\
\lesssim&\| \xi_1\xi_2e^{-\f{|\xi|^2}{4}t} \|_{L^2_{\xi_1}L^2_{\xi_2}}\|u_{2,0}\|_{L^1}\\
\lesssim& t  ^{-1} \|u_{2,0}\|_{L^1}.
\end{align*}
For $J_1^1$ we have
 \begin{align*}
&\Big\| e^{-\f{|\xi_1|^2}{|\xi|^2} t} |\xi_1|^{-\f12}|\xi_1||\xi|\f{|\xi_1|^2}{|\xi|^4}\cF \big(\chi u_{2, 0} \big) (-\xi_2) \Big\|_{L^2_{\xi_1}L^1_{\xi_2}} \non\\
&= \Big\| e^{-\f{|\xi_1|^2}{|\xi|^2} t} \f{|\xi_1|^\f52}{|\xi|^3}\cF \big(\chi u_{2, 0} \big) \Big\|_{L^2_{\xi_1}L^1_{\xi_2}},\\
&\lesssim \Big\| e^{-\f{|\xi_1|^2}{|\xi|^2} t} \f{|\xi_1|^\f32}{|\xi|}\cF \big(\chi u_{2, 0} \big) \Big\|_{L^2_{\xi_1}L^1_{\xi_2}},\\
&\lesssim  t^{-\f34} \sum_{j\geq 0}  \Big\|j^\f12 e^{-\f{|\xi_1|^2}{j^2} t}  \cF\big(P_j (\chi u_{2, 0})\big)\Big\|_{L^2_{\xi_1}L^1_{\xi_2}}\non\\
&\lesssim  t^{-\f34} \sum_{j\geq 0} j^\f12 \Big\| e^{-\f{|\xi_1|^2}{j^2} t}\Big\|_{L^2_{\xi_1}} \Big\|\cF(P_j  (\chi u_{2, 0} ))\Big\|_{L^\infty_{\xi_1}L^2_{\xi_2}}\Big(\int_{|\xi_2|\leq j} d \xi_2\Big)^\f12 \non\\
&\lesssim t^{-1}\Big(\sum_{j\leq1}j^\f32\Big\|\cF(P_j  (\chi u_{2, 0} ))   \Big\|_{L^\infty_{\xi_1}L^2_{\xi_2}}+\sum_{j\geq1}j^{-\delta} \Big\| j^{\f32+\delta}\cF(P_j  (\chi u_{2, 0} ))   \Big\|_{L^\infty_{\xi_1}L^2_{\xi_2}}\Big)\non\\
&\lesssim t^{-1} \Big( \| u_{2, 0} \|_{L^1_{x_1}L^2_{x_2}}+\| \na^{\f32+\delta} u_{2, 0} \|_{L^1_{x_1}L^2_{x_2}}\Big),
\end{align*}
and for $J_1^2+J_1^3$,
\begin{align*}
&\Big\| \f{ \lam_- e^{\lam_- t}- \lam_+ e^{\lam_+t} }{\lam_+-\lam_-}\f{\xi_2}{i\xi_2-|\xi_1|}|\xi_1|^{-\f12} |\xi_1||\xi|\cF_{y_2}\Big(\chi(y_2)  \wh{u}_{2, 0} (y_2)  \Big)(\xi_2)\Big\|_{L^2_{\xi_1}L^1_{\xi_2}}\\
\lesssim&\Big\||\xi_1|^{\f12}|\xi_2|e^{-\f{|\xi|^2}{4}t}\cF_{y_2}\Big(\chi(y_2)  \wh{u}_{2, 0} (y_2) \Big)(\xi_2) \Big\|_{L^2_{\xi_1}L^1_{\xi_2}}\\
\lesssim &  \Big\||\xi_1|^\f12|\xi_2|e^{-\f{|\xi|^2}{4} t}   \Big\|_{L^2_{\xi_1}L^1_{\xi_2}}\Big\|\cF \big(\chi u_{2,0} \big)\Big\|_{L^\infty_{\xi_1}L^\infty_{\xi_2}}\non\\
\lesssim& t  ^{-\f32} \|u_{2,0}\|_{L^1},
\end{align*}

{\bf Step 3.} The  $L^2$ decay rate of $\pa_2  u_L$ and $\pa_2 \na  u_L$.\\

We mention that the whole space part will reduce the decay rate.   For $\|\pa_2  u_L\|_{L^2}$, we only consider the typical term:
\begin{align*}
&\Big\|\f{|\xi_1|^2}{|\xi|^4} e^{-\f{|\xi_1|^2}{|\xi|^2} t}|\xi|\f{\xi_2 }{i\xi_2+|\xi_1|}\cF\big(\chi u_{1, 0} \big) (-\xi_2)   \Big\|_{L^2_{\xi_1}L^2_{\xi_2}}\non\\
&\lesssim\Big\|\f{|\xi_1|^2}{|\xi|^3} e^{-\f{|\xi_1|^2}{|\xi|^2} t} \cF \big(\chi u_{1, 0} \big)   \Big\|_{L^2_{\xi_1}L^2_{\xi_2}}\non\\
&\lesssim\Big\|\f{|\xi_1|}{|\xi|} e^{-\f{|\xi_1|^2}{|\xi|^2} t} \cF \big(\chi u_{1, 0} \big)   \Big\|_{L^2_{\xi_1}L^2_{\xi_2}}\non\\
&\lesssim  t^{-\f12} \sum_{j\geq 0}  \Big\|e^{-\f{|\xi_1|^2}{j^2} t}  \cF\big(P_j (\chi u_{1, 0})\big)\Big\|_{L^2_{\xi_1}L^2_{\xi_2}}\non\\
&\lesssim  t^{-\f12} \sum_{j\geq 0} \Big\| e^{-\f{|\xi_1|^2}{j^2} t}\Big\|_{L^2_{\xi_1}} \Big\|\cF(P_j  (\chi u_{1, 0} ))\Big\|_{L^\infty_{\xi_1}L^2_{\xi_2}} \non\\
&\lesssim t^{-\f34}\Big(\sum_{j\leq1}j^\f12\Big\|\cF(P_j  (\chi u_{1, 0} )) \Big\|_{L^\infty_{\xi_1}L^2_{\xi_2}}+\sum_{j\geq1}j^{-\delta}\Big\| j^{\f12+\delta} \cF(P_j  (\chi u_{1, 0} ))    \Big\|_{L^\infty_{\xi_1}L^2_{\xi_2}}\Big)\non\\
&\lesssim t^{-\f34} \Big(\| u_{2, 0} \|_{L^1_{x_1}L^2_{x_2}}+\| \na^{\f12+\delta} u_{1, 0} \|_{L^1_{x_1}L^2_{x_2}}\Big).
\end{align*}
And for $\|\pa_2  \na u_L\|_{L^2}$, we consider
\begin{align*}
&\Big\|\f{|\xi_1|}{|\xi|^2} e^{-\f{|\xi_1|^2}{|\xi|^2} t}|\xi|^2\f{\xi_2 }{i\xi_2+|\xi_1|}\cF\big(\chi u_{1, 0} \big) (-\xi_2)   \Big\|_{L^2_{\xi_1}L^2_{\xi_2}}\non\\
&\lesssim  t^{-\f12} \sum_{j\geq 0}j \Big\| e^{-\f{|\xi_1|^2}{j^2} t}\Big\|_{L^2_{\xi_1}} \Big\|\cF(P_j  (\chi u_{1, 0} ))\Big\|_{L^\infty_{\xi_1}L^2_{\xi_2}} \non\\
&\lesssim t^{-\f34}\Big(\sum_{j\leq1}j^\f32\Big\|\cF(P_j  (\chi u_{1, 0} )) \Big\|_{L^\infty_{\xi_1}L^2_{\xi_2}}+\sum_{j\geq1}j^{-\delta}\Big\| j^{\f32+\delta} \cF(P_j  (\chi u_{1, 0} ))    \Big\|_{L^\infty_{\xi_1}L^2_{\xi_2}}\Big)\non\\
&\lesssim t^{-\f34} \Big(\| u_{2, 0} \|_{L^1_{x_1}L^2_{x_2}}+\| \na^{\f32+\delta} u_{1, 0} \|_{L^1_{x_1}L^2_{x_2}}\Big).
\end{align*}
{\bf Step 4.} The  $L^\infty$ decay rate of $\pa_1 u_L$.\\

We first consider $J_2^1$ to see
\begin{align}\label{5.1}
&\Big\|\f{|\xi_1|^2}{|\xi|^4} e^{-\f{|\xi_1|^2}{|\xi|^2} t}|\xi_1|\f{\xi_2 }{i\xi_2+|\xi_1|}\cF\big(\chi u_{2, 0} \big) (-\xi_2)   \Big\|_{L^1_{\xi_1}L^1_{\xi_2}}\non\\
&\lesssim\Big\|\f{|\xi_1|^3}{|\xi|^4} e^{-\f{|\xi_1|^2}{|\xi|^2} t} \cF \big(\chi u_{2, 0} \big)   \Big\|_{L^1_{\xi_1}L^1_{\xi_2}}\non\\
&\lesssim\Big\|\f{|\xi_1|^2}{|\xi|^2} e^{-\f{|\xi_1|^2}{|\xi|^2} t} \cF \big(\chi u_{2, 0} \big)   \Big\|_{L^1_{\xi_1}L^1_{\xi_2}}\non\\
&\lesssim  t^{-1} \sum_{j\geq 0}  \Big\| e^{-\f{|\xi_1|^2}{j^2} t}  \cF\big(P_j (\chi u_{2, 0})\big)\Big\|_{L^1_{\xi_1}L^1_{\xi_2}}\non\\
&\lesssim  t^{-1} \sum_{j\geq 0}  \Big\| e^{-\f{|\xi_1|^2}{j^2} t}\Big\|_{L^2_{\xi_1}} \Big\|\cF(P_j  (\chi u_{2, 0} ))\Big\|_{L^2_{\xi_1}L^2_{\xi_2}}\Big(\int_{|\xi_2|\leq j} d \xi_2\Big)^\f12 \non\\
&\lesssim t^{-\f54}\Big(\sum_{j\leq1}j\Big\|\cF(P_j  (\chi u_{2, 0} ))   \Big\|_{L^2_{\xi_1}L^2_{\xi_2}}+\sum_{j\geq1}j^{-\delta} \Big\| j^{1+\delta}\cF(P_j  (\chi u_{2, 0} ))   \Big\|_{L^2_{\xi_1}L^2_{\xi_2}}\Big)\non\\
&=t^{-\f54} \Big(\| u_{2, 0} \|_{L^2}+\| \na^{1+\delta} u_{2, 0} \|_{L^2}\Big),
\end{align}
and the  $J_2^2+J_2^3$ can be controlled by heat kernel
\begin{align}\label{5.2}
&\Big\| \f{ \lam_- e^{\lam_- t}- \lam_+ e^{\lam_+t} }{\lam_+-\lam_-}\f{|\xi_1|\xi_2}{i\xi_2-|\xi_1|} \cF_{y_2}\Big(\chi(y_2)  \wh{u}_{2, 0} (y_2)  \Big)(\xi_2)\Big\|_{L^1_{\xi_1}L^1_{\xi_2}}\non\\
\lesssim&\Big\|\f{ \lam_- e^{\lam_- t}- \lam_+ e^{\lam_+t} }{\lam_+-\lam_-}
\f{|\xi_1|\xi_2}{i\xi_2-|\xi_1|}\Big\|_{L^1_{\xi_1}L^1_{\xi_2}}\Big\|\cF_{y_2}\Big(\chi(y_2)  \wh{u}_{2, 0} (y_2)  \Big)\Big\|_{L^\infty_{\xi_2}}\non\\
\lesssim&\| |\xi_1|e^{-\f{|\xi|^2}{4}t} \|_{L^1_{\xi_1}L^1_{\xi_2}}\|u_{2,0}\|_{L^1}\non\\
\lesssim& t^{-\f32}\|u_{2,0}\|_{L^1}.
\end{align}

The $J_1$ is very similar as $J_2$, since
 \begin{align*}
\|e^{-|\xi_1|x_2}  \|_{L^\infty_{x_2}}\lesssim 1.
\end{align*}

{\bf Step 5.} The  $L^2$ decay rate of $\pa_1 b_L$.\\

 Recalling the structure of  $\wh{b}_{L}$  in \eqref{b1}, \eqref{b2}, we only consider the typical term. For $L_2^1$, we have
\begin{align*}
&\Big\|\f{|\xi_1|^2}{|\xi|^4} e^{-\f{|\xi_1|^2}{|\xi|^2} t}|\xi_1|\f{|\xi|^4}{|\xi_1|^2}\f{\xi_2 }{i\xi_2+|\xi_1|}\cF\big(\chi b_{2, 0} \big) (-\xi_2)   \Big\|_{L^2_{\xi_1}L^2_{\xi_2}}\non\\
&\lesssim\Big\| |\xi_1| e^{-\f{|\xi_1|^2}{|\xi|^2} t} \cF \big(\chi b_{2, 0} \big)   \Big\|_{L^2_{\xi_1}L^2_{\xi_2}}\non\\
&\lesssim  t^{-\f12} \sum_{j\geq 0}  \Big\|j e^{-\f{|\xi_1|^2}{j^2} t}  \cF\big(P_j (\chi b_{2, 0})\big)\Big\|_{L^2_{\xi_1}L^2_{\xi_2}}\non\\
&\lesssim  t^{-\f12} \sum_{j\geq 0} j \Big\| e^{-\f{|\xi_1|^2}{j^2} t}\Big\|_{L^2_{\xi_1}} \Big\|\cF(P_j  (\chi b_{2, 0} ))\Big\|_{L^\infty_{\xi_1}L^2_{\xi_2}} \non\\
&\lesssim t^{-\f34}\Big(\sum_{j\leq1}j^\f32\Big\|  \cF\big(P_j (\chi b_{2, 0})\big)\Big\|_{L^2_{\xi_1}L^2_{\xi_2}}+\sum_{j\geq1}j^{-\delta} \Big\| j^{\f32+\delta} \cF\big(P_j (\chi b_{2, 0})\big)\Big\|_{L^2_{\xi_1}L^2_{\xi_2}}\Big)\non\\
&\lesssim t^{-\f34}  \Big(\| b_{2, 0} \|_{L^1_{x_1}L^2_{x_2}}+\| \na^{\f32+\delta} b_{2, 0} \|_{L^1_{x_1}L^2_{x_2}}\Big),
\end{align*}
and the  $L_2^2+L_2^3$ can do the same as $J_2^2+J_2^3$ in $\pa_1 u$, since $\f{|\xi_1|}{|\lam|}\sim1$.

For $L_1^1$, we see that
 \begin{align*}
&\Big\|\f{|\xi_1|^2}{|\xi|^4} e^{-\f{|\xi_1|^2}{|\xi|^2} t} |\xi_1|^{-\f12}|\xi_1|\f{|\xi|^4}{|\xi_1|^2}\f{\xi_2 }{i\xi_2+|\xi_1|}\cF \big(\chi b_{2, 0} \big) \Big\|_{L^2_{\xi_1}L^1_{\xi_2}} \non\\
&\lesssim \Big\| e^{-\f{|\xi_1|^2}{|\xi|^2} t} |\xi_1|^\f12\cF \big(\chi b_{2, 0} \big) \Big\|_{L^2_{\xi_1}L^1_{\xi_2}},\\
&\lesssim\Big\| |\xi_1| e^{-\f{|\xi_1|^2}{|\xi|^2} t} \cF \big(\chi b_{1, 0} \big)   \Big\|_{L^2_{\xi_1}L^2_{\xi_2}}\non\\
&\lesssim  t^{-\f12} \sum_{j\geq 0}  \Big\|j e^{-\f{|\xi_1|^2}{j^2} t}  \cF\big(P_j (\chi b_{1, 0})\big)\Big\|_{L^2_{\xi_1}L^2_{\xi_2}}\non\\
&\lesssim  t^{-\f12} \sum_{j\geq 0} j \Big\| e^{-\f{|\xi_1|^2}{j^2} t}\Big\|_{L^2_{\xi_1}} \Big\|\cF(P_j  (\chi b_{1, 0} ))\Big\|_{L^\infty_{\xi_1}L^2_{\xi_2}} \non\\
&\lesssim t^{-\f34}\Big(\sum_{j\leq1}j^\f32\Big\|  \cF\big(P_j (\chi b_{1, 0})\big)\Big\|_{L^2_{\xi_1}L^2_{\xi_2}}+\sum_{j\geq1}j^{-\delta} \Big\| j^{\f32+\delta} \cF\big(P_j (\chi b_{1, 0})\big)\Big\|_{L^2_{\xi_1}L^2_{\xi_2}}\Big)\non\\
&\lesssim t^{-\f34}  \Big(\| b_{2, 0} \|_{L^1_{x_1}L^2_{x_2}}+\| \na^{\f32+\delta} b_{1, 0} \|_{L^1_{x_1}L^2_{x_2}}\Big),
\end{align*}
and the  $L_1^2+L_1^3$ can do the same as $J_1^2+J_1^3$ in Step 1, since $\f{|\xi_1|}{|\lam|}\sim1$.

Thus we complete the proof of Proposition \ref{linear nau}.

\end{proof}
\smallskip

\section{The  decay  rate of nonlinear system}
Now we consider the nonlinear part of \eqref{eq:u sol2}
\begin{align}\label{nonlinear}
\left\{
\begin{array}{l}
\wh{u}_{1, N}(\xi_1,x_2, t)=\f{1}{2\pi i}\int_{\Ga}e^{\lam t}
\Big\{ -\Big(E_{\om}[ \wh{f}_{1, \lam}+\f{i\xi_1}{\lam} \wh{g}_{1, \lam}]_0+i\xi_1 E_{\om} [E_{|\xi_1|} [i\xi_1 \wh{f}_{1, \lam}+\p_2 \wh{f}_{2, \lam}]]_0\\
\qquad\qquad +\f{i\xi_1}{|\xi_1|} \{E_{\om} [\wh{f}_{2, \lam} +\f{i\xi_1}{\lam} \wh{g}_{2, \lam}]_0
+E_{\om}[\p_{2} E_{|\xi_1|}[i\xi_1\wh{f}_{1, \lam}+\p_2 \wh{f}_{2, \lam}]]_0\}\Big) e^{-\om x_2}\\
\qquad\qquad+E_{\om}[  \wh{f}_{1, \lam}+\f{i\xi_1}{\lam} \wh{g}_{1, \lam}]+i\xi_1 E_{\om} [E_{|\xi_1|}[i\xi_1 \wh{f}_{1, \lam}+\p_2 \wh{f}_{2, \lam}]]\\
\qquad\qquad+ \f{i\xi_1E_{\om}[e^{-|\xi_1|x_2}]}{|\xi_1|E_{\om}[e^{-|\xi_1|x_2}]_0}\{E_{\om} [\wh{f}_{2, \lam}+\f{i\xi_1}{\lam} \wh{g}_{2, \lam}]_0+E_{\om}[\p_{2}E_{|\xi_1|}[ i\xi_1\wh{f}_{1, \lam}+\p_2 \wh{f}_{2, \lam}]]_0\}\Big\}d\lam,\\
\wh{u}_{2, N}(\xi_1,x_2, t)=\f{1}{2\pi i}\int_{\Ga}e^{\lam t}
\Big\{  E_{\om}[  \wh{f}_{2, \lam}+\f{i\xi_1}{\lam} \wh{g}_{2, \lam}]+E_{\om} [\p_2 E_{|\xi_1|}[i\xi_1 \wh{f}_{1, \lam}+\p_2 \wh{f}_{2, \lam}]]\\
\qquad\qquad-\f{E_{\om} [ e^{-|\xi_1|x_2}]}{E_{\om}[e^{-|\xi_1|x_2}]_0}\{E_{\om} [\wh{f}_{2, \lam}+\f{i\xi_1}{\lam} \wh{g}_{2, \lam}]_0+E_{\om}[\p_{2}E_{|\xi_1|}[ i\xi_1\wh{f}_{1, \lam}+\p_2 \wh{f}_{2, \lam}]]_0\}\Big\}d\lam,\\
\wh{b}_{1, N}(\xi_1, x_2,t)=\f{1}{2\pi i}\int_{\Ga}e^{\lam t}
\Big\{\Big( -\f{i\xi_1} {\lam} E_{\om} [\wh{f}_{1, \lam}+\f{i\xi_1}{\lam} \wh{g}_{1, \lam} ]_0+\f{|\xi_1|^2}{\lam}  E_{\om} [E_{|\xi_1|}[i\xi_1  \wh{f}_{1, \lam}+\p_2\wh{f}_{2, \lam}]]_0\\
\qquad\qquad +\f{|\xi_1|}{\lam} \{ E_{\om} [\wh{f}_{2, \lam}+\f{i\xi_1}{\lam} \wh{g}_{2, \lam}]_0+E_{\om}[\p_2 E_{|\xi_1|}[ i\xi_1\wh{f}_{1, \lam}+\p_2 \wh{f}_{2, \lam}]]_0\} \Big) e^{-\om x_2}\\
 \qquad \qquad + \f{i\xi_1} {\lam} E_{\om} [\wh{f}_{1, \lam}+\f{i\xi_1}{\lam} \wh{g}_{1, \lam} ]-\f{|\xi_1|^2}{\lam}  E_{\om} [E_{|\xi_1|}[i\xi_1  \wh{f}_{1, \lam}+\p_2\wh{f}_{2, \lam}]] \\
 \qquad\qquad-\f{|\xi_1|}{\lam} \f{E_{\om}[e^{-|\xi_1|x_2}]}{ E_{\om}[ e^{-|\xi_1|x_2}]_0} \{ E_{\om} [\wh{f}_{2, \lam}+\f{i\xi_1}{\lam} \wh{g}_{2, \lam}]_0+E_{\om}[\p_2 E_{|\xi_1|}[ i\xi_1\wh{f}_{1, \lam}+\p_2 \wh{f}_{2, \lam}]]_0\} \Big\}d\lam+\wh{g}_{1},\\
 \wh{b}_{2, N}(\xi_1, x_2,t)=\f{1}{2\pi i}\int_{\Ga}e^{\lam t}
\Big\{ \f{i\xi_1} {\lam} E_{\om} [\wh{f}_{2, \lam}+\f{i\xi_1}{\lam} \wh{g}_{2, \lam} ]+\f{i\xi_1}{\lam}E_{\om} [  \p_2 E_{|\xi_1|}[i\xi_1  \wh{f}_{1, \lam}+\p_2\wh{f}_{2, \lam}]]\\
\qquad \qquad  -\f{i\xi_1 }{\lam} \f{E_{\om}[e^{-|\xi_1|x_2}]}{E_{\om}[e^{-|\xi_1|x_2}]_0} \{ E_{\om} [\wh{f}_{2, \lam}+\f{i\xi_1}{\lam} \wh{g}_{2, \lam}]_0+E_{\om}[\p_2 E_{|\xi_1|}[ i\xi_1\wh{f}_{1, \lam}+\p_2 \wh{f}_{2, \lam}]]_0\}\Big\}d\lam+\wh{g}_2,
\end{array}\right.
\end{align}
and $\wh{u}_{1, N}$ can be rewritten as
\begin{align*}
&\wh{u}_{1, N}(\xi_1,x_2, t)\non\\
&=-\f{1}{2\pi i}\int_{\Ga}e^{\lam t} \f{1}{2\om} \int^\infty_0  e^{-\om(x_2+y_2)} ( \wh{f}_{1, \lam}+\f{i\xi_1}{\lam} \wh{g}_{1, \lam})(y_2) dy_2 d\lam\non\\
&\quad+\f{1}{2\pi i}\int_{\Ga}e^{\lam t} \f{1}{2\om} \int^\infty_0  e^{-\om|x_2-y_2|} ( \wh{f}_{1, \lam}+\f{i\xi_1}{\lam} \wh{g}_{1, \lam})(y_2) dy_2 d\lam\non\\
&\quad+\f{i\xi_1}{|\xi_1|} \f{1}{2\pi i}\int_{\Ga}e^{\lam t} \int^\infty_0  e^{-\om y_2} ( \wh{f}_{2, \lam}+\f{i\xi_1}{\lam} \wh{g}_{2, \lam})(y_2) dy_2 \f{\lam (\om+|\xi_1|)}{\lam^2+|\xi_1|^2 }\Big( e^{-|\xi_1|x_2}-e^{-\om x_2}\Big)  d\lam\\
&\quad -\f{1}{2\pi i}\int_{\Ga}e^{\lam t} \f{i\xi_1}{2\om} \int^\infty_0  e^{-\om(x_2+y_2)} E_{|\xi_1|} [i\xi_1 \wh{f}_{1, \lam}+\p_2 \wh{f}_{2, \lam}] (y_2) dy_2 d\lam\non\\
&\quad+\f{1}{2\pi i}\int_{\Ga}e^{\lam t} \f{i\xi_1}{2\om} \int^\infty_0  e^{-\om|x_2-y_2|} E_{|\xi_1|} [i\xi_1 \wh{f}_{1, \lam}+\p_2 \wh{f}_{2, \lam}] (y_2) dy_2 d\lam\non\\
&\quad+\f{i\xi_1}{|\xi_1|} \f{1}{2\pi i}\int_{\Ga}e^{\lam t} \int^\infty_0  e^{-\om y_2} \p_{2} E_{|\xi_1|}[i\xi_1\wh{f}_{1, \lam}+\p_2 \wh{f}_{2, \lam}] (y_2) dy_2 \f{\lam (\om+|\xi_1|)}{\lam^2+|\xi_1|^2 }\Big( e^{-|\xi_1|x_2}-e^{-\om x_2}\Big)  d\lam\\
&=\f{1}{2\pi i}\int^t_0 \int_{\Ga}e^{\lam  (t-\tau)} \f{1}{2\om} \int^\infty_0  \Big(e^{-\om|x_2-y_2|}-e^{-\om(x_2+y_2)} \Big)  ( \wh{f}_{1}+\f{i\xi_1}{\lam} \wh{g}_{1})(\tau, y_2) dy_2 d\lam d\tau \non\\
&\quad+\f{i\xi_1}{|\xi_1|} \f{1}{2\pi i} \int^t_0 \int_{\Ga}e^{\lam (t-\tau)} \int^\infty_0  e^{-\om y_2} ( \wh{f}_{2 }+\f{i\xi_1}{\lam} \wh{g}_{2})(\tau, y_2) dy_2 \f{1}{\om-|\xi_1| }e^{-|\xi_1|x_2} d\lam d\tau\non\\
&\quad+\f{i\xi_1}{|\xi_1|} \f{1}{2\pi i} \int^t_0 \int_{\Ga}e^{\lam (t-\tau)} \int^\infty_0  e^{-\om (x_2+y_2)} ( \wh{f}_{2 }+\f{i\xi_1}{\lam} \wh{g}_{2})(\tau, y_2) dy_2 \f{1}{\om-|\xi_1| } d\lam d\tau\non\\
&\quad+\f{1}{2\pi i}\int^t_0 \int_{\Ga}e^{\lam  (t-\tau)} \f{i\xi_1}{2\om} \int^\infty_0  \Big(e^{-\om|x_2-y_2|}-e^{-\om(x_2+y_2)} \Big)  E_{|\xi_1|} [i\xi_1 \wh{f}_{1}+\p_2 \wh{f}_{2}](\tau, y_2) dy_2 d\lam d\tau \non\\
&\quad+\f{i\xi_1}{|\xi_1|} \f{1}{2\pi i} \int^t_0 \int_{\Ga}e^{\lam (t-\tau)} \int^\infty_0  e^{-\om y_2} \p_2 E_{|\xi_1|} [i\xi_1 \wh{f}_{1}+\p_2 \wh{f}_{2}] (\tau, y_2) dy_2  \f{1}{\om-|\xi_1| } e^{-|\xi_1|x_2}  d\lam d\tau\non\\
&\quad+\f{i\xi_1}{|\xi_1|} \f{1}{2\pi i} \int^t_0 \int_{\Ga}e^{\lam (t-\tau)} \int^\infty_0  e^{-\om (x_2+y_2)} \p_2 E_{|\xi_1|} [i\xi_1 \wh{f}_{1}+\p_2 \wh{f}_{2}] (\tau, y_2) dy_2  \f{1}{\om-|\xi_1| }   d\lam d\tau\non\\
&:=\sum_{i=1}^6 M_i,
\end{align*}
 $\wh{b}_{1, N}$ can be rewritten as
\begin{align*}
&\wh{b}_{1, N}(\xi_1,x_2, t)\non\\
&=\f{1}{2\pi i}\int^t_0 \int_{\Ga}e^{\lam  (t-\tau)}\f{i\xi_1}{\lam} \f{1}{2\om} \int^\infty_0  \Big(e^{-\om|x_2-y_2|}-e^{-\om(x_2+y_2)} \Big)  ( \wh{f}_{1}+\f{i\xi_1}{\lam} \wh{g}_{1})(\tau, y_2) dy_2 d\lam d\tau \non\\
&\quad- \f{1}{2\pi i} \int^t_0 \int_{\Ga}e^{\lam (t-\tau)}\f{|\xi_1|}{\lam}  \int^\infty_0  e^{-\om y_2} ( \wh{f}_{2 }+\f{i\xi_1}{\lam} \wh{g}_{2})(\tau, y_2) dy_2 \f{1}{\om-|\xi_1| }e^{-|\xi_1|x_2} d\lam d\tau\non\\
&\quad+ \f{1}{2\pi i} \int^t_0 \int_{\Ga}e^{\lam (t-\tau)} \f{|\xi_1|}{\lam} \int^\infty_0  e^{-\om (x_2+y_2)} ( \wh{f}_{2 }+\f{i\xi_1}{\lam} \wh{g}_{2})(\tau, y_2) dy_2 \f{1}{\om-|\xi_1| } d\lam d\tau\non\\
&\quad-\f{1}{2\pi i}\int^t_0 \int_{\Ga}e^{\lam  (t-\tau)} \f{|\xi_1|^2}{\lam} \f{1}{2\om} \int^\infty_0  \Big(e^{-\om|x_2-y_2|}-e^{-\om(x_2+y_2)} \Big)  E_{|\xi_1|} [i\xi_1 \wh{f}_{1}+\p_2 \wh{f}_{2}](\tau, y_2) dy_2 d\lam d\tau \non\\
&\quad-\f{i\xi_1}{|\xi_1|} \f{1}{2\pi i} \int^t_0 \int_{\Ga}e^{\lam (t-\tau)}\f{|\xi_1|}{\lam} \int^\infty_0  e^{-\om y_2} \p_2 E_{|\xi_1|} [i\xi_1 \wh{f}_{1}+\p_2 \wh{f}_{2}] (\tau, y_2) dy_2  \f{1}{\om-|\xi_1| } e^{-|\xi_1|x_2}  d\lam d\tau\non\\
&\quad+\f{i\xi_1}{|\xi_1|} \f{1}{2\pi i} \int^t_0 \int_{\Ga}e^{\lam (t-\tau)}\f{|\xi_1|}{\lam} \int^\infty_0  e^{-\om (x_2+y_2)} \p_2 E_{|\xi_1|} [i\xi_1 \wh{f}_{1}+\p_2 \wh{f}_{2}] (\tau, y_2) dy_2  \f{1}{\om-|\xi_1| }   d\lam d\tau+\wh{g}_1\non\\
&:=\sum_{i=1}^7 N_i,
\end{align*}
where
\begin{align*}
\left\{
\begin{array}{l}
\wh{f}=\cF_{x_1} (-u\cdot\na u+b\cdot\na b)=i\xi_1 \cF_{x_1}(- u_1u+b_1 b)+\cF_{x_1}(\p_2(- u_2 u+b_2 b)),\\
\wh{g}= \cF_{x_1} (-u\cdot\na b+b\cdot\na u)=i\xi_1\cF_{x_1}(- u_1b+b_1 u)+\cF_{x_1}(\p_2(- u_2 b+b_2 u)).
\end{array}\right.
\end{align*}

\begin{proposition}\label{nonlinear}
We have
 \begin{align*}
E(t)\leq E_0+E(t)^2+\cE^2(t)+\int ^t_0 \cF^2(\tau) d\tau,
 \end{align*}
 where
\begin{align*}
 E_0&= \|(u_0, b_0)\|_{L^1\cap L^2\cap L^1_{x_1}L^2_{x_2}}+\|\na^{\f52+\delta} (u_0, b_0) \|_{L^1_{x_1}L^2_{x_2}},\non\\
 E(t)&= \lan t\ran^\f14 \|b_1\|_{L^2} +\lan t\ran^\f12\Big( \|u\|_{L^2}+ \|b_2\|_{L^2}+\|b_1\|_{L^\infty} \Big) +\lan t\ran^\f34 \Big(\|\p _1 b\|_{L^2}+\|\na u\|_{H^1}\Big) \non\\
 &\quad +\lan t \ran\Big(\|u\|_{L^\infty}+\|b_2\|_{L^\infty}+\|\p_1 u\|_{H^1}\Big)+\lan t \ran^{1+\delta}\|\p_1 u\|_{L^\infty},
\end{align*}
and $\cE(t), \cF(t)$  are defined in \eqref{cE}.
\end{proposition}
\begin{proof}
With the help of linear decay rate of $u_L, b_L$ in Section 2, and notice that
\begin{align*}
&\Big|\cF_{y_2}\Big(\chi E_{|\xi_1|}[i\xi_1 \wh{f}_1+\p_2 \wh{f}_2]\Big)\Big|\non\\
&=\f{1}{2|\xi_1|} \Big|\cF_{y_2} \Big(e^{-|\xi_1||y_2|}*\big(\chi ( i\xi_1 \wh{f}_1+\p_2 \wh{f}_2)\big)\Big)\Big|\non\\
&=\f{1}{2|\xi_1|} \Big|\cF_{y_2}\Big (e^{-|\xi_1||y_2|}\Big)\cF_{y_2} \Big(\chi ( i\xi_1 \wh{f}_1+\p_2 \wh{f}_2)\Big)\Big|\non\\
&=\f{1}{2|\xi_1|} \Big|\f{1}{|\xi_1|-i\xi_2}+\f{1}{|\xi_1|+i\xi_2}\Big|\Big|\cF_{y_2}\Big(\chi ( i\xi_1 \wh{f}_1+\p_2 \wh{f}_2)\Big)\Big|\non\\
&\lesssim \f{1}{|\xi_1|}\f{1}{|\xi|}\Big|\cF\Big(\chi \na f\Big)\Big|\non\\
&\lesssim \f{1}{|\xi_1|}\Big| \cF(f)\Big|,
\end{align*}
and
\begin{align*}
\Big|\cF_{y_2}\Big(\chi \p_2 E_{|\xi_1|}[i\xi_1 \wh{f}_1+\p_2 \wh{f}_2]\Big)\Big|=|\xi_2|\Big|\cF_{y_2}\Big(\chi  E_{|\xi_1|}[i\xi_1 \wh{f}_1+\p_2 \wh{f}_2]\Big)\Big|\lesssim \Big|\cF(f)\Big|,
\end{align*}
$M_i, N_i, \ i=1,3,4,6,7$ can be controlled similar as the whole space part, and $M_i, N_i, \ i=2,5$ are the different part coming from the half space domain. We only consider the weak decay case.

{\bf Step 1.} The  decay rate of $u_N$.

Using the same way with linear terms, we have the $L^2$ estimate for $M_{1, 3, 4, 6}$
\begin{align*}
&\int^t_0 \Big\||\xi| \f{|\xi_1|}{|\xi|^2} e^{-\f{|\xi_1|^2}{|\xi|^2} (t-\tau)} \cF\Big(u_i\, u+u_i\, b+b_i\, u+b_2\, b\Big)(\tau) \Big\|_{L^2_{\xi_1}L^2_{\xi_2}}d\tau\non\\
&\lesssim \int^t_0 \lan t-\tau\ran^{-\f12}\Big(\|u\|_{L^\infty}\|u\|_{L^2}+\|u\|_{L^\infty}\|b\|_{L^2}+\|b_2\|_{L^\infty}\|b\|_{L^2} \Big) d\tau\non\\
&\lesssim \int^t_0 \lan t-\tau\ran^{-\f12}\Big(\lan\tau\ran^{-1} \lan\tau\ran \|u\|_{L^\infty}\lan\tau\ran^{-\f12}\lan\tau\ran^{\f12} \|u\|_{L^2}+\lan\tau\ran^{-1}\lan \tau\ran \|u\|_{L^\infty} \lan \tau \ran^{-\f14}\lan \tau\ran^{\f14} \|b\|_{L^2}\non\\
&\qquad+\lan\tau\ran^{-1}\lan \tau\ran \|b_2\|_{L^\infty} \lan \tau \ran^{-\f14}\lan \tau\ran^{\f14} \|b\|_{L^2} \Big)d\tau\non\\
&\lesssim \int^t_0 \lan t-\tau\ran^{-\f12}\lan \tau\ran^{-\f54}  d\tau  \,E^2 (t) \non\\
&\lesssim \lan t\ran^{-\f12} E^2(t),
\end{align*}
where we use
\begin{align}\label{example}
& \int^t_0 \lan t-\tau\ran^{-\alpha}\lan \tau\ran^{-\beta}  d\tau \non\\
&\lesssim  \lan t\ran ^{-\alpha} \int^{\f{t}{2}}_0 \lan \tau\ran ^{-\beta} d\tau+\lan t\ran ^{-\beta} \int^t_{\f{t}{2}} \lan t-\tau\ran^{-\alpha} d\tau\non\\
&\lesssim  \lan t\ran ^{-\alpha} \Big(\int^1_0  d\tau+\int^{\f{t}{2}} _1 \tau ^{-\beta} d\tau\Big)+\lan t\ran ^{-\beta} \Big(\int^{t-1}_{\f{t}{2}} (t-\tau)^{-\alpha} d\tau+\int^t_{t-1}  d\tau\Big)\non\\
&\lesssim  \lan t\ran ^{-\alpha} \Big(1+
\left\{
\begin{array}{l}
t^{1-\beta},  \ \ \ \ \beta\neq 1 \\
\log t,  \ \ \ \ \beta=1
\end{array}\right.
\Big)+\lan t\ran ^{-\beta} \Big(\int^{\f{t}{2}} _1 \tau^{-\alpha} d\tau+1\Big)\non\\
&\lesssim  \lan t\ran ^{-\alpha} \times
\left\{
\begin{array}{l}
1, \ \ \ \ \beta>1\\
\log t,  \ \ \ \ \beta=1\\
t^{1-\beta},  \ \ \ \ \beta< 1 \\
\end{array}\right.
+ \lan t\ran ^{-\beta} \Big(
\left\{
\begin{array}{l}
t^{1-\alpha},  \ \ \ \ \alpha\neq 1 \\
\log t,  \ \ \ \ \alpha=1
\end{array}\right.
+1\Big)\non\\
&\lesssim  \lan t\ran ^{-\alpha} \times
\left\{
\begin{array}{l}
1, \ \ \ \ \beta>1\\
\log t,  \ \ \ \ \beta=1\\
t^{1-\beta},  \ \ \ \ \beta< 1 \\
\end{array}\right.
+\lan t\ran ^{-\beta} \times
\left\{
\begin{array}{l}
1, \ \ \ \ \alpha>1\\
\log t,  \ \ \ \ \alpha=1\\
t^{1-\alpha},  \ \ \ \ \alpha< 1 \\
\end{array}\right..
\end{align}
Similarly, we have
\begin{align*}
&\Big\|\int^t_0\f{|\xi_1|}{|\xi|^2} e^{-\f{|\xi_1|^2}{|\xi|^2} (t-\tau)}\cF\Big(\p_1(b_1 \, b)\Big)(\tau) d\tau\Big\|_{L^2_{\xi_1}L^2_{\xi_2}}\non\\
&\lesssim\int^t_0 \Big\| \f{|\xi_1|^2}{|\xi|^2} e^{-\f{|\xi_1|^2}{|\xi|^2} (t-\tau)} \cF(b_1\, b)(\tau) \Big\|_{L^2_{\xi_1}L^2_{\xi_2}} d\tau\non\\
&\lesssim\int^t_0  \lan t-\tau\ran^{-1} \|b_1\|_{L^\infty} \|b\|_{L^2} d\tau\non\\
&\lesssim \int^t_0 \lan t-\tau\ran^{-1}\lan \tau\ran^{-\f34}  d\tau  \,E^2(t) \non\\
&\lesssim \lan t\ran^{-\f12} E^2(t).
\end{align*}

The $L^2$ estimate for $M_{2,5}$ is obtained as
\begin{align*}
&\int^t_0 \Big\||\xi| \f{|\xi_1|^\f12}{|\xi|^2} e^{-\f{|\xi_1|^2}{|\xi|^2} (t-\tau)} \cF\Big(u_i\, u+u_i\, b_2+b_i\, u_2+b_2\, b\Big)(\tau) \Big\|_{L^2_{\xi_1}L^1_{\xi_2}}d\tau\non\\
&\lesssim\int^t_0 \lan t-\tau\ran ^{-\f14}\Big[ \sum_{j\geq 1} j^{-\f12} \Big\| e^{-\f{|\xi_1|^2}{j^2} (t-\tau)} \Big\|_{L^2_{\xi_1}} \Big\| \cF\Big(P_j (u_i\, u+u_i\, b_2+b_i\, u_2+b_2\, b)\Big)(\tau)\Big\|_{L^\infty_{\xi_1}L^2_{\xi_2}}\Big( \int_{|\xi_2|\leq j} d\xi_2\Big)  ^{\f12} \non\\
&\quad+ \sum_{0\leq j\leq 1} j^{-\f12} \Big\| e^{-\f{|\xi_1|^2}{j^2} (t-\tau)} \Big\|_{L^2_{\xi_1}} \Big\| \cF\Big(P_j (u_i\, u+u_i\, b_2+b_i\, u_2+b_2\, b)\Big)(\tau)\Big\|_{L^\infty_{\xi_1}L^\infty_{\xi_2}}\Big( \int_{|\xi_2|\leq j} d\xi_2\Big) \Big] d\tau\non\\
&\lesssim \int^t_0 \lan t-\tau\ran^{-\f12} \Big[\sum_{j\geq 1} j^{-\delta}
\Big\|j^{\f12+\delta} \cF\Big(P_j (u_i\, u+u_i\, b_2+b_i\, u_2+b_2\, b)\Big)\Big\|_{L^\infty_{\xi_1}L^2_{\xi_2}}\non\\
&\qquad+\sum_{0\leq j\leq 1} j
\Big\|  \cF\Big(P_j (u_i\, u+u_i\, b_2+b_i\, u_2+b_2\, b)\Big)\Big\|_{L^\infty_{x_1}L^\infty_{x_2}} \Big]d\tau\non\\
&\lesssim \int^t_0 \lan t-\tau\ran^{-\f12} \Big(\|\na^{\f12+\delta}(u_i\, u+u_i\, b_2+b_i\, u_2+b_2\, b)\|_{L^1_{x_1}L^2_{x_2}}+\|u_i\, u+u_i\, b_2+b_i\, u_2+b_2\, b\|_{L^1} \Big)d\tau\non\\
&\lesssim \int^t_0 \lan t-\tau\ran^{-\f12}\lan \tau\ran^{-1-\delta}  d\tau\Big( E^2(t)+\cE^2(t)\Big) \non\\
&\lesssim \lan t\ran^{-\f12} \Big( E^2(t)+\cE^2(t)\Big),
\end{align*}
where we use
\begin{align*}
\|\na^{\f12+\delta}b_2b\|_{L^1_{x_1}L^2_{x_2}}
&\leq \|\na^{\f12+\delta} b_2\|_{L^2}\|b\|_{L^2}^\f12\|\p_2 b\|_{L^2}^\f12\non\\
&\leq \|b_2\|_{L^2}^{\f12-\delta} \|\na b_2\|_{L^2}^{\f12+\delta}\|b\|_{L^2}^\f12\|\p_2 b\|_{L^2}^\f12,
\end{align*}
\begin{align*}
\|b_2\na^{\f12+\delta} b\|_{L^1_{x_1}L^2_{x_2}}
&\leq \|b_2\|_{L^2}^\f12\|\p_2 b_2\|_{L^2}^\f12 \|\na^{\f12+\delta} b\|_{L^2}\non\\
&\leq \|b_2\|_{L^2}^\f12\|\p_2 b_2\|_{L^2}^\f12 \| b\|_{L^2}^{\f12-\delta} \|\na b\|_{L^2}^{\f12+\delta}.
\end{align*}
Similarly, we have
\begin{align*}
&\int^t_0 \Big\| \f{|\xi_1|^\f32}{|\xi|^2} e^{-\f{|\xi_1|^2}{|\xi|^2} (t-\tau)} \cF(b_1\, b)(\tau) \Big\|_{L^2_{\xi_1}L^1_{\xi_2}} d\tau\non\\
&\lesssim\int^t_0 \lan t-\tau\ran ^{-\f34}\Big[ \sum_{j\geq 1} j^{-\f12} \Big\| e^{-\f{|\xi_1|^2}{j^2} (t-\tau)} \Big\|_{L^2_{\xi_1}} \Big\| \cF\Big(P_j (b_1\, b)\Big)(\tau)\Big\|_{L^\infty_{\xi_1}L^2_{\xi_2}}\Big( \int_{|\xi_2|\leq j} d\xi_2\Big)  ^{\f12} \non\\
&\quad+ \sum_{0\leq j\leq 1} j^{-\f12} \Big\| e^{-\f{|\xi_1|^2}{j^2} (t-\tau)} \Big\|_{L^\infty_{\xi_1}} \Big\| \cF\Big(P_j (b_1\, b)\Big)(\tau)\Big\|_{L^2_{\xi_1}L^2_{\xi_2}}\Big( \int_{|\xi_2|\leq j} d\xi_2\Big)^\f12 \Big] d\tau\non\\
&\lesssim \int^t_0 \lan t-\tau\ran^{-1} \|\na^{\f12+\delta}(b_1\, b)\|_{L^1_{x_1}L^2_{x_2}}+\int^t_0 \lan t-\tau\ran^{-\f34} \|b_1\, b\|_{L^2} d\tau\non\\
&\lesssim \Big(\int^t_0 \lan t-\tau\ran^{-1}\lan \tau\ran^{-\f12-\delta}  d\tau+\int^t_0 \lan t-\tau\ran^{-\f34}\lan \tau\ran^{-\f34}  d\tau\Big)\Big( E^2(t)+\cE^2(t)\Big) \non\\
&\lesssim \lan t\ran^{-\f12} \Big( E^2(t)+\cE^2(t)\Big).
\end{align*}

Since $\|e^{-|\xi_1|x_2}\|_{L^\infty_{x_2}}\lesssim 1$, the $L^\infty$ estimate of $u_N$ is obtained as
\begin{align*}
&\Big\|\int^t_0\f{|\xi_1|^2}{|\xi|^4} e^{-\f{|\xi_1|^2}{|\xi|^2} (t-\tau)}\cF\Big(u\cdot\na u+\p_2(b_2 \, b)\Big)(\tau) d\tau\Big\|_{L^1_{\xi_1}L^1_{\xi_2}}\non\\
&\lesssim\int^t_0 \Big\| \f{|\xi_1|^2}{|\xi|^3} e^{-\f{|\xi_1|^2}{|\xi|^2} (t-\tau)} \cF\Big(u_i\, u+b_2\, b\Big)(\tau)\Big\|_{L^1_{\xi_1}L^1_{\xi_2}} d\tau\non\\
&\lesssim\int^t_0 \lan t-\tau\ran ^{-1} \sum_{j\geq 0} j^{-1} \Big\| e^{-\f{|\xi_1|^2}{j^2} (t-\tau)} \Big\|_{L^2_{\xi_1}} \Big\| \cF\Big(P_j (u_i\, u+b_2\, b)\Big)(\tau)\Big\|_{L^2_{\xi_1}L^2_{\xi_2}}\Big( \int_{|\xi_2|\leq j} d\xi_2\Big)  ^{\f12}  d\tau\non\\
&\lesssim \int^t_0 \lan t-\tau\ran^{-\f54}\Big( \sum_{j\geq 1} j^{-\delta}
\Big\|j^{\delta} \cF\Big(P_j (u_i\, u +b_2\, b)\Big)\Big\|_{L^2_{\xi_1}L^2_{\xi_2}}+
  \sum_{0\leq j\leq 1} \Big\|  \cF\Big(P_j (u_i\, u +b_2\, b)\Big)\Big\|_{L^2_{\xi_1}L^2_{\xi_2}} \Big)d\tau\non\\
&\lesssim \int^t_0 \lan t-\tau\ran^{-\f54} \Big(\|\na^\delta u\|_{L^2}\|u\|_{L^\infty}+\|\na^\delta b_2\|_{L^2}\|b\|_{L^\infty} +\|b_2\|_{L^\infty}\|\na^\delta b\|_{L^2} +\|u\|_{L^\infty}\|u\|_{L^2}+\|b_2\|_{L^\infty}\| b\|_{L^2} \Big)d\tau\non\\
&\lesssim \int^t_0 \lan t-\tau\ran^{-\f54}\lan \tau\ran^{-1-\delta}  d\tau\Big( E^2(t)+\cE^2(t)\Big) \non\\
&\lesssim \lan t\ran^{-1} \Big( E^2(t)+\cE^2(t)\Big),
\end{align*}
where we use $|\xi_1|\leq |\xi|^2$ and
\begin{align*}
\|\na^\delta b_2\|_{L^2}\leq \|b_2\|_{L^2}^{1-\delta} \|\na b_2\|_{L^2}^\delta\leq \|b_2\|_{L^2}^{1-\delta} \|\p_1 b\|_{L^2}^\delta.
\end{align*}

Similarly, by \eqref{example}, we have
\begin{align}\label{p1}
&\Big\|\int^t_0\f{|\xi_1|^2}{|\xi|^4} e^{-\f{|\xi_1|^2}{|\xi|^2} (t-\tau)}\cF\Big(\p_1(b_1 \, b)\Big)(\tau) d\tau\Big\|_{L^1_{\xi_1}L^1_{\xi_2}}\non\\
&\lesssim\int^{\f{t}{2}}_0 \Big\|\f{|\xi_1|^3}{|\xi|^4} e^{-\f{|\xi_1|^2}{|\xi|^2} (t-\tau)} \cF\Big(b_1\, b\Big)(\tau) \Big\|_{L^1_{\xi_1}L^1_{\xi_2}}d\tau+\int^t_{\f{t}{2}} \Big\| \f{|\xi_1|^2}{|\xi|^4} e^{-\f{|\xi_1|^2}{|\xi|^2} (t-\tau)} \cF\Big(\p_1(b_1\, b)\Big)(\tau) \Big\|_{L^1_{\xi_1}L^1_{\xi_2}}d\tau,
\end{align}
where the first term can be estimated as
\begin{align*}
&\int^{\f{t}{2}}_0 \Big\|\f{|\xi_1|^3}{|\xi|^4} e^{-\f{|\xi_1|^2}{|\xi|^2} (t-\tau)} \cF\Big(b_1\, b\Big)(\tau) \Big\|_{L^1_{\xi_1}L^1_{\xi_2}}d\tau\non\\
&\lesssim\int^{\f{t}{2}}_0 \lan t-\tau\ran ^{-\f32} \sum_{j\geq 0}j^{-1}  \Big\| e^{-\f{|\xi_1|^2}{j^2} (t-\tau)} \Big\|_{L^1_{\xi_1}} \Big\| \cF\Big(P_j (b_1\, b)\Big)(\tau)\Big\|_{L^\infty_{\xi_1}L^\infty_{\xi_2}} \Big(\int_{|\xi_2|\leq j} d\xi_2\Big) d\tau\non\\
&\lesssim \int^{\f{t}{2}}_0 \lan t-\tau\ran^{-2} \Big(\sum_{j\geq 1}j^{-\delta}
\Big\| j^{1+\delta} \cF\Big(P_j (b_1\, b)\Big)\Big\|_{L^\infty_{\xi_1}L^\infty_{\xi_2}}+ \sum_{0\leq j\leq 1} j
\Big\| \cF\Big(P_j (b_1\, b)\Big)\Big\|_{L^\infty_{\xi_1}L^\infty_{\xi_2}}\Big)d\tau\non\\
&\lesssim \int^{\f{t}{2}}_0 \lan t-\tau\ran^{-2}\Big(\|\na^{1+\delta } b\|_{L^2}\|b\|_{L^2} +\|b\|_{L^2}^2\Big) d\tau\non\\
&\lesssim \int^{\f{t}{2}}_0 \lan t-\tau\ran^{-2}\lan \tau\ran^{-\f14}  d\tau  \Big( E^2(t)+\cE^2(t)\Big)  \non\\
&\lesssim \lan t\ran^{-1} \Big( E^2(t)+\cE^2(t)\Big),
\end{align*}
the second term can be estimated as
\begin{align*}
&\int^t_{\f{t}{2}} \Big\| \f{|\xi_1|^2}{|\xi|^4} e^{-\f{|\xi_1|^2}{|\xi|^2} (t-\tau)} \cF\Big(\p_1(b_1\, b)\Big)(\tau) \Big\|_{L^1_{\xi_1}L^1_{\xi_2}}d\tau\non\\
&\lesssim\int_{\f{t}{2}}^t \Big\| \sum_{j\geq 1} \f{|\xi_1|^2}{j^4} e^{-\f{|\xi_1|^2}{j^2} (t-\tau)} \cF\Big(P_j \p_1(b_1\, b)\Big)(\tau) \Big\|_{L^1_{\xi_1}L^1_{\xi_2}}d\tau\non\\
&\quad +\int_{\f{t}{2}}^t \Big\| \sum_{0\leq j\leq 1} \f{|\xi_1|}{j^2} e^{-\f{|\xi_1|^2}{j^2} (t-\tau)} \cF\Big(P_j \p_1(b_1\, b)\Big)(\tau) \Big\|_{L^1_{\xi_1}L^1_{\xi_2}}d\tau\non\\
&\lesssim\int_{\f{t}{2}}^t \lan t-\tau\ran ^{-1} \sum_{j\geq 1}
 j^{-2} \Big\| e^{-\f{|\xi_1|^2}{j^2} (t-\tau)} \Big\|_{L^2_{\xi_1}} \Big\| \cF\Big(P_j \p_1(b_1\, b)\Big)(\tau)\Big\|_{L^2_{\xi_1}L^\infty_{\xi_2}} \Big(\int_{|\xi_2|\leq j} d\xi_2\Big) d\tau\non\\
 &\quad +\int_{\f{t}{2}}^t \lan t-\tau\ran ^{-\f12} \sum_{0\leq j\leq 1}
 j^{-1} \Big\| e^{-\f{|\xi_1|^2}{j^2} (t-\tau)} \Big\|_{L^2_{\xi_1}} \Big\| \cF\Big(P_j \p_1(b_1\, b)\Big)(\tau)\Big\|_{L^2_{\xi_1}L^\infty_{\xi_2}} \Big(\int_{|\xi_2|\leq j} d\xi_2\Big) d\tau\non\\
&\lesssim \int_{\f{t}{2}}^t \lan t-\tau\ran^{-\f54} \sum_{j\geq 1} j^{-2}  j^{\f32}
\| b \p_1 b\|_{L^2_{x_1}L^1_{x_2}} d\tau+ \int^t_{\f{t}{2}} \lan t-\tau\ran^{-\f34} \sum_{0\leq j\leq 1} j^{-1} j ^{\f32}\| b \p_1 b\|_{L^2_{x_1}L^1_{x_2}}  d\tau\non\\
&\lesssim \int_{\f{t}{2}}^t\lan t-\tau\ran^{-\f34} \lan \tau \ran^{-\f54} d\tau\non\\
&\lesssim \lan t\ran^{-1} \Big( E^2(t)+\cE^2(t)\Big),
\end{align*}
with
\begin{align*}
\|b\, \p_1 b\|_{L^2_{x_1}L^1_{x_2}}&\leq \Big\| \|b\|_{L^\infty_{x_1}}\|\p_1 b\|_{L^2_{x_1}}\Big\|_{L^1_{x_2}}\non\\
&\leq \Big\| \|b\|_{L^2_{x_1}}^\f12\|\p_1 b\|_{L^2_{x_1}}^\f12 \|\p_1 b\|_{L^2_{x_1}}\Big\|_{L^1_{x_2}}\non\\
&\leq \|b\|_{L^2}^\f12\|\p_1 b\|_{L^2}^\f32.
\end{align*}
For the $\wh{g}_N$ part, we have
\begin{align*}
&\Big\|\int^t_0\f{|\xi_1|}{|\xi|^2} e^{-\f{|\xi_1|^2}{|\xi|^2} (t-\tau)}\cF\Big( b\cdot\na u \Big)(\tau) d\tau\Big\|_{L^1_{\xi_1}L^1_{\xi_2}}\non\\
&\lesssim\int^t_0 \lan t-\tau\ran ^{-\f12}\Big[ \sum_{j\geq 1} j^{-1} \Big\| e^{-\f{|\xi_1|^2}{j^2} (t-\tau)} \Big\|_{L^1_{\xi_1}} \Big\| \cF\Big(P_j (b\cdot\na u)\Big)(\tau)\Big\|_{L^\infty_{\xi_1}L^2_{\xi_2}}\Big( \int_{|\xi_2|\leq j} d\xi_2\Big)  ^{\f12} \non\\
&\quad+ \sum_{0\leq j\leq 1} j^{-1} \Big\| e^{-\f{|\xi_1|^2}{j^2} (t-\tau)} \Big\|_{L^1_{\xi_1}} \Big\| \cF\Big(P_j (b\cdot\na u)\Big)(\tau)\Big\|_{L^\infty_{\xi_1}L^\infty_{\xi_2}}\Big( \int_{|\xi_2|\leq j} d\xi_2\Big) \Big] d\tau\non\\
&\lesssim \int^t_0 \lan t-\tau\ran^{-1} \Big[\sum_{j\geq 1} j^{-\delta}
\Big\|j^{\f12+\delta} \cF\Big(P_j (b\cdot\na u)\Big)\Big\|_{L^\infty_{\xi_1}L^2_{\xi_2}}+\sum_{0\leq j\leq 1} j
\Big\|  \cF\Big(P_j (b\cdot\na u)\Big)\Big\|_{L^\infty_{x_1}L^\infty_{x_2}} \Big]d\tau\non\\
&\lesssim \int^t_0 \lan t-\tau\ran^{-1} \Big(\|\na^{\f12+\delta}(b\cdot\na u)\|_{L^1_{x_1}L^2_{x_2}}+\|b\cdot\na u\|_{L^1} \Big)d\tau\non\\
&\lesssim \int^t_0 \lan t-\tau\ran^{-1} \Big(\|\na^{\f12+\delta}b_1\p_1 u\|_{L^1_{x_1}L^2_{x_2}}+\|b_1 \na^{\f12+\delta}\p_1 u\|_{L^1_{x_1}L^2_{x_2}}+\|\na^{\f12+\delta}b_2\p_2 u\|_{L^1_{x_1}L^2_{x_2}}\non\\
&\quad+\|b_2 \na^{\f12+\delta}\p_2 u\|_{L^1_{x_1}L^2_{x_2}}+\|b_1\p_1u+b_2\p_2u\|_{L^1} \Big)d\tau\non\\
&\lesssim \int^t_0 \lan t-\tau\ran^{-1}\lan \tau\ran^{-1-\delta}  d\tau\Big( E^2(t)+\cE^2(t)\Big) \non\\
&\lesssim \lan t\ran^{-1} \Big( E^2(t)+\cE^2(t)\Big),
\end{align*}
where we use
\begin{align}\label{bnau}
\|\na^{\f12+\delta}b_1\p_1 u\|_{L^1_{x_1}L^2_{x_2}}
&\leq \|\na^{\f12+\delta} b_1\|_{L^2}\|\p_1 u\|_{L^2}^\f12\|\p_2 \p_1 u\|_{L^2}^\f12\non\\
&\leq \|b_1\|_{L^2}^{\f12-\delta} \|\na b_1\|_{L^2}^{\f12+\delta}\|\p_1 u\|_{L^2}^\f12\|\na \p_1 u\|_{L^2}^\f12,
\end{align}
\begin{align*}
\|b_1\na^{\f12+\delta}\p_1 u\|_{L^1_{x_1}L^2_{x_2}}
&\leq \|b_1\|_{L^2}^\f12\|\p_2 b_1\|_{L^2}^\f12 \|\na^{\f12+\delta} \p_1 u\|_{L^2}\non\\
&\leq \|b_1\|_{L^2}^\f12\|\p_2 b_1\|_{L^2}^\f12 \| \p_1 u\|_{L^2}^{\f12-\delta} \|\na \p_1 u\|_{L^2}^{\f12+\delta},
\end{align*}
\begin{align*}
\|\na^{\f12+\delta}b_2\p_2 u\|_{L^1_{x_1}L^2_{x_2}}
&\leq \|\na^{\f12+\delta} b_2\|_{L^2}\|\p_2 u\|_{L^2}^\f12\| \p^2_2 u\|_{L^2}^\f12\non\\
&\leq \|b_2\|_{L^2}^{\f12-\delta} \|\p_1 b_1\|_{L^2}^{\f12+\delta}\|\p_2 u\|_{L^2}^\f12\|\p_2^2 u\|_{L^2}^\f12,
\end{align*}
\begin{align*}
\|b_2\na^{\f12+\delta}\p_2 u\|_{L^1_{x_1}L^2_{x_2}}
&\leq \|b_2\|_{L^2}^\f12\|\p_1 b_1\|_{L^2}^\f12 \|\na^{\f12+\delta} \p_2 u\|_{L^2}\non\\
&\leq \|b_2\|_{L^2}^\f12\|\p_1 b_1\|_{L^2}^\f12 \| \p_2 u\|_{L^2}^{\f12-\delta} \|\na \p_2 u\|_{L^2}^{\f12+\delta},
\end{align*}
\beno
\|b_1\p_1u+b_2\p_2u\|_{L^1}\leq \|b_1\|_{L^2}\|\p_1 u\|_{L^2}+\|b_2\|_{L^2}\|\p_2 u\|_{L^2}.
\eeno
And
\begin{align*}
&\Big\|\int^t_0\f{|\xi_1|}{|\xi|^2} e^{-\f{|\xi_1|^2}{|\xi|^2} (t-\tau)}\cF\Big(u\cdot\na b_2 \Big)(\tau) d\tau\Big\|_{L^1_{\xi_1}L^1_{\xi_2}}\non\\
&\quad \leq \Big\|\int^t_0\f{|\xi_1|^2}{|\xi|^2} e^{-\f{|\xi_1|^2}{|\xi|^2} (t-\tau)}\cF(u\, b)(\tau) d\tau\Big\|_{L^1_{\xi_1}L^1_{\xi_2}},
\end{align*}
we can obtain  $\lan t\ran^{-1}$ decay similarly.

So, we only need to consider the trouble term
\begin{align*}
&\Big\|\int^t_0\f{|\xi_1|}{|\xi|^2} e^{-\f{|\xi_1|^2}{|\xi|^2} (t-\tau)}\cF\Big(u\cdot\na b_1 \Big)(\tau) d\tau\Big\|_{L^1_{\xi_1}L^1_{\xi_2}}\non\\
&\lesssim \int^t_0\Big\| \f{|\xi_1|^2}{|\xi|^2} e^{-\f{|\xi_1|^2}{|\xi|^2} (t-\tau)} \cF(u_1 b_1) \Big\|_{L^1_{\xi_1}L^1_{\xi_2}} d\tau+\int^t_0\Big\| \f{|\xi_1|}{|\xi|} e^{-\f{|\xi_1|^2}{|\xi|^2} (t-\tau)} \cF(u_2 b_1) \Big\|_{L^1_{\xi_1}L^1_{\xi_2}} d\tau\non\\
&\lesssim \int^t_0 \lan t-\tau\ran^{-1} \Big\|  \cF(u_1 b_1) \Big\|_{L^1_{\xi_1}L^1_{\xi_2}} d\tau+\int^\f{t}{2} _0  \Big\| \f{|\xi_1|}{|\xi|} e^{-\f{|\xi_1|^2}{|\xi|^2} (t-\tau)} \cF(u_2 b_1) \Big\|_{L^1_{\xi_1}L^1_{\xi_2}} d\tau+\int^t_{\f{t}{2}}\lan t-\tau\ran^{-\f12} \Big\|\cF(u_2 b_1) \Big\|_{L^1_{\xi_1}L^1_{\xi_2}} d\tau,
\end{align*}
and the most trouble term is  the second term.

We  divide $u_2=u_{2, <\lan \tau\ran^{-s_1}}+u_{2, >\lan \tau\ran^{-s_2}} +u_{2, \sim} $ to see that
\begin{align}\label{32delta}
&\int^\f{t}{2} _0\Big\| \f{|\xi_1|}{|\xi|} e^{-\f{|\xi_1|^2}{|\xi|^2} (t-\tau)} \cF\Big( (u_{2, <\lan \tau\ran^{-s_1}}+u_{2, >\lan \tau\ran^{-s_2}} ) b_1\Big) \Big\|_{L^1_{\xi_1}L^1_{\xi_2}} d\tau\non\\
&\lesssim\int^\f{t}{2} _0 \lan t-\tau\ran^{-\f12} \sum _{j\geq 0} \Big\|e^{-\f{|\xi_1|^2}{j^2} (t-\tau)}\Big\|_{L^1_{\xi_1}}\Big\|\cF\Big(P_j\big( (u_{2, <\lan \tau\ran^{-s_1}}+u_{2, >\lan \tau\ran^{-s_2}} ) b_1\big) \Big)\Big\|_{L^\infty_{\xi_1}L^2_{\xi_2}} \Big(\int_{|\xi_2|\leq j} d\xi_2\Big)^\f12 d\tau\non\\
&\lesssim\int^\f{t}{2} _0 \lan t-\tau\ran^{-1} \Big[\sum_{j\geq 1} j^{\f32}  j^{-\f32-\delta}
\Big\|  \na^{\f32+\delta} \Big( (u_{2, <\lan \tau\ran^{-s_1}}+u_{2, >\lan \tau\ran^{-s_2}} ) b_1\Big)  \Big\|_{L^1_{x_1}L^2_{x_2}} \non\\
&\quad+ \sum_{0\leq j\leq 1} j ^{\f32}\Big\|  (u_{2, <\lan \tau\ran^{-s_1}}+u_{2, >\lan \tau\ran^{-s_2}} ) b_1\Big \|_{L^1_{x_1}L^2_{x_2}} \Big] d\tau\non\\
&\lesssim\int^\f{t}{2} _0 \lan t-\tau\ran^{-1}
\Big\| \lan \na\ran^{\f32+\delta}  \Big( (u_{2, <\lan \tau\ran^{-s_1}}+u_{2, >\lan \tau\ran^{-s_2}} ) b_1\Big)\Big \|_{L^1_{x_1}L^2_{x_2}} d\tau\non\\
&\lesssim \int^\f{t}{2} _0 \lan t-\tau\ran^{-1}\lan \tau\ran^{-1-\delta}  d\tau\Big( E^2(t)+\cE^2(t)\Big) \non\\
&\lesssim \lan t\ran^{-1} \Big( E^2(t)+\cE^2(t)\Big),
\end{align}
where we use
\begin{align*}
\|\na^{\f32+\delta} (u_{2, <\lan \tau\ran^{-s_1}}+u_{2, >\lan \tau\ran^{-s_2}} ) \, b_1\|_{L^1_{x_1}L^2_{x_2}}
&\lesssim \|\na^{\f32+\delta} u_2\|_{L^2}\|b_1\|_{L^2}^\f12\|\p_2 b_1\|_{L^2}^\f12\non\\
&\lesssim \| \p_1 u\|_{L^2}^{\f12-\delta}\|\na \p_1 u\|_{L^2}^{\f12+\delta} \|b_1\|_{L^2}^\f12\|\p_2 b_1\|_{L^2}^\f12,
\end{align*}
and for $s_1>1$,
\begin{align*}
\|u_{2, <\lan \tau\ran^{-s_1}} \na^{\f32+\delta}  b_1\|_{L^1_{x_1}L^2_{x_2}}
&\lesssim \|u_{2, <\lan \tau\ran^{-s_1}}\|_{L^2}^\f12 \|\p_2 u_{2, <\lan \tau\ran^{-s_1}}\|_{L^2}^\f12 \|\na^{\f32+\delta}  b_1\|_{L^2}\non\\
&\leq \|u_2\|_{L^2}^\f12 \Big(\lan\tau\ran ^{-s_1}\|u_2 \|_{L^2}\Big)^\f12 \|b\|_{H^2}\non\\
&\lesssim \lan \tau\ran^{-\f{s_1}{2}-\f12} E(t) \cE(t),
\end{align*}
for $s_2<\f18-\f{\delta}{4}$,
\begin{align*}
&\|u_{2, >\lan \tau\ran^{-s_2}} \na^{\f32+\delta}  b_1\|_{L^1_{x_1}L^2_{x_2}}\non\\
&\lesssim \|u_{2, >\lan \tau\ran^{-s_2}}\|_{L^2}^\f12 \|\p_2 u_{2, >\lan \tau\ran^{-s_2}}\|_{L^2}^\f12\Big( \|\na^{\f32+\delta}  b_{1, <\lan \tau\ran^{\f18}}\|_{L^2}+ \|\na^{\f32+\delta}  b_{1, >\lan \tau\ran^{\f18}}\|_{L^2}\Big)\non\\
&\lesssim  \Big(\lan\tau\ran ^{s_2}\|\p_2 u_2 \|_{L^2}\Big)^\f12  \|\p_1 u_{1, >\lan \tau\ran^{-s_2}}\|_{L^2}^\f12\Big(\lan \tau\ran^{\f18(\f32+\delta)} \|b_1\|_{L^2}+\lan \tau\ran^{-\f18(\f12-\delta)} \|\na ^2 b_1\|_{L^2}\Big) \non\\
&\lesssim \lan \tau\ran^{-1+\f{s_2}{2}-(\f{1}{16}-\f{\delta}{8})} E(t) \cE(t).
\end{align*}
For the last case, we use the magnetic field equation $\eqref{eq:MHDT}_{2}$ to obtain
\begin{align}\label{middle}
&\int^\f{t}{2} _0\Big\| \f{|\xi_1|}{|\xi|} e^{-\f{|\xi_1|^2}{|\xi|^2} (t-\tau)} \cF( u_{2,\sim} \,b_1) \Big\|_{L^1_{\xi_1}L^1_{\xi_2}} d\tau\non\\
&=\int^\f{t}{2} _0\Big\| \f{|\xi_1|}{|\xi|} e^{-\f{|\xi_1|^2}{|\xi|^2} (t-\tau)} \cF\Big( (-\Delta)^{-1} \na^T\cdot \p_1 u_\sim  \,b_1\Big) \Big\|_{L^1_{\xi_1}L^1_{\xi_2}} d\tau\non\\
&=\int^\f{t}{2} _0\Big\| \f{|\xi_1|}{|\xi|} e^{-\f{|\xi_1|^2}{|\xi|^2} (t-\tau)} \cF\Big( (-\Delta)^{-1} \na^T\cdot (\p_\tau b+u\cdot\na b-b\cdot\na u)_{\sim}  \,b_1\Big) \Big\|_{L^1_{\xi_1}L^1_{\xi_2}} d\tau.
\end{align}
Since the nonlinear term will bring more decay, here we only estimate the following term:
\begin{align*}
&\int^\f{t}{2} _0\Big\| \f{|\xi_1|}{|\xi|} e^{-\f{|\xi_1|^2}{|\xi|^2} (t-\tau)} \cF\Big( (-\Delta)^{-1} \na^T\cdot \p_\tau b_{\sim}  \,b_1\Big) \Big\|_{L^1_{\xi_1}L^1_{\xi_2}} d\tau\non\\
&\lesssim\Big\| \f{|\xi_1|}{|\xi|} e^{-\f{|\xi_1|^2}{2|\xi|^2}t } \cF\Big( (-\Delta)^{-1} \na^T\cdot b_{\sim}  \,b_{1}\Big) \Big\|_{L^1_{\xi_1}L^1_{\xi_2}}+\Big\| \f{|\xi_1|}{|\xi|} e^{-\f{|\xi_1|^2}{|\xi|^2} t} \cF\Big( (-\Delta)^{-1} \na^T\cdot b_{\sim, 0}  \,b_{1, 0}\Big) \Big\|_{L^1_{\xi_1}L^1_{\xi_2}} \non\\
&\quad +\int^\f{t}{2} _0\Big\| \f{|\xi_1|^3}{|\xi|^3} e^{-\f{|\xi_1|^2}{|\xi|^2} (t-\tau)} \cF\Big( (-\Delta)^{-1} \na^T\cdot  b_{\sim}  \,b_1\Big) \Big\|_{L^1_{\xi_1}L^1_{\xi_2}} d\tau\non\\
&\quad +\int^\f{t}{2} _0\Big\| \f{|\xi_1|}{|\xi|} e^{-\f{|\xi_1|^2}{|\xi|^2} (t-\tau)} \cF\Big( (-\Delta)^{-1} \na^T\cdot  b_{\sim}  \,\p_\tau b_1\Big) \Big\|_{L^1_{\xi_1}L^1_{\xi_2}} d\tau\non\\
&\lesssim \lan t\ran^{-\f12} \Big\| \cF\Big( (-\Delta)^{-1} \na^T\cdot b_{\sim}  \,b_{1}\Big) \Big\|_{L^1_{\xi_1}L^1_{\xi_2}} +\lan t\ran^{-1} \Big\| \lan\na \ran^{\f32+\delta} \Big( (-\Delta)^{-1} \na^T\cdot b_{\sim, 0}  \,b_{1, 0}\Big) \Big\|_{L^1_{x_1}L^2_{x_2}}\non\\
&\quad +\int^{\f{t}{2}}_0 \lan t-\tau\ran^{-\f32} \Big\| \cF\Big( (-\Delta)^{-1} \na^T\cdot b_{\sim}  \,b_{1}\Big) \Big\|_{L^1_{\xi_1}L^1_{\xi_2}} d\tau\non\\
&\quad +\int^\f{t}{2} _0\Big\| \f{|\xi_1|}{|\xi|} e^{-\f{|\xi_1|^2}{|\xi|^2} (t-\tau)} \cF\Big( (-\Delta)^{-1} \na^T\cdot  b_{\sim}  \,(\p_1 u_1-u\cdot\na b+b\cdot\na u\Big) \Big\|_{L^1_{\xi_1}L^1_{\xi_2}} d\tau,
\end{align*}
where we use
\begin{align}\label{na-1}
\Big \|\cF\Big( (-\Delta)^{-1} \na^{T} \cdot b_{\sim}\Big)\Big\|_{L^1_{\xi_1}L^1_{\xi_2}}
&\lesssim\Big \|\cF\Big(|\na |^{-1} b_{2, \sim}\Big)\Big\|_{L^1_{\xi_1}L^1_{\xi_2}}\non\\
&\lesssim \sum_{\lan t \ran^{-s_2}\leq j\leq \lan t \ran^{-s_1} }\Big \|\cF\Big( P_j (|\na |^{-1} b_{2}) \Big)\Big\|_{L^1_{\xi_1}L^1_{\xi_2}}\non\\
&\lesssim \sum_{\lan t \ran^{-s_2}\leq j\leq \lan t \ran^{-s_1} } j^{-1} \Big \|\cF\Big( P_j  b_{2} \Big)\Big\|_{L^1_{\xi_1}L^1_{\xi_2}}\non\\
&\lesssim  \sum_{\lan t \ran^{-s_2}\leq j\leq \lan t \ran^{-s_1} }\|P_j b_2\|_{L^2}\non\\
&\lesssim  \|b_2\|_{L^2} \sum_{\lan \tau \ran^{-s_2}\leq j\leq \lan \tau\ran^{-s_1} }j^{\delta} j^{-\delta} \non\\
&\lesssim \lan t \ran^{-\f12+\delta} E(t),
\end{align}
and
\begin{align*}
&\int^\f{t}{2} _0\Big\| \f{|\xi_1|}{|\xi|} e^{-\f{|\xi_1|^2}{|\xi|^2} (t-\tau)} \cF\Big( (-\Delta)^{-1} \na^T\cdot  b_{\sim}  \p_1 u_1\Big) \Big\|_{L^1_{\xi_1}L^1_{\xi_2}} d\tau\non\\
&\lesssim \int^\f{t}{2} _0\Big\| \f{|\xi_1|^2}{|\xi|} e^{-\f{|\xi_1|^2}{|\xi|^2} (t-\tau)} \cF\Big( (-\Delta)^{-1} \na^T\cdot  b_{\sim}  \, u_1\Big) \Big\|_{L^1_{\xi_1}L^1_{\xi_2}} d\tau\non\\
&\quad +\int^\f{t}{2} _0\Big\| \f{|\xi_1|}{|\xi|} e^{-\f{|\xi_1|^2}{|\xi|^2} (t-\tau)} \cF\Big( (-\Delta)^{-1} \na^T\cdot  \p_1 b_{\sim}  \, u_1\Big) \Big\|_{L^1_{\xi_1}L^1_{\xi_2}} d\tau\non\\
&\lesssim \int^\f{t}{2} _0\Big\| \f{|\xi_1|^2}{|\xi|^2} e^{-\f{|\xi_1|^2}{|\xi|^2} (t-\tau)} \cF\Big( \na (-\Delta)^{-1} \na^T\cdot  b_{\sim}  \, u_1+ (-\Delta)^{-1} \na^T\cdot  b_{\sim}  \, \na u_1  \Big) \Big\|_{L^1_{\xi_1}L^1_{\xi_2}} d\tau\non\\
&\quad +\int^\f{t}{2} _0\Big\| \f{|\xi_1|}{|\xi|} e^{-\f{|\xi_1|^2}{|\xi|^2} (t-\tau)} \cF\Big( b_{2, \sim}  \, u_1\Big) \Big\|_{L^1_{\xi_1}L^1_{\xi_2}} d\tau\non\\
&\lesssim \int^\f{t}{2} _0 \lan t-\tau\ran^{-1} \Big\| \cF\Big( \na (-\Delta)^{-1} \na^T\cdot b_\sim \, u_1\Big)\Big\|_{L^1_{\xi_1}L^1_{\xi_2}} d\tau\non\\
&\quad+\int^{\f{t}{2}} _0  \lan t-\tau\ran^{-1} \sum_{j\geq 0} \Big\| e^{-\f{|\xi_1|^2}{j^2} (t-\tau)}  \cF\Big( P_j \big((-\Delta)^{-1} \na^T\cdot  b_{\sim}  \, \na u_1 \big) \Big)\Big\|_{L^1_{\xi_1}L^1_{\xi_2}}    d\tau\non\\
&\quad +\int^\f{t}{2} _0\Big\| \f{|\xi_1|}{|\xi|} e^{-\f{|\xi_1|^2}{|\xi|^2} (t-\tau)} \cF\Big( b_{2, \sim}  \, u_1\Big) \Big\|_{L^1_{\xi_1}L^1_{\xi_2}} d\tau;
\end{align*}
for  the above inequality, the second term can be estimated as
\begin{align*}
&\int^{\f{t}{2}} _0  \lan t-\tau\ran^{-1} \sum_{j\geq 0} \Big\| e^{-\f{|\xi_1|^2}{j^2} (t-\tau)}  \cF\Big( P_j \big((-\Delta)^{-1} \na^T\cdot  b_{\sim}  \, \na u_1 \big) \Big)\Big\|_{L^1_{\xi_1}L^1_{\xi_2}}    d\tau\non\\
&\lesssim \int^{\f{t}{2}} _0  \lan t-\tau\ran^{-1} \sum_{j\geq 0} \Big\| e^{-\f{|\xi_1|^2}{j^2} (t-\tau)}  \Big\|_{L^2_{\xi_1}} \Big\| \cF\Big( P_j \big((-\Delta)^{-1} \na^T\cdot  b_{\sim}  \, \na u_1 \big) \Big)\Big\|_{L^2_{\xi_1}L^2_{\xi_2}}    \Big( \int _{|\xi_2|\leq j}  d\xi_2\Big) ^\f12 d\tau\non\\
&\lesssim \int^{\f{t}{2}} _0  \lan t-\tau\ran^{-\f54} \sum_{j\geq 0}  j^\f12 \Big\| \cF\Big( P_j \big((-\Delta)^{-1} \na^T\cdot  b_{\sim}  \, \na u_1 \big) \Big)\Big\|_{L^2_{\xi_1}L^2_{\xi_2}}    \Big( \int _{|\xi_2|\leq j}  d\xi_2\Big) ^\f12 d\tau\non\\
&\lesssim \int^{\f{t}{2}} _0  \lan t-\tau\ran^{-\f54} \sum_{j\geq 0}  j \Big\| \cF\Big(P_j (-\Delta)^{-1} \na^T\cdot  b_{\sim}\Big)\Big\|_{L^1_{\xi_1}L^1_{\xi_2}}\Big\|\cF(\na u_1 )\Big\|_{L^2_{\xi_1}L^2_{\xi_2}}   d\tau\non\\
&\lesssim \int^{\f{t}{2}} _0  \lan t-\tau\ran^{-\f54} \Big(\sum_{j\geq 1}  j^{-\delta}  \Big\| \cF\Big(P_j  b_{\sim}\Big)\Big\|_{L^1_{\xi_1}L^1_{\xi_2}}\Big\|\cF(\na^{1+\delta} u_1 )\Big\|_{L^2_{\xi_1}L^2_{\xi_2}}  \non\\
&\qquad +
\sum_{0\leq j\leq 1}  j \Big\| \cF\Big((-\Delta)^{-1} \na^T\cdot  b_{\sim}\Big)\Big\|_{L^1_{\xi_1}L^1_{\xi_2}}\Big\|\cF(\na u_1 )\Big\|_{L^2_{\xi_1}L^2_{\xi_2}} \Big) d\tau\non\\
&\lesssim \int^\f{t}{2} _0 \lan t-\tau\ran^{-\f54}\lan \tau\ran^{-1-\delta}  d\tau\Big( E^2(t)+\cE^2(t)\Big) \non\\
&\lesssim \lan t\ran^{-1} \Big( E^2(t)+\cE^2(t)\Big),
\end{align*}
and the third term, like \eqref{32delta}, we have
\begin{align*}
&\int^\f{t}{2} _0\Big\| \f{|\xi_1|}{|\xi|} e^{-\f{|\xi_1|^2}{|\xi|^2} (t-\tau)} \cF\Big( b_{2, \sim}  \, u_1\Big) \Big\|_{L^1_{\xi_1}L^1_{\xi_2}} d\tau\non\\
&\lesssim  \int^{\f{t}{2}}_0 \lan t-\tau\ran^{-1}  \Big\|\lan   \na \ran ^{\f32+\delta} (b_{2, \sim}  u)\Big\|_{L^1_{x_1}L^2_{x_2}}  d\tau\non\\
&\lesssim  \int^{\f{t}{2}}_0 \lan t-\tau\ran^{-1}  \Big(\|\na^{\f32+\delta } b_2\|_{L^2}\|u\|_{L^2}^\f12\|\p_2 u\|_{L^2}^\f12+\|b_2\|_{L^2}^\f12\|\p_2b_2\|_{L^2}^\f12 \|\na^{\f32+\delta} u\|_{L^2}\Big)d\tau\non\\
&\lesssim  \int^{\f{t}{2}}_0 \lan t-\tau\ran^{-1}  \Big(\|\p_1 b\|_{L^2}^{\f34-\f{\delta}{2}}\|\p_1 \na^2 b \|_{L^2}^{\f14+\f{\delta}{2}}\|u\|_{L^2}^\f12\|\p_2 u\|_{L^2}^\f12+\|b_2\|_{L^2}^\f12\|\p_1 b_1\|_{L^2}^\f12 \|\na u\|_{L^2}^{\f34-\f{\delta}{2}}\|\na^3  u\|_{L^2}^{\f14+\f{\delta}{2}}\Big)d\tau\non\\
&\lesssim \int^\f{t}{2} _0 \lan t-\tau\ran^{-1}\lan \tau\ran^{-1-\delta}  d\tau\Big( E^2(t)+\cE^2(t)+\int ^t_0 \cF^2(\tau) d\tau\Big) \non\\
&\lesssim \lan t\ran^{-1} \Big( E^2(t)+\cE^2(t)+\int ^t_0 \cF^2(\tau) d\tau\Big).
\end{align*}

{\bf Step 2.} The decay rate of $b_N$

We have the $L^2$ estimate for $N_{1, 3, 4, 6,7}$ as
\begin{align*}
&\Big\|\int^t_0e^{-\f{|\xi_1|^2}{|\xi|^2} (t-\tau)}\cF\Big(u_1\p_1 b_1+b\cdot\na u_1\Big)(\tau) d\tau\Big\|_{L^2_{\xi_1}L^2_{\xi_2}}\non\\
&\lesssim\int^t_0 \sum_{j\geq 0} \Big\| e^{-\f{|\xi_1|^2}{j^2} (t-\tau)} \Big\|_{L^2_{\xi_1}} \Big\| \cF\Big(P_j (u_1\p_1 b_1+b\cdot\na u_1)\Big)(\tau)\Big\|_{L^\infty_{\xi_1}L^2_{\xi_2}} d\tau\non\\
&\lesssim \int^t_0 \lan t-\tau\ran^{-\f14} \Big(\sum_{0\leq j\leq 1} j^{\f12}
\Big\| \cF\Big(P_j (u_1\p_1 b_1+b\cdot\na u_1)\Big)\Big\|_{L^\infty_{\xi_1}L^2_{\xi_2}}\non\\
&\qquad+ \sum_{j\geq 1} j^{-\delta}
\Big\| j^{\f12+\delta} \cF\Big(P_j (u_1\p_1 b_1+b\cdot\na u_1)\Big)\Big\|_{L^\infty_{\xi_1}L^2_{\xi_2}} \Big)d\tau\non\\
&\lesssim \int^t_0 \lan t-\tau\ran^{-\f14} \Big\|\lan \na\ran ^{\f12+\delta}(u_1\p_1 b+b\cdot\na u)\Big\|_{L^1_{x_1}L^2_{x_2}}  d\tau \non\\
&\lesssim  \int^{t}_0\lan t-\tau\ran^{-\f14}  \lan \tau \ran ^{-1-\delta} d\tau \, \Big( E^2(t)+\cE^2(t)+\int ^t_0 \cF^2(\tau) d\tau\Big)\non\\
&\lesssim \lan t\ran^{-\f14} \Big( E^2(t)+\cE^2(t)+\int ^t_0 \cF^2(\tau) d\tau\Big),
\end{align*}
where we use \eqref{bnau} and
\begin{align}\label{h3}
&\Big\|\na^{\f12+\delta} u\, \p_1 b+ u\, \na^{\f12+\delta} \p_1 b\Big\|_{L^1_{x_1}L^2_{x_2}}\non\\
&\leq \|\na^{\f12+\delta} u\|_{L^2}^\f12\|\na ^{\f32+\delta} u\|_{L^2}^\f12 \|\p_1 b\|_{L^2}+\|u\|_{L^2}^\f12\|\na u\|_{L^2 }^\f12 \|\na ^{\f12+\delta} \p_1 b\|_{L^2}\non\\
&\leq \|u\|_{L^2}^{\f14-\f{\delta}{2}}  \|\na u\|_{L^2}^{\f58+\f{\delta}{4}} \|\na ^3 u\|_{L^2}^{\f18+\f{\delta}{4}} \|\p_1 b\|_{L^2}+ \|u\|_{L^2}^\f12\|\na u\|_{L^2}^\f12\|\p_1 b\|_{L^2}^{\f34-\f{\delta}{2}}\|\p_1 \na^2 b \|_{L^2}^{\f14+\f{\delta}{2}}.
\end{align}
For the term
\begin{align*}
&\Big\|\int^t_0e^{-\f{|\xi_1|^2}{|\xi|^2} (t-\tau)}\cF\Big(u_2\p_2 b_1\Big)(\tau) d\tau\Big\|_{L^2_{\xi_1}L^2_{\xi_2}}\non\\
&=\Big\|\int^t_0e^{-\f{|\xi_1|^2}{|\xi|^2} (t-\tau)}\cF\Big(\p_2 (u_2 b_1)\Big)(\tau) d\tau\Big\|_{L^2_{\xi_1}L^2_{\xi_2}}+\Big\|\int^t_0  e^{-\f{|\xi_1|^2}{|\xi|^2} (t-\tau)}\cF\Big(\p_1 u_1 b_1\Big)(\tau) d\tau\Big\|_{L^2_{\xi_1}L^2_{\xi_2}},
\end{align*}
we use
\begin{align*}
\Big\|\int^t_0e^{-\f{|\xi_1|^2}{|\xi|^2} (t-\tau)}\cF\Big(\p_1 u_1 b_1\Big)(\tau) d\tau\Big\|_{L^2_{\xi_1}L^2_{\xi_2}}
&\lesssim \int^t_0 \lan t-\tau\ran^{-\f14} \Big\|\lan \na\ran ^{\f12+\delta}(\p_1 u_1 b_1)\Big\|_{L^1_{x_1}L^2_{x_2}}  d\tau \non\\
&\lesssim  \int^t _0\lan t-\tau\ran^{-\f14}  \lan \tau \ran ^{-1-\delta} d\tau \, \Big( E^2(t)+\cE^2(t)\Big)\non\\
&\lesssim \lan t\ran^{-\f14} \Big( E^2(t)+\cE^2(t)\Big),
\end{align*}
and
\begin{align*}
&\Big\|\int^t_0e^{-\f{|\xi_1|^2}{|\xi|^2} (t-\tau)}\cF\Big(\p_2 (u_2 b_1)\Big)(\tau) d\tau\Big\|_{L^2_{\xi_1}L^2_{\xi_2}}\non\\
&\leq \Big\|\int^t_0 |\xi| e^{-\f{|\xi_1|^2}{|\xi|^2} (t-\tau)}\cF\Big(\big(u_{2, <\lan \tau\ran^{-s_1}}+u_{2, >\lan \tau\ran^{-s_2}} +u_{2, \sim}\big) b_1\Big)(\tau) d\tau\Big\|_{L^2_{\xi_1}L^2_{\xi_2}}.
\end{align*}
For the first two parts,  similarly to \eqref{32delta}, we have
\begin{align*}
&\Big\|\int^t_0 |\xi| e^{-\f{|\xi_1|^2}{|\xi|^2} (t-\tau)}\cF\Big(\big(u_{2, <\lan \tau\ran^{-s_1}}+u_{2, >\lan \tau\ran^{-s_2}} \big) b_1\Big)(\tau) d\tau\Big\|_{L^2_{\xi_1}L^2_{\xi_2}}\non\\
&\lesssim\int^t _0 \lan t-\tau\ran^{-\f14} \sum _{j\geq 0} j^{\f32}\Big\|\cF\Big(P_j\big( (u_{2, <\lan \tau\ran^{-s_1}}+u_{2, >\lan \tau\ran^{-s_2}} )b_1\big) \Big)\Big\|_{L^\infty_{\xi_1}L^2_{\xi_2}} d\tau\non\\
&\lesssim\int^t_0 \lan t-\tau\ran^{-\f14} \Big[\sum_{j\geq 1} j^{\f32}  j^{-\f32-\delta}
\Big\|  \na^{\f32+\delta} \Big( (u_{2, <\lan \tau\ran^{-s_1}}+u_{2, >\lan \tau\ran^{-s_2}} ) b_1\Big)  \Big\|_{L^1_{x_1}L^2_{x_2}} \non\\
&\quad+ \sum_{0\leq j\leq 1} j ^{\f32}\Big\|  (u_{2, <\lan \tau\ran^{-s_1}}+u_{2, >\lan \tau\ran^{-s_2}} )  b_1\Big \|_{L^1_{x_1}L^2_{x_2}} \Big] d\tau\non\\
&\lesssim\int^t_0 \lan t-\tau\ran^{-\f14}
\Big\| \lan \na\ran^{\f32+\delta}  \Big( (u_{2, <\lan \tau\ran^{-s_1}}+u_{2, >\lan \tau\ran^{-s_2}} ) b_1\Big)\Big \|_{L^1_{x_1}L^2_{x_2}} d\tau\non\\
&\lesssim \int^t_0 \lan t-\tau\ran^{-\f14}\lan \tau\ran^{-1-\delta}  d\tau\Big( E^2(t)+\cE^2(t)\Big) \non\\
&\lesssim \lan t\ran^{-\f14} \Big( E^2(t)+\cE^2(t)\Big).
\end{align*}
For the last part, similarly to  \eqref{middle}, we have
\begin{align*}
&\Big\|\int^t_0 |\xi| e^{-\f{|\xi_1|^2}{|\xi|^2} (t-\tau)}\cF\Big(u_{2, \sim}  b_1\Big)(\tau) d\tau\Big\|_{L^2_{\xi_1}L^2_{\xi_2}}\non\\
&=\int^t _0\Big\| |\xi| e^{-\f{|\xi_1|^2}{|\xi|^2} (t-\tau)} \cF\Big( (-\Delta)^{-1} \na^T\cdot \p_1 u_\sim  \,b_1\Big) \Big\|_{L^2_{\xi_1}L^2_{\xi_2}} d\tau\non\\
&=\int^t _0\Big\||\xi|  e^{-\f{|\xi_1|^2}{|\xi|^2} (t-\tau)} \cF\Big( (-\Delta)^{-1} \na^T\cdot (\p_\tau b+u\cdot\na b-b\cdot\na u)_{\sim}  \,b_1\Big) \Big\|_{L^2_{\xi_1}L^2_{\xi_2}} d\tau;
\end{align*}
 we only estimate the most trouble term
\begin{align*}
&\int^t_0\Big\| |\xi| e^{-\f{|\xi_1|^2}{|\xi|^2} (t-\tau)} \cF\Big( (-\Delta)^{-1} \na^T\cdot \p_\tau b_{\sim}  \,b_1\Big) \Big\|_{L^2_{\xi_1}L^2_{\xi_2}} d\tau\non\\
&\lesssim\Big\||\xi| \cF\Big( (-\Delta)^{-1} \na^T\cdot b_{\sim}  \,b_{1}\Big) \Big\|_{L^2_{\xi_1}L^2_{\xi_2}}+\Big\| |\xi|  e^{-\f{|\xi_1|^2}{|\xi|^2} t} \cF\Big( (-\Delta)^{-1} \na^T\cdot b_{\sim, 0}  \,b_{1, 0}\Big) \Big\|_{L^2_{\xi_1}L^2_{\xi_2}} \non\\
&\quad +\int^t_0\Big\| \f{|\xi_1|^2}{|\xi|}  e^{-\f{|\xi_1|^2}{|\xi|^2} (t-\tau)} \cF\Big( (-\Delta)^{-1} \na^T\cdot  b_{\sim}  \,b_1\Big) \Big\|_{L^2_{\xi_1}L^2_{\xi_2}} d\tau\non\\
&\quad +\int^t _0\Big\| |\xi| e^{-\f{|\xi_1|^2}{|\xi|^2} (t-\tau)} \cF\Big( (-\Delta)^{-1} \na^T\cdot  b_{\sim}  \,\p_\tau b_1\Big) \Big\|_{L^2_{\xi_1}L^2_{\xi_2}} d\tau\non\\
&\lesssim  \Big\| \na \Big( (-\Delta)^{-1} \na^T\cdot b_{\sim}  \,b_{1}\Big) \Big\|_{L^2} + \lan t\ran^{-\f14} \Big\| \lan\na \ran^{\f32+\delta} \Big( (-\Delta)^{-1} \na^T\cdot b_{\sim, 0}  \,b_{1, 0}\Big) \Big\|_{L^1_{x_1}L^2_{x_2}}\non\\
&\quad +\int^t_0 \lan t-\tau\ran^{-1} \Big\| \na \Big( (-\Delta)^{-1} \na^T\cdot b_{\sim}  \,b_{1}\Big) \Big\|_{L^2} d\tau\non\\
&\quad +\int^t_0\Big\| |\xi| e^{-\f{|\xi_1|^2}{|\xi|^2} (t-\tau)} \cF\Big( (-\Delta)^{-1} \na^T\cdot  b_{\sim}  \,(\p_1 u_1-u\cdot\na b+b\cdot\na u)\Big) \Big\|_{L^2_{\xi_1}L^2_{\xi_2}} d\tau,
\end{align*}
where we use \eqref{na-1} to have
\begin{align*}
&\int^t_0 \lan t-\tau\ran^{-1} \Big\| \na \Big( (-\Delta)^{-1} \na^T\cdot b_{\sim}  \,b_{1}\Big) \Big\|_{L^2} d\tau\non\\
&\leq \int^t_0 \lan t-\tau\ran^{-1}\Big( \|b_{\sim}\|_{L^2}\|b_1\|_{L^\infty} +\|(-\Delta)^{-1} \na^{T} \cdot b_{\sim} \|_{L^\infty}\|\na b_1\|_{L^2}\Big)  d\tau\non\\
&\lesssim \lan t\ran^{-\f14} \Big( E^2(t)+\cE^2(t)\Big);
\end{align*}
and  similarly to  \eqref{na-1}, \eqref{h3}, we have
\begin{align*}
&\int^t_0\Big\| |\xi| e^{-\f{|\xi_1|^2}{|\xi|^2} (t-\tau)} \cF\Big( (-\Delta)^{-1} \na^T\cdot  b_{\sim}  \p_1 u_1\Big) \Big\|_{L^2_{\xi_1}L^2_{\xi_2}} d\tau\non\\
&\lesssim \int^t _0\Big\||\xi_1||\xi| e^{-\f{|\xi_1|^2}{|\xi|^2} (t-\tau)} \cF\Big( (-\Delta)^{-1} \na^T\cdot  b_{\sim}  \, u_1\Big) \Big\|_{L^2_{\xi_1}L^2_{\xi_2}} d\tau\non\\
&\quad +\int^t_0\Big\| |\xi| e^{-\f{|\xi_1|^2}{|\xi|^2} (t-\tau)} \cF\Big( (-\Delta)^{-1} \na^T\cdot  \p_1 b_{\sim}  \, u_1\Big) \Big\|_{L^2_{\xi_1}L^2_{\xi_2}} d\tau\non\\
&\lesssim \int^t _0 \lan t-\tau\ran^{-\f12} \Big\|  \na^2\Big(  (-\Delta)^{-1} \na^T\cdot  b_{\sim}  \, u_1 \Big) \Big\|_{L^2_{x_1}L^2_{x_2}} d\tau+\int^t _0  \lan t-\tau\ran^{-\f14} \Big\|\lan \na \ran^{\f32+\delta} \Big(b_{2, \sim} u_1\Big)\Big\|_{L^1_{x_1}L^2_{x_2}}    d\tau \non\\
&\lesssim \int^t_0 \lan t-\tau\ran^{-\f12} \Big\|\na b_\sim\,  u_1+ (-\Delta)^{-1} \na^T\cdot b_\sim \, \na ^2 u_1\Big\|_{L^2_{x_1}L^2_{x_2}} d\tau\non\\
&\quad+\int^t _0  \lan t-\tau\ran^{-\f14} \Big\|\na^{\f12+\delta} \p_1 b_{\sim }\na u_1+b_{2, \sim} \na^{\f32+\delta} u_1 \Big\|_{L^1_{x_1}L^2_{x_2}}        d\tau\non\\
&\lesssim \int^t_0 \lan t-\tau\ran^{-\f12} \Big( \|\na b_\sim\|_{L^2} \|  u_1\|_{L^\infty}+ \|(-\Delta)^{-1} \na^T\cdot b_\sim\|_{L^\infty}  \|\na ^2 u_1\|_{L^2} \Big) d\tau\non\\
&\quad+\int^t _0  \lan t-\tau\ran^{-\f14} \Big( \|\na^{\f12+\delta} \p_1 b_{\sim }\|_{L^2} \|\na u_1\|_{L^2}^\f12\|\na ^2 u_1\|_{L^2}^\f12+\|b_{2, \sim}\|_{L^2}^\f12\|\p_1 b\|_{L^2}^\f12 \| \na^{\f32+\delta} u_1 \|_{L^2}  \Big)  d\tau\non\\
&\lesssim \int^t_0 \lan t-\tau\ran^{-\f12} \Big(\|b\|_{L^2}^\f12\|\na^2 b\|_{L^2}^\f12\|u\|_{L^\infty}+\|(-\Delta)^{-1}\na^{T}\cdot\na b_\sim\|_{L^\infty} \|\na u\|_{L^2}^\f12\|\na^3 u\|_{L^2}^\f12\Big) d\tau\non\\
&\quad +\int^t_0 \lan t-\tau\ran^{-\f14}\Big(\|\p_1 b\|_{L^2}^{\f12-\delta}\|\na \p_1 b\|_{L^2}^{\f12+\delta}\|\na u\|_{L^2}^\f12\|\na^2 u\|_{L^2}^\f12+\|b_2\|_{L^2}^\f12\|\p_1 b\|_{L^2}^\f12 \|\na u\|_{L^2}^{\f34-\f{\delta}{2}}\|\na^3 u\|_{L^2}^{\f14+\f{\delta}{2}} \Big)d\tau \non\\
&\lesssim \lan t\ran^{-\f14} \Big( E^2(t)+\cE^2(t)+\int ^t_0 \cF^2(\tau) d\tau\Big).
\end{align*}
Now we consider the $L^2$ estimate for $N_{2,5}$. Since
\begin{align*}
\div (-u\cdot\na b+b\cdot\na u)=0\quad\Rightarrow\quad \p_2(-u\cdot\na b_2+b\cdot\na u_2)=\p_1(u\cdot\na b_1-b\cdot\na u_1),
\end{align*}
similarly to \eqref{bnau}, \eqref{32delta} and \eqref{middle}, we have
\begin{align}\label{divN}
&\Big\|\int^t_0e^{-\f{|\xi_1|^2}{|\xi|^2} (t-\tau)}|\xi_1|^{-\f12}\cF\Big((-u\cdot\na b_2+b\cdot\na u_2)\Big)(\tau) d\tau\Big\|_{L^2_{\xi_1}L^1_{\xi_2}}\non\\
&=\Big\|\int^t_0e^{-\f{|\xi_1|^2}{|\xi|^2} (t-\tau)}\f{|\xi_1|^{-\f12}}{|\xi|}\cF\Big(\na(-u\cdot\na b_2+b\cdot\na u_2)\Big)(\tau) d\tau\Big\|_{L^2_{\xi_1}L^1_{\xi_2}}\non\\
&=\Big\|\int^t_0e^{-\f{|\xi_1|^2}{|\xi|^2} (t-\tau)}\f{|\xi_1|^{\f12}}{|\xi|}\cF\Big(u\cdot\na b-b\cdot\na u\Big)(\tau) d\tau\Big\|_{L^2_{\xi_1}L^1_{\xi_2}}\non\\
&\lesssim \int^{t}_0\lan t-\tau\ran^{-\f14}\sum_{j\geq 0}\Big\|j^{-\f12}  e^{-\f{|\xi_1|^2}{j^2} (t-\tau)} \cF\Big(u\cdot\na b-b\cdot\na u\Big)(\tau)  \Big\|_{L^2_{\xi_1}L^1_{\xi_2}}d\tau\non\\
&\lesssim    \int^{t}_0\lan t-\tau\ran^{-\f14} \sum_{j\geq 0} j^{-\f12}\Big\| e^{-\f{|\xi_1|^2}{j^2} (t-\tau)}\Big\|_{L^2_{\xi_1}} \Big\|\cF\Big(u\cdot\na b-b\cdot\na u\Big)\Big\|_{L^\infty_{\xi_1}L^2_{\xi_2}} \Big(\int_{|\xi_2|\leq j} d \xi_2\Big)^\f12\non\\
&\lesssim  \int^{t}_0\lan t-\tau\ran^{-\f12}  \Big\|\lan \na \ran ^{\f12+\delta} (u\cdot\na b-b\cdot\na u) \Big\|_{L^1_{x_1}L^2_{x_2}}d\tau \non\\
&\lesssim  \int^{t}_0\lan t-\tau\ran^{-\f12}  \Big\|\lan \na \ran ^{\f32+\delta}\Big( \big(u_{2, <\lan \tau\ran^{-s_1}}+u_{2, >\lan \tau\ran^{-s_2}} +u_{2, \sim}\big) b\Big) +\lan \na \ran ^{\f12+\delta}(u_1\p_1b+b\cdot\na u) \Big\|_{L^1_{x_1}L^2_{x_2}}d\tau \non\\
&\lesssim  \int^{t}_0\lan t-\tau\ran^{-\f12} \lan \tau \ran ^{-1-\delta} d\tau \Big( E^2(t)+\cE^2(t)+\int ^t_0 \cF^2(\tau) d\tau\Big) \non\\
&\lesssim \lan t\ran^{-\f12} \Big( E^2(t)+\cE^2(t)+\int ^t_0 \cF^2(\tau) d\tau\Big).
\end{align}

The $L^\infty$ estimate of $b_1$ is obtained as
\begin{align}\label{b1infty}
&\Big\| \int^t_0e^{-\f{|\xi_1|^2}{|\xi|^2} (t-\tau)}\cF\Big(b\cdot\na u_1\Big)(\tau) d\tau\Big\|_{L^1_{\xi_1}L^1_{\xi_2}}\non\\
&= \Big\| \int^t_0e^{-\f{|\xi_1|^2}{|\xi|^2} (t-\tau)}\cF\Big(b_1\p_1 u_1+b_2\p_2 u_1\Big)(\tau) d\tau\Big\|_{L^1_{\xi_1}L^1_{\xi_2}}.
\end{align}
We use the magnetic field equation $\eqref{eq:MHDT}_2$ to estimate the first part of \eqref{b1infty},
\begin{align*}
& \Big\| \int^{\f{t}{2}}_0e^{-\f{|\xi_1|^2}{|\xi|^2} (t-\tau)}\cF\Big(b_1\p_1 u_1\Big)(\tau) d\tau\Big\|_{L^1_{\xi_1}L^1_{\xi_2}}\non\\
&= \Big\| \int^{\f{t}{2}}_0e^{-\f{|\xi_1|^2}{|\xi|^2} (t-\tau)}\cF\Big(b_1(\p_{\tau} b_1+u\cdot\na b_1-b\cdot\na u_1)\Big)(\tau) d\tau\Big\|_{L^1_{\xi_1}L^1_{\xi_2}}\non\\
&\lesssim\Big\|e^{-\f{|\xi_1|^2}{2|\xi|^2} t}\cF (b_1^2)\Big\|_{L^1_{\xi_1}L^1_{\xi_2}}+\Big\|e^{-\f{|\xi_1|^2}{|\xi|^2} t}\cF (b_{1, 0}^2)\Big\|_{L^1_{\xi_1}L^1_{\xi_2}}+\int^{\f{t}{2}}_0 \Big\|\f{|\xi_1|^2}{|\xi|^2}e^{-\f{|\xi_1|^2}{|\xi|^2} ( t-\tau)}\cF (b_1^2)\Big\|_{L^1_{\xi_1}L^1_{\xi_2}} d\tau \non\\
&+\Big\| \int^{\f{t}{2}}_0e^{-\f{|\xi_1|^2}{|\xi|^2} (t-\tau)}\cF\Big(b_1(u\cdot\na b_1-b\cdot\na u_1)\Big)(\tau) d\tau\Big\|_{L^1_{\xi_1}L^1_{\xi_2}},
\end{align*}
there are no difficulties here.

For the second part of \eqref{b1infty}, we have
\begin{align*}
&\Big\| \int^t_0e^{-\f{|\xi_1|^2}{|\xi|^2} (t-\tau)}\cF\Big(b_2\p_2 u_1\Big)(\tau) d\tau\Big\|_{L^1_{\xi_1}L^1_{\xi_2}}\non\\
&\lesssim\int^t_0 \sum_{j\geq 0} \Big\| e^{-\f{|\xi_1|^2}{j^2} (t-\tau)} \Big\|_{L^1_{\xi_1}} \Big\| \cF\Big(P_j (b_2\p_2 u_1)\Big)(\tau)\Big\|_{L^\infty_{\xi_1}L^2_{\xi_2}}  \Big(\int_{|\xi_2|\leq  j} d\xi_2\Big)^\f12 d\tau \non\\
&\lesssim \int^t_0 \lan t-\tau\ran^{-\f12} \Big(\sum_{0\leq j\leq 1} j^{\f32}
\Big\| \cF\Big(P_j (b_2\p_2 u_1)\Big)\Big\|_{L^\infty_{\xi_1}L^2_{\xi_2}}+ \sum_{j\geq 1} j^{-\delta}
\Big\| j^{\f32+\delta} \cF\Big(P_j (b_2\p_2 u_1)\Big)\Big\|_{L^\infty_{\xi_1}L^2_{\xi_2}} \Big)d\tau\non\\
&\lesssim \int^t_0 \lan t-\tau\ran^{-\f12} \Big\|\lan \na\ran ^{\f32+\delta}(b_2\p_2 u_1)\Big\|_{L^1_{x_1}L^2_{x_2}} d\tau \non\\
&\lesssim  \int^{t}_0\lan t-\tau\ran^{-\f12}  \lan \tau \ran ^{-1-\delta} d\tau \, \Big( E^2(t)+\cE^2(t)+\int ^t_0 \cF^2(\tau) d\tau\Big)\non\\
&\lesssim \lan t\ran^{-\f12} \Big( E^2(t)+\cE^2(t)+\int ^t_0 \cF^2(\tau) d\tau\Big),
\end{align*}
where  we use
\begin{align*}
\|\na^{\f32+\delta} b_2 \,\p_2 u_1\|_{L^1_{x_1}L^2_{x_2}}
&\lesssim \|\na^{\f32+\delta} b_2\|_{L^2}\|\p_2 u_1\|_{L^2}^\f12\|\p_2^2 u_1\|_{L^2}^\f12\non\\
&\lesssim \| \p_1 b\|_{L^2}^{\f12-\delta}\|\na \p_1b\|_{L^2}^{\f12+\delta}\|\p_2 u_1\|_{L^2}^\f12\|\p_2^2 u_1\|_{L^2}^\f12,
\end{align*}
\begin{align*}
\|b_2\na^{\f32+\delta} \p_2 u_1\|_{L^1_{x_1}L^2_{x_2}}
&\lesssim\|b_2\|_{L^2}^\f12\|\p_1 b_1\|_{L^2}^\f12\|\na^{\f32+\delta} \p_2 u\|_{L^2}\non\\
&\lesssim\|b_2\|_{L^2}^\f12\|\p_1 b_1\|_{L^2}^\f12\|\na^2 u\|_{L^2}^{\f34-\f{\delta}{2}}\|\na^4 u\|_{L^2}^{\f14+\f{\delta}{2}}.
\end{align*}

For
\begin{align}\label{b1infty2}
&\Big\|\int^t_0 e^{-\f{|\xi_1|^2}{|\xi|^2} (t-\tau)}\cF\Big(u\cdot\na b_1 \Big)(\tau) d\tau\Big\|_{L^1_{\xi_1}L^1_{\xi_2}}\non\\
&\lesssim \int^t_0\Big\| \f{|\xi_1|}{|\xi|} e^{-\f{|\xi_1|^2}{|\xi|^2} (t-\tau)} \cF\Big(\na(u_1 b_1)\Big)  \Big\|_{L^1_{\xi_1}L^1_{\xi_2}} d\tau+\int^t_0\Big\||\xi|  e^{-\f{|\xi_1|^2}{|\xi|^2} (t-\tau)} \cF(u_2  b_1) \Big\|_{L^1_{\xi_1}L^1_{\xi_2}} d\tau,
\end{align}

the trouble term is the second term, similarly to \eqref{32delta}, we divided it as
\begin{align*}
&\int^t _0\Big\||\xi|  e^{-\f{|\xi_1|^2}{|\xi|^2} (t-\tau)} \cF\Big( (u_{2, <\lan \tau\ran^{-s_1}}+u_{2, >\lan \tau\ran^{-s_3}} ) b_1\Big) \Big\|_{L^1_{\xi_1}L^1_{\xi_2}} d\tau\non\\
&\lesssim\int^t _0  \sum _{j\geq 0}  j \Big\|e^{-\f{|\xi_1|^2}{j^2} (t-\tau)}\Big\|_{L^1_{\xi_1}}\Big\|\cF\Big(P_j\big( (u_{2, <\lan \tau\ran^{-s_1}}+u_{2, >\lan \tau\ran^{-s_3}} ) b_1\big) \Big)\Big\|_{L^\infty_{\xi_1}L^2_{\xi_2}} \Big(\int_{|\xi_2|\leq j} d\xi_2\Big)^\f12 d\tau\non\\
&\lesssim\int^t _0 \lan t-\tau\ran^{-\f12} \Big[\sum_{j\geq 1} j^{\f52}  j^{-\f52-\delta}
\Big\|  \na^{\f52+\delta} \Big( (u_{2, <\lan \tau\ran^{-s_1}}+u_{2, >\lan \tau\ran^{-s_3}} ) b_1\Big)  \Big\|_{L^1_{x_1}L^2_{x_2}} \non\\
&\quad+ \sum_{0\leq j\leq 1} j ^{\f52}\Big\|  (u_{2, <\lan \tau\ran^{-s_1}}+u_{2, >\lan \tau\ran^{-s_3}} ) b_1\Big \|_{L^1_{x_1}L^2_{x_2}} \Big] d\tau\non\\
&\lesssim\int^t_0 \lan t-\tau\ran^{-\f12}
\Big\| \lan \na\ran^{\f52+\delta}  \Big( (u_{2, <\lan \tau\ran^{-s_1}}+u_{2, >\lan \tau\ran^{-s_3}} ) b_1\Big)\Big \|_{L^1_{x_1}L^2_{x_2}} d\tau\non\\
&\lesssim \int^t_0 \lan t-\tau\ran^{-\f12}\lan \tau\ran^{-1-\delta}  d\tau\Big( E^2(t)+\cE^2(t)\Big) \non\\
&\lesssim \lan t\ran^{-\f12} \Big( E^2(t)+\cE^2(t)\Big),
\end{align*}
with ($s_3<\f{1}{12}-\f{13}{6}\delta$)
\begin{align*}
&\|u_{2, >\lan \tau\ran^{-s_3}} \na^{\f52+\delta}  b_1\|_{L^1_{x_1}L^2_{x_2}}\non\\
&\lesssim \|u_{2, >\lan \tau\ran^{-s_3}}\|_{L^2}^\f12 \|\p_2 u_{2, >\lan \tau\ran^{-s_3}}\|_{L^2}^\f12\Big( \|\na^{\f52+\delta}  b_{1, <\lan \tau\ran^{\f{1}{12}}}\|_{L^2}+ \|\na^{\f52+\delta}  b_{1, >\lan \tau\ran^{\f{1}{12}}}\|_{L^2}\Big)\non\\
&\lesssim  \Big(\lan\tau\ran ^{s_3}\|\p_2 u_2 \|_{L^2}\Big)^\f12  \|\p_1 u_{1, >\lan \tau\ran^{-s_3}}\|_{L^2}^\f12\Big(\lan \tau\ran^{\f{1}{12}(\f52+\delta)} \|b_1\|_{L^2}+\lan \tau\ran^{-\f{1}{12} (\f12-\delta)} \|\na ^3 b_1\|_{L^2}\Big) \non\\
&\lesssim \lan \tau\ran^{-1+\f{s_3}{2}-\f{1}{12}(\f12-\delta)} E(t) \cE(t).
\end{align*}

And we claim that
\beno
 \int^t_0\Big\||\xi|  e^{-\f{|\xi_1|^2}{|\xi|^2} (t-\tau)} \cF\Big( u_{2,\sim} b_1\Big) \Big\|_{L^1_{\xi_1}L^1_{\xi_2}} d\tau\lesssim \lan t\ran^{-\f12} \Big( E^2(t)+\cE^2(t)+\int ^t_0 \cF^2(\tau) d\tau\Big).
\eeno

Indeed,  we only need to give the estimate of the most trouble terms, using the magnetic field equation,
\begin{align*}
&\int^\f{t}{2} _0\Big\| |\xi|e^{-\f{|\xi_1|^2}{|\xi|^2} (t-\tau)} \cF( u_{2,\sim} \,b_1) \Big\|_{L^1_{\xi_1}L^1_{\xi_2}} d\tau\non\\
&=\int^\f{t}{2} _0\Big\||\xi|e^{-\f{|\xi_1|^2}{|\xi|^2} (t-\tau)} \cF\Big( (-\Delta)^{-1} \na^T\cdot \p_1 u_\sim  \,b_1\Big) \Big\|_{L^1_{\xi_1}L^1_{\xi_2}} d\tau\non\\
&=\int^\f{t}{2} _0\Big\| |\xi| e^{-\f{|\xi_1|^2}{|\xi|^2} (t-\tau)} \cF\Big( (-\Delta)^{-1} \na^T\cdot (\p_\tau b+u\cdot\na b-b\cdot\na u)_{\sim}  \,b_1\Big) \Big\|_{L^1_{\xi_1}L^1_{\xi_2}} d\tau,
\end{align*}
where
\begin{align*}
&\int^\f{t}{2} _0\Big\||\xi| e^{-\f{|\xi_1|^2}{|\xi|^2} (t-\tau)} \cF\Big( (-\Delta)^{-1} \na^T\cdot \p_\tau b_{\sim}  \,b_1\Big) \Big\|_{L^1_{\xi_1}L^1_{\xi_2}} d\tau\non\\
&\lesssim  \Big\| \cF\Big(\na\big( (-\Delta)^{-1} \na^T\cdot b_{\sim}  \,b_{1}\big)\Big) \Big\|_{L^1_{\xi_1}L^1_{\xi_2}} +\lan t\ran^{-\f12} \Big\| \lan\na \ran^{\f52+\delta} \Big( (-\Delta)^{-1} \na^T\cdot b_{\sim, 0}  \,b_{1, 0}\Big) \Big\|_{L^1_{x_1}L^2_{x_2}}\non\\
&\quad +\int^{\f{t}{2}}_0 \lan t-\tau\ran^{-1} \Big\| |\xi| \cF\Big( (-\Delta)^{-1} \na^T\cdot b_{\sim}  \,b_{1}\Big) \Big\|_{L^1_{\xi_1}L^1_{\xi_2}} d\tau\non\\
&\quad +\int^\f{t}{2} _0\Big\||\xi| e^{-\f{|\xi_1|^2}{|\xi|^2} (t-\tau)} \cF\Big( (-\Delta)^{-1} \na^T\cdot  b_{\sim}  \,(\p_1 u_1-u\cdot\na b_1+b\cdot\na u_1)\Big) \Big\|_{L^1_{\xi_1}L^1_{\xi_2}} d\tau.
\end{align*}
The most trouble term is
\begin{align*}
&\int^\f{t}{2} _0\Big\| |\xi| e^{-\f{|\xi_1|^2}{|\xi|^2} (t-\tau)} \cF\Big( (-\Delta)^{-1} \na^T\cdot  b_{\sim}  \p_1 u_1\Big) \Big\|_{L^1_{\xi_1}L^1_{\xi_2}} d\tau\non\\
&\lesssim \int^\f{t}{2} _0\Big\| \f{|\xi_1|}{|\xi|}|\xi|^2 e^{-\f{|\xi_1|^2}{|\xi|^2} (t-\tau)} \cF\Big( (-\Delta)^{-1} \na^T\cdot  b_{\sim}  \, u_1\Big) \Big\|_{L^1_{\xi_1}L^1_{\xi_2}} d\tau\non\\
&\quad +\int^\f{t}{2} _0\Big\||\xi| e^{-\f{|\xi_1|^2}{|\xi|^2} (t-\tau)} \cF\Big( (-\Delta)^{-1} \na^T\cdot  \p_1 b_{\sim}  \, u_1\Big) \Big\|_{L^1_{\xi_1}L^1_{\xi_2}} d\tau\non\\
&\lesssim \int^\f{t}{2} _0 \lan t-\tau\ran^{-\f12} \Big\|e^{-\f{|\xi_1|^2}{|\xi|^2} (t-\tau)} \cF\Big( \na^2\big( (-\Delta)^{-1} \na^T\cdot b_\sim \, u_1\big)\Big)\Big\|_{L^1_{\xi_1}L^1_{\xi_2}} d\tau\non\\
&\quad +\int^\f{t}{2} _0\Big\| e^{-\f{|\xi_1|^2}{|\xi|^2} (t-\tau)} \cF\Big(\na\big( b_{2, \sim}  \, u_1\big)\Big) \Big\|_{L^1_{\xi_1}L^1_{\xi_2}} d\tau,
\end{align*}
where
\begin{align*}
&\int^\f{t}{2} _0 \lan t-\tau\ran^{-\f12} \Big\|e^{-\f{|\xi_1|^2}{|\xi|^2} (t-\tau)} \cF\Big( \na^2\big( (-\Delta)^{-1} \na^T\cdot b_\sim \, u_1\big)\Big)\Big\|_{L^1_{\xi_1}L^1_{\xi_2}} d\tau\non\\
&\lesssim \int^{\f{t}{2}} _0  \lan t-\tau\ran^{-\f12} \sum_{j\geq 0} \Big\| e^{-\f{|\xi_1|^2}{j^2} (t-\tau)}  \Big\|_{L^2_{\xi_1}} \Big\| \cF\Big( P_j \na^2\big((-\Delta)^{-1} \na^T\cdot  b_{\sim}  \, u_1 \big) \Big)\Big\|_{L^2_{\xi_1}L^2_{\xi_2}}    \Big( \int _{|\xi_2|\leq j}  d\xi_2\Big) ^\f12 d\tau\non\\
&\lesssim \int^{\f{t}{2}} _0  \lan t-\tau\ran^{-\f34}  \Big(\| \cF\na^{2+\delta} b\|_{L^2_{\xi_1}L^2_{\xi_2}} \|\cF u\|_{L^1_{\xi_1}L^1_{\xi_1}}+\|\cF(-\Delta)^{-1}\na^{T} \cdot b_\sim\|_{L^1_{\xi_1}L^1_{\xi_2}}\|\cF\na^{3+\delta} u\|_{L^2_{\xi_1}L^2_{\xi_2}}\Big) d\tau \non\\
&\lesssim \int^{\f{t}{2}} _0  \lan t-\tau\ran^{-\f34}  \Big(\|\na^{2+\delta} b\|_{L^2_{x_1x_2}} \|\cF u\|_{L^1_{\xi_1}L^1_{\xi_1}}+\|\cF(-\Delta)^{-1}\na^{T} \cdot b_\sim\|_{L^1_{\xi_1}L^1_{\xi_2}} \|\na^2 u\|_{L^2}^{\f12-\delta}\|\na^4 u\|_{L^2}^{\f12+\delta}\Big) d\tau \non\\
&\lesssim \lan t\ran^{-\f12} \Big( E^2(t)+\cE^2(t)+\int ^t_0 \cF^2(\tau) d\tau\Big),
\end{align*}
and
\begin{align*}
&\int^\f{t}{2} _0\Big\| e^{-\f{|\xi_1|^2}{|\xi|^2} (t-\tau)} \cF\Big(\na\big( b_{2, \sim}  \, u_1\big)\Big) \Big\|_{L^1_{\xi_1}L^1_{\xi_2}} d\tau\non\\
&=\int^\f{t}{2} _0\Big\| e^{-\f{|\xi_1|^2}{|\xi|^2} (t-\tau)} \cF\Big(\p_1b_{1, \sim}  \, u_1+b_{2, \sim} \na u_1\Big) \Big\|_{L^1_{\xi_1}L^1_{\xi_2}} d\tau\non\\
&\lesssim\int^\f{t}{2} _0\Big\|\f{|\xi_1|}{|\xi|} e^{-\f{|\xi_1|^2}{|\xi|^2} (t-\tau)} \cF\Big(\na(b_{1, \sim}  \, u_1)\Big) \Big\|_{L^1_{\xi_1}L^1_{\xi_2}} d\tau+\int^\f{t}{2} _0\Big\|e^{-\f{|\xi_1|^2}{|\xi|^2} (t-\tau)} \cF\Big(b_{1, \sim}  \, \p_1u_1\Big) \Big\|_{L^1_{\xi_1}L^1_{\xi_2}} d\tau\non\\
&\quad+\int^\f{t}{2} _0\Big\| e^{-\f{|\xi_1|^2}{|\xi|^2} (t-\tau)} \cF\Big(b_{2, \sim} \na u_1)\Big) \Big\|_{L^1_{\xi_1}L^1_{\xi_2}} d\tau\non\\
&\lesssim \lan t\ran^{-\f12} \Big( E^2(t)+\cE^2(t)+\int ^t_0 \cF^2(\tau) d\tau\Big)£¬
\end{align*}
where we use \eqref{b1infty}.

Since the  estimate of $b_2$ is easier than $b_1$, we omit the details.

{\bf Step 3.} The decay rate of $ \|\p_1 u_N\|_{H^1}$ and $\|\na u_N\|_{H^1}$.

We first consider the $L^2$ decay rate of $\p_1 u$, for $M_{1, 3, 4, 6}$
\begin{align*}
&\Big\|\int^t_0|\xi_1| \f{|\xi_1|}{|\xi|^2} e^{-\f{|\xi_1|^2}{|\xi|^2} (t-\tau)}\cF\Big(u\cdot\na u+b\cdot\na b+u\cdot\na b+b\cdot\na u\Big)(\tau) d\tau\Big\|_{L^2_{\xi_1}L^2_{\xi_2}}\non\\
&\lesssim \int^t_0 \lan t-\tau\ran^{-1}\Big(\|u\|_{L^\infty}\|\na u\|_{L^2} +\|b_1\|_{L^\infty} \|\p_1 b\|_{L^2}+\|b_2\|_{L^\infty}\|\p_2 b\|_{L^2}\non\\
&\quad +\|u\|_{L^\infty}\|\na b\|_{L^2}+\|b_1\|_{L^\infty} \|\p_1 u \|_{L^2}+\|b_2\|_{L^\infty}\|\na u\|_{L^2} \Big) d\tau\non\\
&\lesssim \int^t_0 \lan t-\tau\ran^{-1}\lan \tau\ran^{-1-\delta}  d\tau  \Big( E^2(t)+\cE^2(t)\Big) \non\\
&\lesssim \lan t\ran^{-1} \Big( E^2(t)+\cE^2(t)\Big),
\end{align*}
where we use
\beno
\|\na b\|_{L^2}\leq \|b\|_{L^2}^\f12\|\na^2 b\|_{L^2}^\f12, \quad \|\na u\|_{L^2}\leq \|u\|_{L^2}^\f12\|\na^2 u\|_{L^2}^\f12,
\eeno
and for $M_{2, 5}$,  we use the result of the linear part to have
\begin{align*}
&\Big\|\int^t_0|\xi_1|^{-\f12}|\xi_1| \f{|\xi_1|}{|\xi|^2} e^{-\f{|\xi_1|^2}{|\xi|^2} (t-\tau)}\cF\Big(u\cdot\na b_2+b\cdot\na u_2\Big)(\tau) d\tau\Big\|_{L^2_{\xi_1}L^1_{\xi_2}}\non\\
&\lesssim \int^t_0 \lan t-\tau\ran^{-1}\Big\|\lan \na\ran ^{\f12+\delta}( u\cdot\na b_2+b\cdot\na u_2)\Big\|_{L^1_{x_1}L^2_{x_2}} d\tau\non\\
&\lesssim \int^t_0 \lan t-\tau\ran^{-1}\lan \tau\ran^{-1-\delta}  d\tau  \Big( E^2(t)+\cE^2(t)+\int ^t_0 \cF^2(\tau) d\tau\Big) \non\\
&\lesssim \lan t\ran^{-1} \Big( E^2(t)+\cE^2(t)+\int ^t_0 \cF^2(\tau) d\tau\Big),
\end{align*}
and
\begin{align*}
&\Big\|\int^t_0|\xi_1|^{-\f12}|\xi_1| \f{|\xi_1|^2}{|\xi|^4} e^{-\f{|\xi_1|^2}{|\xi|^2} (t-\tau)}\cF\Big(\p_i(u_i u)+\p_2(b_2b)\Big)(\tau) d\tau\Big\|_{L^2_{\xi_1}L^1_{\xi_2}}\non\\
&=\Big\|\int^t_0 \f{|\xi_1|^\f52}{|\xi|^3} e^{-\f{|\xi_1|^2}{|\xi|^2} (t-\tau)}\cF\Big(u_i u+b_2 b\Big)(\tau) d\tau\Big\|_{L^2_{\xi_1}L^1_{\xi_2}}\non\\
&\lesssim \int^t_0 \lan t-\tau\ran^{-1}\Big(\sum_{j\leq 1}\Big\|e^{-\f{|\xi_1|^2}{j^2} (t-\tau)}\Big\|_{L^\infty_{\xi_1}}\Big\|\cF(P_j (u_i u+b_2 b))\Big\|_{L^2_{\xi_1}L^2_{\xi_2}}\Big(\int_{|\xi_2|\leq j} d\xi_2\Big)^\f12\non\\
&\qquad+  \sum_{j\geq 1}\Big\|e^{-\f{|\xi_1|^2}{j^2} (t-\tau)}\Big\|_{L^\infty_{\xi_1}}\Big\|\cF(P_j (u_i u+b_2 b))\Big\|_{L^2_{\xi_1}L^2_{\xi_2}}\Big(\int_{|\xi_2|\leq j} d\xi_2\Big)^\f12\Big)d\tau\non\\
&\lesssim \int^t_0 \lan t-\tau\ran^{-1}\lan \tau\ran^{-1-\delta}  d\tau  \Big( E^2(t)+\cE^2(t)\Big) \non\\
&\lesssim \lan t\ran^{-1} \Big( E^2(t)+\cE^2(t)\Big),
\end{align*}
and similarly to \eqref{p1}, we have
\begin{align*}
\Big\|\int^t_0|\xi_1|^{-\f12}|\xi_1| \f{|\xi_1|^2}{|\xi|^4} e^{-\f{|\xi_1|^2}{|\xi|^2} (t-\tau)}\cF\Big(\p_1(b_1b)\Big)(\tau) d\tau\Big\|_{L^2_{\xi_1}L^1_{\xi_2}}\lesssim \lan t\ran^{-1} \Big( E^2(t)+\cE^2(t)\Big).
\end{align*}

We consider the $H^1$ decay rate of $\p_1 u$. For $M_{1, 3, 4, 6}$, we have
\begin{align*}
&\Big\|\int^t_0|\xi_1| \f{|\xi_1|}{|\xi|^2} e^{-\f{|\xi_1|^2}{|\xi|^2} (t-\tau)}|\xi| \cF\Big(u\cdot\na b+b\cdot\na u \Big)(\tau) d\tau\Big\|_{L^2_{\xi_1}L^2_{\xi_2}}\non\\
&\lesssim \int^t_0 \lan t-\tau\ran^{-1}\Big(\|\na u_1\|_{L^\infty} \|\p_1 b\|_{L^2}+\|u_1\|_{L^\infty} \|\p_1 \na b\|_{L^2}+\|\na u_2 \|_{L^2_{x_1}L^2_{x_1}} \|\p_2 b\|_{L^2_{x_1}L^\infty_{x_2}}+\|u_2\|_{L^\infty} \|\p_2\na b\|_{L^2}\non\\
&\quad +\|\na b_1\|_{L^2_{x_2}L^\infty_{x_1}} \|\p_1 u\|_{L^2_{x_1}L^\infty_{x_2}}+\|b_1\|_{L^\infty} \|\p_1 \na u\|_{L^2}+\|\na b_2 \|_{L^2}\|\p_ 2u \|_{L^\infty}+ \|b_2\|_{L^\infty} \| \p_2\na u\|_{L^2} \Big) d\tau\non\\
&\lesssim \lan t\ran^{-1} \Big( E^2(t)+\cE^2(t)+\int ^t_0 \cF^2(\tau) d\tau\Big),
\end{align*}
where
\begin{align*}
&\int^t_0 \lan t-\tau\ran^{-1}\|\na u\|_{L^\infty} \|\p_1b\|_{L^2} d\tau\non\\
&\lesssim \int^t_0 \lan t-\tau\ran^{-1}\|\na u\|_{L^2}^\f12\|\na ^3 u\|_{L^2}^\f12\|\p_1b\|_{L^2} d\tau\non\\
&\lesssim \int^t_0 \lan t-\tau\ran^{-1}\big(\lan \tau\ran^{-\f34} E(t)\big)^\f12 \|\na^3 u\|_{L^2}^\f12 \lan \tau\ran^{-\f34} E(t) d\tau\non\\
&\lesssim  E(t)^\f32\Big(\int^t_0 \big(\lan t-\tau\ran^{-1} \lan \tau\ran^{-\f98}\big)^\f43 d\tau\Big)^{\f34} \Big(\int^t_0\big( \|\na^3 u\|_{L^2}^\f12\big)^4 d\tau \Big)^\f14\non\\
&\lesssim \lan t\ran^{-1} \Big(E^2(t) +\int ^t_0 \cF^2(\tau) d\tau \Big).
\end{align*}
For $M_{2, 5}$,  we use the result of the linear part to have
\begin{align*}
&\Big\|\int^t_0|\xi_1|^{-\f12}|\xi_1||\xi| \f{|\xi_1|}{|\xi|^2} e^{-\f{|\xi_1|^2}{|\xi|^2} (t-\tau)}\cF\Big(u\cdot\na b_2+b\cdot\na u_2\Big)(\tau) d\tau\Big\|_{L^2_{\xi_1}L^1_{\xi_2}}\non\\
&\lesssim \int^t_0 \lan t-\tau\ran^{-1}\Big\|\lan \na\ran ^{\f32+\delta}( u\cdot\na b_2+b\cdot\na u_2)\Big\|_{L^1_{x_1}L^2_{x_2}} d\tau\non\\
&\lesssim \int^t_0 \lan t-\tau\ran^{-1}\lan \tau\ran^{-1-\delta}  d\tau  \Big( E^2(t)+\cE^2(t)+\int ^t_0 \cF^2(\tau) d\tau\Big) \non\\
&\lesssim \lan t\ran^{-1} \Big( E^2(t)+\cE^2(t)+\int ^t_0 \cF^2(\tau) d\tau\Big),
\end{align*}
and
\begin{align*}
&\Big\|\int^t_0|\xi_1|^{-\f12}|\xi_1| |\xi|\f{|\xi_1|^2}{|\xi|^4} e^{-\f{|\xi_1|^2}{|\xi|^2} (t-\tau)}\cF\Big(\p_i(u_i u)+\p_2(b_2b)\Big)(\tau) d\tau\Big\|_{L^2_{\xi_1}L^1_{\xi_2}}\non\\
&=\Big\|\int^t_0 \f{|\xi_1|^\f52}{|\xi|^3} e^{-\f{|\xi_1|^2}{|\xi|^2} (t-\tau)}\cF\Big(\na(u_i u+b_2 b)\Big)(\tau) d\tau\Big\|_{L^2_{\xi_1}L^1_{\xi_2}}\non\\
&\lesssim \int^t_0 \lan t-\tau\ran^{-1}\Big(\sum_{j\leq 1}\Big\|e^{-\f{|\xi_1|^2}{j^2} (t-\tau)}\Big\|_{L^\infty_{\xi_1}}\Big\|\cF(P_j (\na(u_i u+b_2 b)))\Big\|_{L^2_{\xi_1}L^2_{\xi_2}}\Big(\int_{|\xi_2|\leq j} d\xi_2\Big)^\f12\non\\
&\qquad+  \sum_{j\geq 1}\Big\|e^{-\f{|\xi_1|^2}{j^2} (t-\tau)}\Big\|_{L^\infty_{\xi_1}}\Big\|\cF(P_j (\na(u_i u+b_2 b)))\Big\|_{L^2_{\xi_1}L^2_{\xi_2}}\Big(\int_{|\xi_2|\leq j} d\xi_2\Big)^\f12\Big)d\tau\non\\
&\lesssim \int^t_0 \lan t-\tau\ran^{-1}\lan \tau\ran^{-1-\delta}  d\tau  \Big( E^2(t)+\cE^2(t)\Big) \non\\
&\lesssim \lan t\ran^{-1} \Big( E^2(t)+\cE^2(t)\Big),
\end{align*}
and similarly to \eqref{p1}, we have
\begin{align*}
\Big\|\int^t_0|\xi_1|^{-\f12}|\xi_1||\xi| \f{|\xi_1|^2}{|\xi|^4} e^{-\f{|\xi_1|^2}{|\xi|^2} (t-\tau)}\cF\Big(\p_1(b_1b)\Big)(\tau) d\tau\Big\|_{L^2_{\xi_1}L^1_{\xi_2}}\lesssim \lan t\ran^{-1} \Big( E^2(t)+\cE^2(t)\Big).
\end{align*}

Now we consider the $L^2$ decay rate of $\na  u$. For $M_{1, 3, 4, 6}$, by \eqref{h3}, we have
\begin{align*}
&\Big\|\int^t_0|\xi| \f{|\xi_1|}{|\xi|^2} e^{-\f{|\xi_1|^2}{|\xi|^2} (t-\tau)}\cF\Big(u_1 \p_1 b+b\cdot\na u\Big)(\tau) d\tau\Big\|_{L^2_{\xi_1}L^2_{\xi_2}}\non\\
&\lesssim \int^t_0 \lan t-\tau\ran^{-\f12}
\sum_{j\geq 0}  \Big\| e^{-\f{|\xi_1|^2}{j^2} (t-\tau)} \Big\|_{L^2_{\xi_1}} \Big\| \cF\Big(P_j (u_1 \p_1 b+b\cdot\na u)\Big)(\tau)\Big\|_{L^\infty_{\xi_1}L^2_{\xi_2}}d\tau \non\\
&\lesssim \int^t_0 \lan t-\tau\ran^{-\f34} \Big[\sum_{j\geq 1} j^{-\delta}
\Big\|j^{\f12+\delta} \cF\Big(P_j (u_1 \p_1 b+b\cdot\na u)\Big)\Big\|_{L^\infty_{\xi_1}L^2_{\xi_2}}+\sum_{0\leq j\leq 1} j^\f12
\Big\|  \cF\Big(P_j (u_1 \p_1 b+b\cdot\na u)\Big)\Big\|_{L^\infty_{\xi_1}L^2_{\xi_2}} \Big]d\tau\non\\
&\lesssim \int^t_0 \lan t-\tau\ran^{-\f34}\Big \|\lan \na\ran ^{\f12+\delta}(u_1 \p_1 b+b\cdot\na u)\Big \|_{L^1_{x_1}L^2_{x_2}} d\tau\non\\
&\lesssim \int^t_0 \lan t-\tau\ran^{-\f34}\lan \tau\ran^{-1-\delta}  d\tau \Big( E^2(t)+\cE^2(t)+\int ^t_0 \cF^2(\tau) d\tau\Big) \non\\
&\lesssim \lan t\ran^{-\f34} \Big( E^2(t)+\cE^2(t)+\int ^t_0 \cF^2(\tau) d\tau \Big).
\end{align*}
Then we rewrite  $u_2 \p_2 b_1$ as $\big(u_{2, <\lan \tau\ran^{-s_1}}+u_{2, >\lan \tau\ran^{-s_2}} +u_{2, \sim}\big) \p_2 b_1$ to have
 \begin{align*}
\Big\|\int^t_0|\xi| \f{|\xi_1|}{|\xi|^2} e^{-\f{|\xi_1|^2}{|\xi|^2} (t-\tau)}\cF\Big(u_2  \p_2 b_1\Big)(\tau) d\tau\Big\|_{L^2_{\xi_1}L^2_{\xi_2}}
\lesssim \lan t\ran^{-\f34} \Big( E^2(t)+\cE^2(t)\Big),
\end{align*}
which is very similar as the decay estimate of $\|b_1\|_{L^2} $.

For $M_{2, 5}$,  we use the same way with \eqref{divN} to have
\begin{align*}
&\Big\|\int^t_0|\xi_1|^{-\f12}|\xi| \f{|\xi_1|}{|\xi|^2} e^{-\f{|\xi_1|^2}{|\xi|^2} (t-\tau)}\cF\Big(-u\cdot\na b_2+b\cdot\na u_2\Big)(\tau) d\tau\Big\|_{L^2_{\xi_1}L^1_{\xi_2}}\non\\
&=\Big\|\int^t_0 \lan t-\tau\ran^{-\f12}|\xi_1|^{-\f12} e^{-\f{|\xi_1|^2}{|\xi|^2} (t-\tau)}\cF\Big(-u\cdot\na b_2+b\cdot\na u_2\Big)(\tau) d\tau\Big\|_{L^2_{\xi_1}L^1_{\xi_2}}\non\\
&\lesssim \lan t\ran^{-1} \Big( E^2(t)+\cE^2(t)\Big).
\end{align*}
Similarly, we use the same way to estimate the  $L^2$ decay rate of $\na^2  u$.

{\bf Step 4.} The decay rate of $ \|\p_1 b_N\|_{L^2}$.

For $N_{1, 3, 4, 6, 7}$, we have
\begin{align*}
&\Big\|\int^t_0\f{|\xi_1|}{|\xi|} e^{-\f{|\xi_1|^2}{|\xi|^2} (t-\tau)}|\xi| \cF\Big(u\cdot\na b+b\cdot\na u\Big)(\tau) d\tau\Big\|_{L^2_{\xi_1}L^2_{\xi_2}}\non\\
&\lesssim\int^t_0 \lan t-\tau\ran^{-\f12}  \sum_{j\geq 0}  j\Big\| e^{-\f{|\xi_1|^2}{j^2} (t-\tau)} \Big\|_{L^2_{\xi_1}} \Big\| \cF\Big(P_j (u\cdot\na b+b\cdot\na u)\Big)(\tau)\Big\|_{L^\infty_{\xi_1}L^2_{\xi_2}} d\tau\non\\
&\lesssim \int^t_0 \lan t-\tau\ran^{-\f34} \Big\|\lan \na\ran ^{\f32+\delta}(u\cdot\na b+b\cdot\na u )\Big\|_{L^1_{x_1}L^2_{x_2}}  d\tau \non\\
&\lesssim  \int^{t}_0\lan t-\tau\ran^{-\f34}  \lan \tau \ran ^{-1-\delta} d\tau \, \Big( E^2(t)+\cE^2(t)+\int ^t_0 \cF^2(\tau) d\tau\Big)\non\\
&\lesssim \lan t\ran^{-\f34} \Big( E^2(t)+\cE^2(t)+\int ^t_0 \cF^2(\tau) d\tau\Big),
\end{align*}
which had been given in $\|b_1\|_{L^\infty}$.

For $N_{2, 5}$, we have
\begin{align*}
&\Big\|\int^t_0e^{-\f{|\xi_1|^2}{|\xi|^2} (t-\tau)}|\xi_1|^{-\f12}|\xi_1|\cF\Big((-u\cdot\na b_2+b\cdot\na u_2)\Big)(\tau) d\tau\Big\|_{L^2_{\xi_1}L^1_{\xi_2}}\non\\
&=\Big\|\int^t_0e^{-\f{|\xi_1|^2}{|\xi|^2} (t-\tau)}\f{|\xi_1|^{\f12}}{|\xi|}\cF\Big(\na(-u\cdot\na b_2+b\cdot\na u_2)\Big)(\tau) d\tau\Big\|_{L^2_{\xi_1}L^1_{\xi_2}}\non\\
&\lesssim \lan t\ran^{-\f34} \Big( E^2(t)+\cE^2(t)+\int ^t_0 \cF^2(\tau) d\tau\Big).
\end{align*}

{\bf Step 5.} The decay rate of $ \|\p_1 u_N\|_{L^\infty}$.

We have
\begin{align*}
&\Big\|\int^{t}_0\f{|\xi_1|^2}{|\xi|^2} e^{-\f{|\xi_1|^2}{|\xi|^2} (t-\tau)}\cF\Big( u\cdot\na b+b\cdot\na u \Big)(\tau) d\tau\Big\|_{L^1_{\xi_1}L^1_{\xi_2}}\non\\
&\lesssim\int^{t}_0 \lan t-\tau\ran ^{-1} \sum_{j\geq 0} \Big\| e^{-\f{|\xi_1|^2}{j^2} (t-\tau)} \Big\|_{L^2_{\xi_1}} \Big\| \cF\Big(P_j (u\cdot\na b+b\cdot\na u)\Big)(\tau)\Big\|_{L^2_{\xi_1}L^2_{\xi_2}}\Big( \int_{|\xi_2|\leq j} d\xi_2\Big)  ^{\f12}d\tau \non\\
&\lesssim \int^{t}_0 \lan t-\tau\ran^{-\f54}\Big\|\lan \na\ran ^{1+\delta} (u\cdot\na b+b\cdot\na u)\Big\|_{L^2_{x_1}L^2_{x_2}}d\tau\non\\
&\lesssim \int^{t}_0 \lan t-\tau\ran^{-\f54} \Big(\|\na^{1+\delta}u_1\|_{L^2} \|\p_1 b\|_{L^\infty}+\|\na^{1+\delta} u_2 \|_{L^2}\|\p_2 b\|_{L^\infty}+\|u_1\|_{L^\infty} \|\p_1 \na^{1+\delta} b\|_{L^2}+\|u_2\|_{L^\infty}\|\p_2 \na^{1+\delta} b\|_{L^2}\non\\
&\quad+\|\na^{1+\delta} b_1\|_{L^2}\|\p_1 u\|_{L^\infty}+\|\na^{1+\delta} b_2\|_{L^2}\|\p_2 u\|_{L^\infty}+\|b_1\|_{L^\infty}\|\p_1 \na^{1+\delta} u\|_{L^2}+\|b_2\|_{L^\infty} \|\p_2 \na^{1+\delta} u\|_{L^2}\Big)d\tau\non\\
&\lesssim \int^{t}_0 \lan t-\tau\ran^{-\f54}\lan \tau\ran^{-1-\delta}  d\tau\Big( E^2(t)+\cE^2(t)+\int ^t_0 \cF^2(\tau) d\tau\Big) \non\\
&\lesssim \lan t\ran^{-1-\delta} \Big( E^2(t)+\cE^2(t)+\int ^t_0 \cF^2(\tau) d\tau\Big).
\end{align*}

This completes the proof of Proposition 3.1.
\end{proof}

\section{Local well-posedness}\label{local}
We give the local well-posedness of systems \eqref{eq:MHDT} in half space $\Omega$  without proof for completeness.
\begin{theorem}\label{Local}
	Assume that the initial data $(u_0,b_0)\in H^3(\Omega)$ $, \, u_0\in H^1_0(\Om)$, and $\div{u_0}=0$ in $\Omega$, $b_{1, 0}=\p_2 b_{1, 0}=0$ on $\p\Om$, $\mathcal{P}(\Delta u_0 -u_0 \cdot \na u_0 +b_0 \cdot \na b_0) \in H_0^1(\Om)$. Then there exists a $T>0$ such that system  \eqref{eq:MHDT} admits a unique solution $(u,b)$ on $[0,T]$ satisfying
	\begin{align*}
		\displaystyle
		(u,b) \in C([0,T];H^3(\Omega)),
	\end{align*}
\end{theorem}
In fact, we can first construct an interation scheme for the system \eqref{eq:MHDT} in half space to obtain the approximate solution and then derive uniform bounds to pass the limit (Similar  Theorem 3.1 in \cite{RXZ}). This procedure is more or less standard and thus we omit their details.

\smallskip

\section{High order energy estimate}
In this section, we will prove the high order a-priori estimate in the half space using  the low order decay rate (which had been proved in section 3).

We first introduce the following energy
\begin{align}\label{cE}
\cE^2(t):=\|u(t)\|^2_{H^3}+\|b(t)\|^2_{H^3}+\|u_t(t)\|_{H^1}^2,
\end{align}
and the dissipated energy
\begin{align}\label{cF}
\cF^2(t):=\|\na u(t)\|^2_{H^3}+\|\p_1b(t)\|^2_{H^2}+\|u_t(t)\|_{H^1}^2+\|b_t(t)\|_{H^1}^2+\|\p_t ^2 u(t)\|_{L^2}^2.
\end{align}

\begin{proposition}\label{high order}
Assume that the solution $(u,b)$ of the system \eqref{eq:MHDT} satisfies
\beno
\sup\limits_{0\le t\le T}\big(\|u(t)\|_{H^3}^2+\|b(t)\|_{H^3}^2\big)\leq  c_0^2.  \eeno
If $ c_0  $ is suitably small, then there holds that
\beno
\cE^2(t)+ \int_0^t\cF^2(s)ds\lesssim \|u_0\|^2_{H^3}+\|b_0\|^2_{H^3}+\cE^2(t)\int^t_0\Big(\|b_2\|_{L^\infty}^2+\|b_2\|_{L^\infty}^4+\|\p_1 u_1\|_{L^\infty}\Big) ds
\eeno
for any  $t\in [0,T]$.
\end{proposition}

\begin{proof} For the half space problem, we will prove the a-priori estimate step by step.

{\bf Step 1.} $L^2$ estimate of $(u,b)$.

Thanks to $u=b=0$ on $\p\Om$, we take the $L^2$ inner product of equations $(\ref{eq:MHDT})_1$ and $(\ref{eq:MHDT})_2$ with $u$ and $b$, respectively, and integrate by parts to obtain
\begin{align}\label{L2}
\f12\f d {dt}\big(\|u(t)\|^2_{L^2}+\|b(t)\|^2_{L^2}\big)+\|\na u(t)\|^2_{L^2}=0
\end{align}
for any $t\in [0, T]$.

{\bf Step 2.} $\dot{H}^1$ estimate of $(u,b)$.

To obtain the $\dot{H}^1$ estimate of $u$, we introduce the  Helmholtz projection
\beno
\cP: \ L^2(\Om)\rightarrow L^2_{\sigma}(\Om), \qquad L^2_{\sigma}(\Om)=\big\{v\:|\: v\in L^2, \div v=0, v\cdot n=0 \  \mathrm{on}\ \p\Om\big\}
\eeno
 to  eliminate the pressure term.

  We take the $L^2$ product of equation $(\ref{eq:MHDT})_1$ with $\cA u:=-\cP\Del u$, apply $\na$ to equation $(\ref{eq:MHDT})_2$ and then take the $L^2$ inner product of the resulting equation with $\na b$, to have
\begin{align}\label{H1}
&\f12\f d {dt}\big(\|\na u\|^2_{L^2}+\|\na b\|^2_{L^2}\big)+\|\cA u\|^2_{L^2}\non\\
&=-\lan u\cdot\na u,\cA u\ran+\lan b\cdot\na b,\cA u\ran-\lan \na(u\cdot\na b),\na b\ran
+\lan\na(b\cdot\na u),\na b\ran,
\end{align}
where we use the integration by parts and boundary condition.

The nonlinear estimate in \eqref{H1}    is similar as \cite{RXZ}, here  we only give the different terms:
\begin{align*}
&\lan b_2\p_2 b, \cA u\ran+\lan b_2 \p_2^2 u_1, \p_2 b_1\ran\non\\
&\lesssim\|b_2\|_{L^2_{x_1} L^\infty_{x_2}} \|\p_2 b\|_{L^2_{x_2}L^\infty_{x_1}}\|\na^2 u\|_{L^2}\non\\
&\lesssim\|b\|_{H^1}^2 \Big( \|\p_1 b\|_{H^1}^2 + \|\na^2 u\|_{L^2}^2\Big),
\end{align*}
and
\begin{align*}
\lan (\p_2 b_1)^2, \p_1 u_1\ran&=2\lan \p_1\p_2 b_1 \,\p_2 b_1, u_1\ran\non\\
&\lesssim \|\p_1 \p_2 b\|_{L^2}\|\p_2 b\|_{L^2_{x_2}L^\infty_{x_1}}\|u\|_{L^2_{x_1}L^\infty_{x_2}}\non\\
&\lesssim\|(u, b)\|_{H^1}^2 \Big( \|\na \p_1 b\|_{L^2}^2 + \|\na u\|_{L^2}^2\Big),
\end{align*}
where we use
\beno
\|\cA u\|_{L^2}\lesssim \|\na ^2 u\|_{L^2} , \qquad \|f\|_{L^\infty(\R)} \lesssim \|f\|_{L^2(\R)}^\f12\|\na f\|_{L^2(\R)}^\f12.
\eeno

{\bf Step 3.} Dissipation estimate of  $\p_1b$.

By taking the $L^2$ product of equations $(\ref{eq:MHDT})_1$ and  $(\ref{eq:MHDT})_2$ with $-\p_1 b$ and $\p_1 u$, respectively, and using the integration by parts,  we deduce that
\begin{align}\label{p1b}
&\f d {dt}\lan b,\p_1 u\ran+\|\p_1 b\|_{L^2}^2-\|\p_1 u\|_{L^2}^2+\lan \Del u, \p_1 b\ran\nonumber\\
&=\lan \p_1 b,u\cdot\na u\ran-\lan \p_1b,b\cdot\na b\ran-\lan\p_1 u,u\cdot\na b\ran+\lan\p_1u,b\cdot\na u\ran,
\end{align}
the different estimate with \cite{RXZ} in \eqref{p1b}  is:
\begin{align*}
&\lan \p_1b_1,b^2\p_2 b_1\ran+\lan \p_1 u_1, u_2 \p_2 b_1\ran\\
&\leq \|\p_1 b\|_{L^2}\|b_2\|_{L^2_{x_1}L^\infty_{x_2}}\|\p_2 b\|_{L^2_{x_2}L^\infty_{x_1}}+\|\p_1 u\|_{L^2}\|u\|_{L^2_{x_1}L^\infty_{x_2}}\|\p_2 b\|_{L^2_{x_2}L^\infty_{x_1}}\\
&\lesssim \|(u, b)\|_{H^1}^2 \Big( \| \p_1 b\|_{H^1}^2 + \|\na u\|_{L^2}^2\Big).
\end{align*}

{\bf Step 4.} Dissipation estimate of $(u_t,b_t)$.

Taking the $L^2$ product of  equations $(\ref{eq:MHDT})_1$ and  $(\ref{eq:MHDT})_2$  with $u_t$ and $b_t$, respectively, and using the integration by parts,   we deduce that
\begin{align}\label{bt}
&\f12\f d {dt}\|\na u\|^2_{L^2}+\big(\|u_t\|^2_{L^2}+\|b_t\|_{L^2}^2\big)-\lan u_t,\p_1 b\ran-\lan b_t,\p_1 u\ran\nonumber\\
&=-\lan u_t,u\cdot\na u\ran+\lan u_t, b\cdot\na b\ran
 -\lan b_t,u\cdot\na b\ran+\lan b_t,b\cdot\na u \ran,
\end{align}
the different estimate with \cite{RXZ} in \eqref{bt}  is:
\begin{align*}
&\lan u_t, b_2 \p_2 b\ran+\lan b_t, u_2 \p_2 b\ran\\
&\leq \|u_t\|_{L^2}\|b_2\|_{L^2_{x_1}L^\infty_{x_2}}\|\p_2 b\|_{L^2_{x_2}L^\infty_{x_1}}+\|b_t\|_{L^2}\|u\|_{L^2_{x_1}L^\infty_{x_2}}\|\p_2 b\|_{L^2_{x_2}L^\infty_{x_1}}\\
&\lesssim \|(u, b)\|_{H^1}^2\Big(\|u_t\|_{L^2}^2+\|b_t\|_{L^2}^2+\|\na u\|_{L^2}^2+\|\p_1 b\|_{H^1}^2\Big).
\end{align*}

{\bf Step 5.} Energy estimate of $u_t$.

Applying $\p_t$ to equation $(\ref{eq:MHDT})_1$ and taking the $L^2$ inner product of the resulting equation with $u_t$, we obtain
\begin{align}\label{H2u}
&\f12\f d {dt}\|u_t\|_{L^2}^2+\|\na u_t\|_{L^2}^2+\lan b_t, \p_1 u_t\ran\non\\
&=-\lan u_t,\p_t(u\cdot\na u)\ran+\lan u_t,\p_t(b\cdot\na b)\ran\non\\
&=-\lan u_t, u_t \cdot\na u\ran+\lan u_t, b_t\cdot\na b+b\cdot\na b_t\ran\non\\
&=\lan u_t\cdot\na u_t, u\ran-\lan b_t\cdot\na u_t, b\ran-\lan b\cdot\na u_t, b_t\ran\non\\
&\lesssim \|u_t\|_{L^2}\|\na u_t \|_{L^2}\|u\|_{L^\infty}+\|b_t\|_{L^2}\|\na u_t\|_{L^2}\|b\|_{L^\infty}\non\\
&\lesssim \|(u, b)\|_{H^2}\Big(\|u_t\|_{H^1}^2+\|b_t\|_{L^2}^2\Big).
\end{align}

{\bf Step 6.} $\dot{H}^2$ estimate of $b$ and the dissipation estimate of $\na\p_1b$.

We apply $\Del$ to equation  $(\ref{eq:MHDT})_2$ and take the $L^2$ inner product of the resulting equation with $-\cA b$ to obtain
\begin{align}\label{h2b}
&\f12\f d {dt}\|\cA b\|^2_{L^2}+\|\na \p_1b\|_{L^2}^2-\lan \na\p_1 b, \na u_t\ran\non\\
&=\lan \Del b,\p_1\cP(u\cdot\na u-b\cdot\na b)\ran
-\lan\Del b, \Del\big(u\cdot\na b-b\cdot\na u\big)\ran, \non\\
&:=K_1+K_2.
\end{align}
where we use the magnetic field equation to rewrite $\Delta u$ as $u_t-\p_1 b+\na p+u\cdot\na u-b\cdot\na b$ and an observation that $\Del(u\cdot\na b-b\cdot\na u)  \in L^2_{\sigma}$.

We mention that the trouble term of $K_1$ is
\begin{align*}
&\lan \p_2^2 b, \p_1 \cP(u\cdot\na u-b\cdot\na b)\ran\non\\
=&\lan \p_1\p_2 b,\p_2\cP(u\cdot\na u-b\cdot\na b)\ran-\int_{\R}  \p_1\p_2 b_1 \, \Big(\cP(u\cdot\na u-b\cdot\na b)\Big)_1\Big|_{x_2=0} dx_1  \non\\
:=&K_{11}+K_{12}.
\end{align*}
The trouble term of  $K_{11}$ is
\begin{align*}
\lan \p_1\p_2 b, \p_2 \cP(b_2 \p_2 b)\ran&\lesssim \|\p_1 \p_2 b\|_{L^2}\|\p_2b_2 \p_2 b+b_2 \p_2 ^2 b\|_{L^2}\non\\
&\lesssim \|\p_1 \p_2 b\|_{L^2} \Big( \|\p_1 b\|_{L^4} \|\p_2 b\|_{L^4}+\|b_2\|_{L^\infty} \|\p_2^2 b\|_{L^2}\Big)\non\\
&\lesssim \|b_2\|_{L^\infty}^2\|\p_2 ^2 b\|_{L^2}^2+(\|b\|_{H^2} +\e)\|\p_1  b\|_{H^1}^2.
\end{align*}
For $K_{12}$, we have
\beno
K_{12}\lesssim \|\p_1 \p_2 b_1\|_{H^{-\f12}_{x_1}}\|\cP(u\cdot\na u-b\cdot\na b)\|_{H^\f12_{x_1}},
\eeno
with ($E$ is the extension operator in the $x_2$ direction)
\begin{align*}
\|\p_1 \p_2 b_1\|_{H^{-\f12}_{x_1}}
&\lesssim \Big(E^2\big(\|\p_1 \p_2 b_1\|_{H^{-\f12}_{x_1}} \big)\Big)^\f12\non\\
&\lesssim \Big(\int^\infty_{x_2}\p_{y_2} E^2 \big(\|\p_1\p _2 b_1\|_{H^{-\f12}_{x_1}} \big)dy_2\Big)^\f12\non\\
&\lesssim \Big(\int^\infty_{x_2}\int_{\R} \p_2^2 b\, \p_1\p_2 b \, dy_1dy_2\Big)^\f12\non\\
&\lesssim \|\p_2^2 b\|_{L^2}^\f12 \|\p_1 \p_2 b\|_{L^2}^\f12,
\end{align*}
 and
\begin{align}\label{trouble1}
&\|\cP(u\cdot\na u-b\cdot\na b)\|_{H^\f12_{x_1}}\non\\
&\lesssim \Big(E^2\big(\|\cP(u\cdot\na u-b\cdot\na b)\|_{H^\f12_{x_1}} \big)\Big)^\f12\non\\
&\lesssim \Big(\int^\infty_{x_2}\p_{y_2} E^2 \big(\|\cP(u\cdot\na u-b\cdot\na b)\|_{H^\f12_{x_1}} \big)dy_2\Big)^\f12\non\\
&\lesssim \Big(\int^\infty_{x_2}\int_{\R} \p_2 \cP(u\cdot\na u-b\cdot\na b) \ \p_1\cP(u\cdot\na u-b\cdot\na b) \, dy_1dy_2\Big)^\f12\non\\
&\lesssim \|\p_2 (u\cdot\na u-b\cdot\na b) \|_{L^2}^\f12 \|\p_1 (u\cdot\na u-b\cdot\na b) \|_{L^2}^\f12,
\end{align}
where the trouble term of \eqref{trouble1} is treated as
\begin{align*}
&\|\p_2 (b_2 \p_2 b_1)\|_{L^2}^\f12\|\p_1 (b_2 \p_2 b_1) \|_{L^2}^\f12\non\\
& \lesssim \Big( \|\p_1b_1\|_{L^2_{x_1}L^\infty_{x_2}}\|\p_2 b_1\|_{L^2_{x_2}L^\infty_{x_1}} +\|b_2\|_{L^\infty} \|\p_2 ^2 b_1\|_{L^2}\Big)^\f12 \Big(\|\p_1 b_2\|_{L^2_{x_1}L^\infty_{x_2}}\|\p_2 b_1\|_{L^2_{x_2}L^\infty_{x_1}} +\|b_2\|_{L^\infty} \|\p_1\p_2  b_1\|_{L^2}\Big)^\f12\non\\
&\lesssim \|\p_1 b\|_{H^1}^\f32\|\p_2 b\|^\f12_{H^1}+\|\p_1 b\|_{H^1}^\f54\|\p_2 b\|_{L^2}^\f14\|b_2\|_{L^\infty}^\f12+\|b_2\|_{L^\infty} \|\p_2 b\|_{H^1}^\f34\|\p_1 b\|_{H^1}^\f34+\|b_2\|_{L^\infty} \|\p_2^2 b\|_{L^2}^\f12 \|\p_1 \p_2 b\|_{L^2}^\f12.
\end{align*}
Thus we give the estimate of  the typical term of $K_{12}$
\begin{align*}
&\|b_2\|_{L^\infty}^\f12\|\p_2 b\|_{H^1}^\f34\|\p_1b\|_{H^1}^\f74+\|b_2\|_{L^\infty}\|\p_2 b\|_{H^1}^\f74\|\p_1 b \|_{H^1}^\f54\non\\
&\lesssim \Big(\|b_2\|_{L^\infty}^\f12\|\p_2 b\|_{H^1}^\f14\Big)^8+\Big(\|\p_2 b\|_{H^1}^\f12\|\p_1b\|_{H^1}^\f74\Big)^\f87+\Big(\|b_2\|_{L^\infty}\|\p_2 b\|_{H^1}^\f34\Big)^\f83+\Big(\|\p_2 b\|_{H^1}\|\p_1 b \|_{H^1}^\f54\Big)^\f85\non\\
&\lesssim (\|b_2\|_{L^\infty}^4+\|b_2\|_{L^\infty}^\f83 )\|\p_2 b\|_{H^1}^2 +(\|\p_2 b\|_{H^1}^\f47+\|\p_2 b\|_{H^1}^\f47) \|\p_1 b\|_{H^1}^2.
\end{align*}

Now we consider
\begin{align*}
K_2=-\lan \Delta b, \Delta(u\cdot\na b)\ran+\lan \Delta b, \Delta (b\cdot\na u)\ran:= K_{21}+K_{22},
\end{align*}
by the divergence free condition, the trouble term of $K_{21}$ are treated as
\beno
-\lan \p_2^2 b_1, \p_2 u_2 \p_2^2 b_1\ran\leq \|\p_1 u_1\|_{L^\infty}\|\p_2^2 b_1\|_{L^2}^2,
\eeno
and
\begin{align}\label{bh2-trouble}
&-\lan\p_2^2 b_1, \p_2^2u_2\p_2 b_1\ran\non\\
&=\lan \p_2^2 b_1, \p_1\p_2 u_1 \p_2 b_1\ran\non\\
&=\lan \p_2^2 b_1\p_2 b_1,\p_2 b_{1, t}\ran+\lan\p_2^2 b_1\p_2 b_1,\p_2(u\cdot\na b_1-b\cdot\na u_1)\ran  \non\\
&=\f12\f d {dt}\lan \p_2^2 b_1,(\p_2 b_1)^2\ran-\f12\lan \p_2^2(\p_1 u_1-u\cdot\na b_1+b\cdot\na u_1), (\p_2b_1)^2 \ran+\lan \p_2^2 b_1\p_2 b_1,\p_2(u\cdot\na b_1-b\cdot\na u_1)\ran,
\end{align}
where we use the magnetic field equation and
\beno
\f12\lan  u\cdot\na \p_2^2b_1, (\p_2b_1)^2\ran+\lan\p_2^2b_1\p_2b_1, u\cdot\na\p_2b_1\ran=0,
\eeno
the different estimate with \cite{RXZ} in \eqref{bh2-trouble}  is:
\begin{align*}
\lan \p_2 u_2 \, \p_2^2 b_1, (\p_2 b_1)^2 \ran
&=-\lan \p_1 u_1 \, \p_2^2 b_1, (\p_2 b_1)^2\ran\non\\
&\leq\|\p_1 u_1\|_{L^\infty}\|\p_2^2 b_1\|_{L^2}\|\p_2 b_1\|_{L^4}^2.
\end{align*}
The trouble term of $K_{22}$ is
\begin{align*}
&\lan \p^2_2 b_1, \p_2^2(b\cdot\na u_1)\ran\non\\
&=\lan \p^2_2 b_1,  \p_2^2 (b_1\p_1 u_1+b_2 \p_2 u_1)\ran \non\\
&=\lan \p_2^2 b_1, \p_2^2 b_1 \p_1 u_1+\p_2 b_1 \p_1\p_2 u_1+b_1 \p_1\p_2^2 u_1-\p_1\p_2 b_1 \p_2 u_1-\p_1 b_1 \p_2^2 u_1+b_2 \p_2^3 u_1\ran,
\end{align*}
where we only need to check
\begin{align*}
&\lan \p_2^2 b_1, b_1\p_1 \p_2^2 u_1\ran+\lan \p_2^2 b_1, b_2 \p_2^3 u_1\ran\non\\
&=-\lan \p_1\p_2^2 b_1, b_1 \p_2^2 u_1\ran-\lan \p_2^2 b_1, \p_1 b _1 \p_2^2 u_1\ran+\lan \p_2^2 b_1, b_2 \p_2^3 u_1\ran\non\\
&=\lan \p_1 \p_2 b_1, \p_2 b_1\p_2^2 u_1 \ran+\lan\p_1 \p_2 b_1, b_1 \p_2^3 u_1 \ran-\lan \p_2^2 b_1, \p_1 b _1 \p_2^2 u_1\ran+\lan \p_2^2 b_1, b_2 \p_2^3 u_1\ran\non\\
&\lesssim \|b\|_{H^2} \Big(\|\p_1 b\|_{H^1}^2  +\|\na ^2 u\|_{H^1}^2 \Big)+\|b_2\|_{L^\infty}^2 \|\p_2^2 b_1\|_{L^2}^2 +\e \|\p_2^3 u_1\|_{L^2}^2.
\end{align*}

{\bf Step 7.} Dissipation estimate of $(\na u_t,\na b_t)$.

Taking the $L^2$ product of  equations $(\ref{eq:MHDT})_1$ and  $(\ref{eq:MHDT})_2$  with $\cA u_t$ and $\cA b_t$, respectively, and using the integration by parts,   we deduce that
\begin{align}\label{nbt}
&\f12\f d {dt}\|\cA u\|^2_{L^2}+\big(\|\na u_t\|^2_{L^2}+\|\na b_t\|_{L^2}^2\big)-\lan \na u_t, \na \p_1 b\ran-\lan \na b_t, \na \p_1 u\ran \non\\
&=\lan \cA u_t,(- u\cdot\na u+ b\cdot\na b)\ran
 +\lan \cA b_t,(-u\cdot\na b+b\cdot\na u) \ran\non\\
&=\lan \na u_t, \na \cP(u\cdot\na u-b\cdot\na b)\ran+\lan \na b_t, \na \cP(u\cdot\na b-b\cdot\na u)\ran
\end{align}
where we use
\beno
\lan \cA u_t, \p_1 b\ran=\lan \na u_t, \na \p_1 b\ran, \quad \lan \cA b_t, \p_1 u\ran =\lan \na b_t, \na \p_1 u\ran.
\eeno
The trouble term of \eqref{nbt}  is
\begin{align*}
&\lan \na u_t, \na \cP(b_2 \p_2 b) \ran+\lan \na  b_t, \na \cP(u_2 \p_2 b) \ran\\
&\leq \|\na u_t\|_{L^2}\Big(\|\na b_2\|_{L^4}\|\p_2 b\|_{L^4}+\|b_2\|_{L^2_{x_1}L^\infty_{x_2}}\|\na \p_2 b\|_{L^2_{x_2}L^\infty_{x_1}}\Big)\non\\
&\quad +\|\na b_t\|_{L^2}\Big(\|\na u_2\|_{L^4}\|\p_2 b\|_{L^4}+\|u_2\|_{L^2_{x_1}L^\infty_{x_2}}\|\na \p_2 b\|_{L^2_{x_2}L^\infty_{x_1}}\Big) \\
&\lesssim \|(u, b)\|_{H^2}\Big(\|\na u_t\|_{L^2}^2+\|\na b_t\|_{L^2}^2+\|\na u\|_{H^1}^2+\|\p_1 b\|_{H^2}^2\Big).
\end{align*}

{\bf Step 8.} Dissipation estimate of $ \p_{t}^2 u$.

Applying $\p_t$ to equation $(\ref{eq:MHDT})_1$ and taking the $L^2$ inner product of the resulting equation with $\p_{t}^2 u$, we obtain
\begin{align}\label{H3u}
&\f12\f d {dt}\|\na u_t\|_{L^2}^2+\|\p_t^2 u\|_{L^2}^2-\lan \p_t^2 u, \p_1b_t\ran\non\\
&=-\lan \p_t^2 u,\p_t(u\cdot\na u)\ran+\lan\p_t^2 u, \p_t(b\cdot\na b)\ran\non\\
&\lesssim \|\p_t ^2 u \|_{L^2}\Big(\| u_t \|_{L^4}\|\na u\|_{L^4}+\|u\|_{L^\infty} \|\na u_t \|_{L^2}+\|b_t\|_{L^4}\|\na b\|_{L^4}+\|b\|_{L^\infty}\|\na  b_t\|_{L^2}\Big)\non\\
&\lesssim \|(u, b)\|_{H^2}\Big(\|\p_t^2 u\|_{L^2}^2+\|u_t\|_{H^1}^2+\| b_t\|_{H^1}^2\Big).
\end{align}

{\bf Step 9.} $\dot{H}^3$ estimate of $b$ and the dissipation estimate of $\p_1\na^2 b$.

Similarly to the $\dot{H}^2$ estimate of $b$, we apply $\p_1\Del$ to equation  $(\ref{eq:MHDT})_2$ and take the $L^2$ inner product of the resulting equation with $-\p_1\cA b$ to obtain
\begin{align*}
&\f12\f d {dt}\|\p_1\cA b\|^2_{L^2}+\|\p_1^2\na b\|_{L^2}^2-\lan \p_1^2\na b, \p_1\na u_t\ran\non\\
&=\lan \p_1\Del b,\p_1^2\cP(u\cdot\na u-b\cdot\na b)\ran
-\lan\p_1\Del b, \p_1\Del\big(u\cdot\na b-b\cdot\na u\big)\ran.
\end{align*}
The trouble case is the vertical direction. We apply $\p_2\cA$ to equation  $(\ref{eq:MHDT})_2$ and take the $L^2$ inner product of the resulting equation with $\p_2\cA b$ to obtain
\begin{align*}
&\f12\f d {dt}\|\p_2\cA b\|^2_{L^2}+\|\p_1^2\p_2 b\|_{L^2}^2+\|\cP\p_1  \p_2^2 b\|_{L^2}^2+\lan \p_2 \cA b, \p_2 \p_1 u_t\ran\non\\
&=-\lan \p_2\cA b,\p_2\p_1\cP(u\cdot\na u-b\cdot\na b)\ran
-\lan\p_2\cA b, \p_2\cA\big(u\cdot\na b-b\cdot\na u\big)\ran\non\\
&:=L_1+L_2,
\end{align*}
where we use
\begin{align*}
-\lan \p_2 \cA b, \p_2 \cA \p_1 u\ran&= \lan \p_2 \cA b, \p_2 \cP \p_1(u_t-\p_1 b+\na p+u\cdot\na u-b\cdot\na b)\ran\non\\
&=\lan \p_2 \cA b, \p_2 \p_1 u_t-\p_2 \p_1^2 b+\p_2 \p_1 \cP (u\cdot\na u-b\cdot\na b)\ran,
\end{align*}
and
\begin{align*}
-\lan \p_2 \cA b, \p_2 \p_1 ^2 b\ran&=\lan \p_2\cP \Delta b, \p_2 \p_1^2 b\ran\non\\
&=\lan \p_2\cP (\p_1^2+\p_2^2) b, \p_2 \p_1^2 b\ran\non\\
&=\|\p_1^2 \p_2 b\|_{L^2}^2+\|\cP\p_1  \p_2^2 b\|_{L^2}^2
\end{align*}
by the boundary condition $\p_2 b_1=0$ on $\p\Omega$.

We mention that
\begin{align*}
\lan \p_2 \cA b, \p_2 \p_1 u_t\ran&=-\f{d}{dt} \lan \p_2 \cP\Delta b, \p_2\p_1 u\ran+\lan \p_2 \cP \Delta b_t, \p_2 \p_1 u\ran\non\\
&=-\f{d}{dt} \lan \p_2 \cP\Delta b, \p_2\p_1 u\ran+\lan \p_2 \cP \Delta \p_1 u, \p_2 \p_1 u\ran-\lan \p_2  \Delta(u\cdot\na b-b\cdot\na u), \p_2 \p_1 u\ran
\end{align*}
with  the trouble term
\beno
\lan u\cdot\na \p_2 \Delta b, \p_2 \p_1 u\ran=\lan u\cdot\na\p_2 \p_1 u, \p_2 \Delta b\ran.
\eeno

And
\begin{align}\label{p1h2b}
\|\p_1 \na^2 b\|_{L^2}\lesssim \|\p_1 \cA b\|_{L^2}+\|\p_1 \na b\|_{L^2}.
\end{align}

The typical term of $L_1$  is
\begin{align*}
&\lan\p_2 \cP \p_2^2 b, \p_2\p_1 \cP(u\cdot\na u-b\cdot\na b)\ran\non\\
=&\lan \p_1\cP\p_2^2 b,\p_2^2\cP(u\cdot\na u-b\cdot\na b)\ran-\int_{\R}  \Big(\p_1\cP\p_2^2 b\Big)_1 \, \Big(\p_2\cP(u\cdot\na u-b\cdot\na b)\Big)_1\Big|_{x_2=0} dx_1,
\end{align*}
which can be estimate very similar as $K_{1}$, so here we omit the details.

Now we consider
\begin{align*}
L_2=-\lan\p_2\cA b, \p_2\cA\big(u\cdot\na b\big)\ran+\lan\p_2\cA b, \p_2\cA\big(b\cdot\na u\big)\ran:= II_{21}+II_{22}.
\end{align*}
For $L_{21}$, we use the commutator  and the integration by parts to have
\begin{align*}
L_{21}&=-\lan\p_2\cA b, [\p_2\cA, u\cdot\na]  b\ran\non\\
&=-\lan\p_2\cA b, \p_2\cA u \cdot\na b\ran\non\\
&=-\lan \p_2 \cP \Delta b, \p_2\Big( \cP \Delta u\Big)_1 \p_1 b\ran-\lan \p_2 \cP \Delta b, \p_2\Big( \cP \Delta u\Big)_2 \p_2 b\ran\non\\
&\lesssim \|b\|_{H^3} \Big(\|\p_1 b\|_{H^2}^2+\|\na u\|_{H^3}^2\Big),
\end{align*}
where
\begin{align*}
&-\lan \p_2 \cP \Delta b, \p_2\Big( \cP \Delta u\Big)_2 \p_2 b\ran\non\\
&=\lan \p_2 \cP \Delta b, \p_1\Big( \cP \Delta u\Big)_1 \p_2 b\ran\non\\
&=-\lan \p_1\p_2 \cP \Delta b, \Big( \cP \Delta u\Big)_1 \p_2 b\ran-\lan \p_2 \cP \Delta b, \Big( \cP \Delta u\Big)_1 \p_1\p_2 b\ran\non\\
&=\lan \p_1 \cP \Delta b,\p_2 \Big( \cP \Delta u\Big)_1 \p_2 b\ran+\lan \p_1 \cP \Delta b, \Big( \cP \Delta u\Big)_1 \p_2^2 b\ran-\lan \p_2 \cP \Delta b, \Big( \cP \Delta u\Big)_1 \p_1\p_2 b\ran\non\\
&\lesssim \|\p_1\cA b\|_{L^2}\|\p_2 \Delta u\|_{L^2} \|\p_2 b\|_{L^\infty}+\|\p_1\cA b\|_{L^2}\| \Delta u\|_{L^4} \|\p_2^2 b\|_{L^4}+\|\p_2\Delta b\|_{L^2}\| \Delta u\|_{L^4} \|\p_1\p_2 b\|_{L^4}
\end{align*}
by the property of operator $\cP$  and $\p_2 b_1=0$ on $\p\Omega$.

For $II_{22}$,
\begin{align*}
II_{22}&=\lan \p_2 \cA b, \p_2 \cA b\cdot\na u+b\cdot\na \p_2 \cA u\ran\non\\
&=\lan  \p_2 \cA b, \p_2 \Big(\cP\Delta b\Big)_1\p_1 u\ran+ \lan  \p_2 \cA b,\p_2 \Big(\cP\Delta b\Big)_2\p_2 u\ran+\lan  \p_2 \cA b, b_1\p_1 \p_2 \cA u\ran+\lan  \p_2 \cA b, b_2 \p_2^2 \cA u\ran\non\\
&=\lan  \p_2 \cA b, \p_2 \Big(\cP\Delta b\Big)_1\p_1 u\ran- \lan  \p_2 \cA b,\p_1\Big(\cP\Delta b\Big)_1\p_2 u\ran+\lan \p_1  \cA b, \p_2 b_1 \p_2 \cA u\ran+\lan \p_1  \cA b,  b_1 \p_2^2 \cA u\ran\non\\
&\quad-\lan  \p_2 \cA b, \p_1 b_1 \p_2 \cA u\ran +\lan  \p_2 \cA b, b_2 \p_2^2 \cA u\ran\non\\
&\lesssim \|\p_2 \cA b\|_{L^2}^2 \|\p_1 u\|_{L^\infty}+\|\p_2 \cA b\|_{L^2}\|\p_1 \cA b\|_{L^2} \|\p_2 u\|_{L^\infty} +\|\p_1 \cA b\|_{L^2} \|\p_2 b_1\|_{L^\infty}\|\p_2 \cA u\|_{L^2}\non\\
&\quad+\|\p_1 \cA b\|_{L^2} \|b_1\|_{L^\infty}\|\p_2^2 \cA u\|_{L^2}+
\|\p_2 \cA b\|_{L^2}\|\p_1 b\|_{L^\infty} \|\p_2 \cA u\|_{L^2}+\|\p_2 \cA b\|_{L^2}\|b_2\|_{L^\infty} \|\p_2^2 \cA u\|_{L^2}\non\\
&\lesssim \|b\|_{H^3} \Big(\|\p_1 b\|_{H^2}^2  +\|\na  u\|_{H^3}^2 \Big)+\Big(\|b_2\|_{L^\infty}^2+\|\p_1 u\|_{L^\infty}\Big) \|\p_2\cA b\|_{L^2}^2 +\e \|\p_2^2\cA u_1\|_{L^2}^2.
\end{align*}

{\bf Step 10.} $\dot{H}^3$ estimate of $u$ and the dissipation estimate of $\na^4 u$.

To get the dissipation estimate of $\na^2u$, we rewrite equations  \eqref{eq:MHDT}  as
\ben\label{smhd1}
\left\{
\begin{array}{l}
-\Delta u+\na p=\p_1 b-\p_t u-u\cdot\na u+b\cdot \na b, \ \ \ \ x\in\Omega,\\
\div u=0,  \ \ \ \ x\in\Omega,\\
u=0, \ \ \ \ x\in\p\Omega.
\end{array}\right.
\een
By the estimates for the Stokes system, we have
\begin{align*}
\|\na^2u\|_{L^2}+\|\na p\|_{L^2}
\lesssim& \|b\|_{H^1}+\| u_t\|_{L^2}+\|u\|_{L^\infty}\|\na u\|_{L^2}+\|b\|_{L^\infty}\|\na b\|_{L^2}\non\\
\lesssim& \|b\|_{H^1}+\| u_t\|_{L^2}+c_0\| u\|_{H^1}+c_0\| b\|_{H^1}\non\\
\lesssim & \|(u, b)\|_{H^1}+\| u_t\|_{L^2},
\end{align*}
\begin{align}\label{stokes1}
&\|\na^2 u\|_{H^1}+\|\na p\|_{H^1}\non\\
&\lesssim \|\p_1b\|_{H^1}+\| u_t\|_{H^1}+\|\na (u\cdot\na u)\|_{L^2}+\|u\cdot\na u\|_{L^2}+\|\na(b\cdot\na b)\|_{L^2}+\|b\cdot\na b\|_{L^2}\non\\
&\lesssim \|\p_1b\|_{H^1}+\| u_t\|_{H^1}+\||\na u|^2 \|_{L^2}+\|u\na^2 u\|_{L^2}+\|u\cdot\na u\|_{L^2}\non\\
&\quad+ \|\na b_1 \p_1 b+b_1 \p_1 \na b+\na b_2 \p_2 b+b_2 \na \p_2 b\|_{L^2}+\|b_1 \p_1 b +b_2 \p_2 b \|_{L^2}\non\\
&\lesssim \|\p_1b\|_{H^1}+\| u_t\|_{H^1}+\|u\|_{H^2} \|\na u\|_{H^1}+ \|b\|_{H^2}\|\p_1 b\|_{H^1}+\|b_2\|_{L^2_{x_1}L^\infty_{x_2}}\|\p_2^2 b_1\|_{L^2_{x_2}L^\infty_{x_1}}\non\\
&\lesssim  \|\p_1b\|_{H^1}+\| u_t\|_{H^1}+c_0\|\na u\|_{H^1} +c_0 \|\p_1 b\|_{H^2}\non\\
&\lesssim  \|\p_1b\|_{H^2}+\| u_t\|_{H^1}+\|\na u\|_{H^1},
\end{align}
and
\begin{align}\label{h4}
&\|\na^2 u\|_{H^2}+\|\na p\|_{H^2}\non\\
&\lesssim \|\p_1b\|_{H^2}+\| u_t\|_{H^2}+\|\na^2 (u\cdot\na u)\|_{L^2}+\|u\cdot\na u\|_{H^1}+\|\na^2(b\cdot\na b)\|_{L^2}+\|b\cdot\na b\|_{H^1}\non\\
&\lesssim  \|\p_1b\|_{H^2}+\| u_t\|_{H^2}+\|\na^2  u \cdot \na u+u\cdot\na^3 u\|_{L^2}+c_0\|\na u\|_{H^1} +c_0 \|\p_1 b\|_{H^2}\non\\
&\quad +\|\na^2 b_1\p_1b+\na b_1\p_1 \na b+b_1\p_1 \na^2 b+\na^2 b_2 \p_2 b+\na b_2\p_2 \na b+b_2 \p_2 \na^2 b\|_{L^2} \non\\
&\lesssim  \|\p_1b\|_{H^2}+\| u_t\|_{H^2}+ \|u\|_{H^2}\|\na ^2 u\|_{H^1}+\|b\|_{H^2} \|\p_1 b\|_{H^2}+c_0\|\na u\|_{H^1}+\|b_2\|_{L^\infty}\|\p_2^3 b_1\|_{L^2}\non\\
&\lesssim  \|\p_1b\|_{H^2}+\| u_t\|_{H^2}+\|\na u\|_{H^2} +\|b_2\|_{L^\infty}\|\p_2^3 b_1\|_{L^2}.
\end{align}
What is more, applying $\p_t$ to \eqref{smhd1}, we have
$$
\left\{
\begin{array}{l}
-\Delta u_t+\na p_t=\p_1 b_t-\p_t^2 u-\p_t(u\cdot\na u)+\p_t(b\cdot \na b), \ \ \ \ x\in\Omega,\\
\div u_t=0,  \ \ \ \ x\in\Omega,\\
u_t=0, \ \ \ \ x\in\p\Omega,
\end{array}\right.
$$
and hence
\begin{align}\label{smhd2}
\|\na^2u_t\|_{L^2}+\|\na p_t\|_{L^2}
\lesssim& \|b_t\|_{H^1}+\| \p_t^2 u\|_{L^2}+\|u\|_{H^2}\|u_t\|_{H^1}+\|b\|_{H^2}\|b_t\|_{H^1}\non\\
\lesssim& \|b_t\|_{H^1}+\| \p_t^2 u\|_{L^2}+c_0\| u_t\|_{H^1}+c_0\| b_t\|_{H^1}\non\\
\lesssim &  \|b_t\|_{H^1}+\| \p_t^2 u\|_{L^2}+\| u_t\|_{H^1}.
\end{align}
Taking \eqref{stokes1}, \eqref{smhd2} into \eqref{h4}, we have
\begin{align*}
&\|\na^2 u\|_{H^2}+\|\na p\|_{H^2}\non\\
&\lesssim  \|\p_1b\|_{H^2}+\| \p_t^2 u\|_{L^2}+\| u_t\|_{H^1}+\| b_t\|_{H^1}+\|\na u\|_{H^1} +\|b_2\|_{L^\infty}\|\p_2^3 b_1\|_{L^2}.
\end{align*}

{\bf Step 11.} Closing of the {\it a priori} estimates.

Combining Step 1-9 together, we have
\begin{align*}
&\Big(\|u(t)\|^2_{H^1}+\|b(t)\|^2_{H^2}+\|u_t(t)\|_{H^1}^2+\|\na \cA b\|_{L^2}^2\Big) \non\\
&\quad +\int_0^t\Big(\|\na u(s)\|^2_{H^1}+\|\p_1b(s)\|^2_{H^1}+\|u_t(s)\|_{H^1}^2+\|b_t(s)\|_{H^1}^2+\|\p_t^2 u(s)\|_{L^2}^2+\|\p_1\cA  b(s)\|_{L^2}^2 \Big)ds\non\\
&\lesssim \|u_0\|^2_{H^3}+\|b_0\|^2_{H^3}+(c_0+\e)\int_0^t\cF^2(s)ds+ \cE^2(t) \, \int_0^t \|b_2\|_{L^\infty}^2+\|b_2\|_{L^\infty}^4+\|\p_1 u_1\|_{L^\infty} dt
\end{align*}
for any $t\in[0, T]$, where we use
$$
\|u_t\|_{H^1}\lesssim \|u_0\|_{H^3}+\|b_0\|_{H^3},
$$
which follows from the following equation
\beno
u_t(0)=\Delta u_0+\p_1 b_0-u_0\cdot\na u_0+b\cdot\na b_0-\na p_0,
\eeno
where the pressure $p_0$ is determined by
\beno
-\Delta p_0=\na \cdot(u_0\cdot\na u_0-b_0\cdot\na b_0), \quad \text{in} \ \Omega, \quad \na p_0\cdot n=(\Delta u_0)\cdot n, \quad \text{on} \ \p\Omega
\eeno

By the Stokes estimate in Step 10 and \eqref{p1h2b}, we have
\begin{align*}
&\Big(\|u(t)\|^2_{H^3}+\|b(t)\|^2_{H^3}+\|u_t(t)\|_{H^1}^2\Big) \non\\
&\quad +\int_0^t\Big(\|\na u(s)\|^2_{H^2}+\|\p_1b(s)\|^2_{H^2}+\|u_t(s)\|_{H^1}^2+\|b_t(s)\|_{L^2}^2+\|\p_t^2 u(s)\|_{L^2}^2\Big)ds\non\\
&\lesssim  \|u_0\|^2_{H^3}+\|b_0\|^2_{H^3}+ \cE^2(t) \, \int_0^t \|b_2\|_{L^\infty}^2+\|b_2\|_{L^\infty}^4+\|\p_1 u_1\|_{L^\infty} dt
\end{align*}
for suitable $c_0$.
That is,
\begin{align*}
\cE^2(t)+\int_0^t\cF^2(s)ds
\lesssim \|u_0\|^2_{H^3}+\|b_0\|^2_{H^3}+ \cE^2(t) \, \int_0^t \|b_2\|_{L^\infty}^2+\|b_2\|_{L^\infty}^4+\|\p_1 u_1\|_{L^\infty} dt.
\end{align*}
 This completes the proof of Proposition \ref{high order}.

\end{proof}

\vspace{0.2cm}

\no{\bf Proof of  Theorem \ref{thm:main}.} We use Proposition \ref{high order}  and Proposition \ref{nonlinear} to obtain
\begin{align*}
\cE^2(t)+\int_0^t\cF^2(s)ds
\lesssim \cE^2(0)+\e \cE^2(t);
\end{align*}
taking $\e$ small enough, we conclude Theorem \ref{thm:main} by a continuous argument.  \ef

\section{acknowledgement}
The authors would like to express their hearty gratitude to Professor Zhifei Zhang for suggesting this problem and for insightful discussions.

\text{\quad}\\
\author{Jiakun Jin, Xiaoxia Ren and Lei Wang} \\ \address{Department of Mathematics and Physics\\ North China Electric Power University\\ 102206, Beijing, People's Republic of China}\\ \email{e-mail: jinjiakun18@163.com(Jiakun Jin)} \\ \email{e-mail: xiaoxiaren@ncepu.edu.cn(Xiaoxia Ren)} \\
 \email{e-mail: 50901924@ncepu.edu.cn(Lei Wang)} \\

\text{\quad}\\
\author{Yoshiyuki Kagei} \\ \address{Department of Mathematics\\
Tokyo Institute of Technology\\
2-12-1, Ookayama, Meguro-ku,
Tokyo 152-8551, JAPAN}\\ \email{e-mail: kagei@math.titech.ac.jp } \\

\text{\quad}\\
\author{Cuili Zhai} \\ \address{School of Mathematics and Physics\\
University of Science and Technology Beijing\\
100083, Beijing, People's Republic of China}\\ \email{e-mail: zhaicuili035@126.com } \\

\end{document}